\documentclass[12pt,a4paper]{article} 
\usepackage{amsmath,amssymb,amsthm,amsfonts,amscd,euscript,verbatim, t1enc, newlfont, xypic, microtype, leftidx, enumerate}
\usepackage[colorlinks=true,pagebackref=true]{hyperref}
\usepackage{tikz}
\usepackage{tikz-cd}
\usetikzlibrary{positioning}
\theoremstyle{definition}
\newtheorem{theo}{Theorem}[subsection]

\newtheorem{pr}[theo]{Proposition}

 \newtheorem{lem}[theo]{Lemma}
 \newtheorem{coro}[theo]{Corollary}
  \newtheorem{conj}[theo]{Conjecture}
\theoremstyle{remark}
\newtheorem{rema}[theo]{Remark}

\theoremstyle{definition}
\newtheorem{defi}[theo]{Definition}

 \newcommand\lan{\langle}
\newcommand\ra{\rangle}

\newcommand\sz{S_0}
\newcommand\ob{^{-1}}

\newcommand\dmx{DM(X)}
\newcommand\dmcx{DM_c(X)}
\newcommand\dm{DM}
\newcommand\dmc{DM_c}

\newcommand\dmcy{DM_c(Y)}
\newcommand\chows{Chow(S)}

\DeclareMathOperator\cha{\operatorname{char}}
\newcommand\het{{H^{et}_{\ql}}}
\newcommand\hetc{H_{et,\ql}}
\newcommand\hetl{{\mathcal{H}}^{et}_{\ql}{}}

\newcommand\shy{Sh_{per}^{et}(Y,\ql)}

\newcommand\obj{\operatorname{Obj}}

\newcommand\id{\operatorname{id}}
\newcommand\cu{\underline{C}}
\newcommand\du{\underline{D}}

\newcommand\au{\underline{A}}

\newcommand\hw{{\underline{Hw}}}

\newcommand\z{{\mathbb{Z}}}
\newcommand\zll{{\mathbb{Z}_{(l)}}}

\newcommand\ql{{\mathbb{Q}_l}}

\newcommand\zop{{\mathbb{Z}[\frac{1}{p}]}}
\newcommand\zol{{\mathbb{Z}[\frac{1}{l}]}}

\newcommand\com{{\mathbb{C}}}
\newcommand\q{{\mathbb{Q}}}

\newcommand\codim{\operatorname{codim}}
\newcommand\af{\mathbb{A}}
\newcommand\afo{\mathbb{A}^1}
\newcommand\p{\mathbb{P}}

\newcommand\al{\alpha}

\newcommand\de{\delta}
\newcommand\ph{\varphi}

\newcommand\phd{\leftidx{^{\md}}\varphi{}}
\newcommand\vp{V'}
\newcommand\tn{\widetilde{N}}
\newcommand\tvn{\widetilde{VN}}
\newcommand\tnp{\widetilde{N}'}

\newcommand\ovd{\overline{d}}

\newcommand\mg{\mathcal{M}}
\newcommand\mgbd{\mathcal{M}^{\scalebox{0.6}{BM}}_{\dm}{}}
\newcommand\mgbm{\mathcal{MK}^{BM}}
\newcommand\mgd{\leftidx^{\md}\mathcal{M}^{BM}} 
\newcommand\mgdx{\leftidx^{\md}\mathcal{M}^{BM}_X}

\newcommand\mgbms{\mathcal{MK}^{BM}_S}

\newcommand\mgbmx{\mathcal{MK}^{BM}_X}

\newcommand\mper{\mathcal{M}_{per}}

\newcommand\dhsl{D^b_cSh^{et}(X,\ql)}

\newcommand\perpp{{}^{\perp}}
\newcommand\ns{\{0\}}

\DeclareMathOperator\prli{\varprojlim}
\DeclareMathOperator\inli{\varinjlim}

\newcommand\chow{Chow}

\newcommand\chowe{Chow^{eff}}

\newcommand\ab{Ab}

\newcommand\sht{SH}
\newcommand\sinf{\Sigma^{\infty}}
\newcommand\spe{{\operatorname{Spec}\,}}
\newcommand\re{{}_{red}}

\DeclareMathOperator\kar{\operatorname{Kar}}

\DeclareMathOperator\imm{\operatorname{Im}}
\DeclareMathOperator\co{\operatorname{Cone}}

\newcommand\modd{\operatorname{Mod}}

\newcommand\dmcs{DM_c(S)}

\newcommand\hwchow{{\underline{Hw}}_{\chow}}

\newcommand\wchow{{w_{Chow}}}
\newcommand\wchowc{{w^c_{Chow}}}
\newcommand\kgl{\operatorname{KGl}}
\newcommand\kglp{\operatorname{KGl}'}

\newcommand\daoq{D_{\af^1,\q}}
\newcommand\daoqx{D_{\af^1,\q}(X)}

\newcommand\sss{{\mathcal{S}}}
\newcommand\xxx{{\mathcal{X}}}
\newcommand\lam{\Lambda}

\newcommand\kk{\mathcal{K}}
\newcommand\dk{\mathcal{DK}}

\newcommand\md{\mathcal{D}}
\newcommand\dkz{\mathcal{DK}_\z}
\newcommand\dkzc{\mathcal{DK}_\z^c}
\newcommand\dkzx{\mathcal{DK}_\z(X)}
\newcommand\dkzy{\mathcal{DK}_\z(Y)}
\newcommand\dkzs{\mathcal{DK}_\z(S)}
\newcommand\dkzcx{\mathcal{DK}_\z^c(X)}

\newcommand\dkx{\mathcal{DK}(X)}
\newcommand\dkcx{\mathcal{DK}^c(X)}
\newcommand\dky{\mathcal{DK}(Y)}
\newcommand\dkc{\mathcal{DK}^c}
\newcommand\dkcs{\mathcal{DK}^c(S)}
\newcommand\dks{\mathcal{DK}(S)}

\newcommand\oo{{\pmb{1}}}
\newcommand\ood{\leftidx^{\md}\pmb{1}}

\newcommand\comp{M_K}
\numberwithin{equation}{section}
\begin{document}

 \title{On 
relative $K$-motives, 
 weights for them,\\
 and negative $K$-groups}
 \author{Mikhail V. Bondarko\and Alexander Yu. Luzgarev\thanks{ 
The results of sections 1, 2, and 4 were  obtained under support of the Russian Science Foundation grant no. 16-11-10200.
Section 3 was written under the support of the RFBR grant no.  15-01-03034-a and of Dmitry Zimin's Foundation ``Dynasty''.
 }}
 \maketitle
\begin{abstract}
In this paper we study certain triangulated categories of $K$-motives $\dk(-)$ over a wide class of ``nice enough'' base schemes, and define certain ``weights'' for them (so, we introduce Chow weight structures for $K$-motives closely related to the  ones introduced by D. H\'ebert, the first author, and M. Ivanov for  ``relative Voevodsky motives'' of various types). We relate the weights of particular $K$-motives (of the form $f_*(\oo_Y)$, where $\oo_Y$ is the tensor unit of $\dk(Y)$) to (negative) homotopy invariant $K$-groups (tensored by $\z[\sss\ob]$ for $\sss$ being the set of ``non-invertible primes'') $\kk_*(-)$. Our results yield a (new) result on the vanishing of $\kk_i(Y)$ and of certain relative $\kk$-groups for $i$ being ``too negative''; this statement is closely related to a question of Ch. Weibel. We also prove that $\kk_i(Y)$ for $i<0$ is ``supported in codimension  $-i$''.  Moreover, we establish several criteria for bounding (below) the weights of 
the $K$-motives $f_*(\oo_Y)$; this automatically implies the vanishing of the corresponding $E_2$-terms of {\it Chow-weight spectral} sequences (and of the factors of the corresponding Chow-weight filtrations) for any (co)homology of these motives. Our methods of bounding weights use a resolution of singularities  result of O. Gabber. They can be applied to various  ``motivic'' triangulated categories; this yields some new statements on (constructible) complexes of \'etale sheaves (as well as similar bounds on the terms of Chow-weight spectral sequences for Voevodsky motives). 

We also relate the weights of $K$-motives with rational coefficients to that of Beilinson motives; the Chow-weight spectral sequences converging to their  $\ql$-\'etale (co)homology yield Deligne-type weights for the latter. Somewhat surprisingly, we are able to prove in certain (``extreme'') cases that the corresponding weight bounds coming from \'etale (co)homology are  precise; we illustrate these statements by some simple examples.

\end{abstract}

\tableofcontents

\section*{Introduction}

In \cite{weibel}
 Ch. Weibel has asked whether $K_{-n}(X)=\ns$ for any noetherian scheme $X$ of dimension less than $n$ (where $K_*$ denotes the  Bass--Thomason--Trobaugh $K$-theory). Recall here that negative $K$-groups of regular schemes vanish; so the question is related to the singularities of $X$.
The positive answer to Weibel's question for schemes essentially of finite type over a field of characteristic zero was given in \cite{corwei}. The next step in this direction was made in \cite{kellyth} where (essentially) the so-called homotopy invariant $K$-theory was considered;\footnote{The reason for this was the usage of $cdh$-descent provided by \cite{ciskth} for homotopy invariant $K$-theory; note that it fails  for the ``usual'' $K$-theory (of singular schemes). Also, the usage of Gabber's resolution of singularities result has forced Kelly and us to invert  all the positive residue field characteristics of $X$ in the coefficient ring. Yet note that our results certainly yield certain properties of  $K$-theory  when restricted to schemes of fixed positive characteristic; see Remark \ref{rweib}(\ref{icompkth}) below and Theorem 3.7 of \cite{kellykth} for more detail.} this result was improved in \cite{kerzs}. 

In the current paper we propose a ``categorification'' of this question, and relate negative (homotopy invariant) $K$-groups to weights of $K$-motives (similar to the weights constructed in \cite{hebpo}, \cite{brelmot}, and \cite{bonivan} for the corresponding versions $\dm(-)$ of the Voevodsky motivic categories following the pattern introduced in \cite{bws}). For 
this purpose we consider the homotopy categories of modules over the symmetric motivic ring spectra $\kglp_-$ (instead of  modules over the motivic cohomology spectra); we localize them by inverting all the primes that are ``not invertible over our base schemes'' (for this set of primes $\sss$ the localization $\z[\sss\ob]$ will be denoted by $\lam$) to obtain 
our ``main'' motivic categories $\dk(-)$. The version of $K$-theory ``corresponding to'' categories of this type is the homotopy invariant $K$-theory tensored by $\lam$ that is   denoted by $\kk(-)$ below. 

Voevodsky himself started developing a four functor formalism
for his motivic categories (first constructed in~\cite{voe96} and later
developed in ~\cite{voesusfrie} and~\cite{morelvoe}). Part of this
effort was described by Deligne in \cite{deligne-cross},
and later the program was 
 carried on by Ayoub in~\cite{ayoubsix}.
In our case,
the Voevodsky's four functor formalism (as developed in the treatise \cite{cd}, which relied on~\cite{ayoubsix})
yields exact functors $f_*$ and $f_!\colon\dk(Y)\to \dk(X)$ for any separated finite type morphism $f\colon Y\to X$ (as well  as $f^*$ and $f^!\colon\dkx\to \dk(Y)$). The properties of these functors and categories (along with the vanishing of negative $\kk$-theory of fields and regular schemes) allow us to  define the {\it Chow weight structure} for $\dk(S)$ 
for $S$ being a {\it ($\lam$)-nice} scheme (see Remark \ref{rrnice}(\ref{rrnice-1}) below). 
Somewhat similarly to \cite{bonivan}, the non-negative $K$-motives $\dk(S)_{\wchow\ge 0}$ over $S$ are ``generated by'' $v_*(\oo_V)$ for $v\colon V\to S$ running through finite type separated morphisms with regular domain, whereas  $\dk(S)_{\wchow\le 0}$ is ``generated by''
 $v_!(\oo_V)$, where $\oo_V$ denotes the tensor unit object of $\dk(V)$.\footnote{One may consider only regular $V/S$ or all finite type separated $S$-schemes when ``generating'' $\dk(S)_{\wchow\le 0}$. The main distinctions of our weights for $\dk(-)$ from  the weight structure construction used in \cite{bonivan} is that we don't have to ``twist'' the generators due to the fact that $K$-motives are $-\lan 1\ra=-(1)[2]$-periodic (see Remark \ref{rcons}(\ref{itate}) below); moreover, 
 we treat non-compact motives in much more detail than in ibid.}

Now, some of the properties of the Chow weight structure on $\dk(-)$ are quite similar to its $\dm$-versions (as described in the aforementioned papers); furthermore, some of them are immediate consequences of the general formalism of weight structures. Certainly, for any $\lam$-nice $S$ the weight structure $\wchow(S)$ yields weight filtrations and weight spectral sequences $T_{\wchow}(H,-)$ for any (co)homological functor $H$ that factors through $\dk(S)$ (that converge to $H_*(M)$ for $M\in \obj \dk(S)$). In the case where $S$ is (separated) of finite type over field (whose characteristic is invertible in $\lam$; for $\lam=\q$ it suffices to assume that $S$ is of finite type over an excellent noetherian scheme of  dimension at most $3$) these spectral sequences relate the cohomology of 
$M$ to that of (the $\dks$-versions of) Chow motives over $S$. The latter are the $\dkx$-retract of objects of the form $v_!(\oo_V)$ for $v\colon V\to S$ being a 
 projective morphism with regular domain (see Remark \ref{rexplwchow}(\ref{rexplwchow-3}) below). 
Moreover (for arbitrary $S$) the  ``weights'' of $M$ are ``detected''  by $E_2^{**}(T_{\wchow}(H,M))$ (that are $\dks$-functorial in $M$) ; see Corollary \ref{cdetect} below. Furthermore, in the case $\lam=\q$ the weights of $K$-motives can be expressed in terms of the weights of the ``associated'' Beilinson motives (i.e., of the objects of the $\q$-linear version of $\dm(S)$ obtained from $M$ via the functor $\mper(S)$; see Proposition \ref{pcompdm} below). Next, certain conjectures on the so-called mixed motivic sheaves predict that it suffices to consider $H$ being the corresponding {\it perverse \'etale homology} functor to compute the weights of the Beilinson motives obtained 
(see \cite{bmm} where this functor was treated in detail and the aforementioned conjectures were related to ``standard'' motivic conjectures over universal domains).
We succeed in establishing (a form of)  this ``weight-detection'' conjecture in some particular cases that are rather important for the purposes of the current paper in \S\ref{sfield}--\ref{sdelw}; we illustrate the relevance of our weight bounds by certain simple examples.

Let us now describe the main motivation for treating $K$-motives in this paper. The orthogonality axiom of weight structures yields for any $n\ge 0$ that an object $M $ belongs to $\dkx_{\wchow\ge -n}$ if and only if  $\dkx_{\wchow\le -n-1}\perp M$ (i.e., there are only zero morphisms between objects of weights less than $n$ into $M$). Now, for $M=f_*(\oo_Y)$ (where $f\colon Y\to X$ is a separated morphism of $\lam$-nice schemes)  these conditions are equivalent to the following one: for any
separated finite type
 morphism $P\to X$ the groups $\kk_i^{BM,Y}(P\times_X Y)$ (see Remark \ref{rcons}(\ref{imgbm})  vanish for $i<-n$; moreover, it suffices to consider those $P$ that are affine over $X$ and regular in this criterion. Here $\kk_i^{BM,Y}(-)$ are certain Borel--Moore $\kk$-groups; hence the latter condition is equivalent to the vanishing of $\kk_i(U\times_X Y)$ for all smooth $U/X$ and $i<-n$ combined with the surjectivity of  the natural homomorphism $\kk_{-n}(U\times_X Y)\to \kk_{-n}(V\times_X Y)$ for any such $U$ and any open (dense) $V\subset U$. Thus we obtain a close relation between negative $\kk$-groups and the ``weights'' of objects of the type $f_*(\oo_Y)$; the aforementioned motivic conjectures (along with the particular cases in which we can prove them) 
yield quite unexpected relations of these matters with the weight spectral sequences converging to \'etale (co)homology (see \S\ref{sdelw}). Note here that there certainly exist connecting ``regulator'' homomorphisms between $\kk$-groups with rational coefficients and $\ql$-\'etale cohomology; yet these maps are very far from being surjective or injective (in general). 

Next, to obtain results related to the Weibel's question we bound below the weights of $\oo_X$ and of $f_*(\oo_Y)$; we also study the Chow-weight decompositions of these objects. It is no wonder that this requires certain resolution of singularities results (note that $f_*(\oo_Y)\in \dkx_{\wchow\ge 0}$ whenever $Y$ is regular). 
To make our weight bound results as general as possible, we consider a form of Gabber's resolution of singularities results that was previously used in \cite{kellykth}.
We use an argument closely related to the ones that were used in \cite[\S XIII.3]{illgabb}, \cite[\S6.2]{cdet}, and \cite[\S4.2]{cd} for the proof of the constructibility of the corresponding ``analogues'' of $M=f_*(\oo_Y)$ for $f$ being of finite type. Now, the reasonings in the aforementioned papers are far from being short and simple, whereas we are interested in much more precise ``information'' on $M$. So, we use somewhat technical definitions (including a modification of Gabber's dimensions functions) and arguments to put objects of the type $f_*(\oo_Y)$ and $f_!(\oo_Y)$ into certain {\it envelopes} (i.e., we describe  more or less explicit  sets of objects of $\dkx$ such that the motives that interest us can be obtained from  them by means of extensions and retracts). 
It does not make sense to formulate these results in the introduction; so we will try to to describe some of their consequences instead (cf. also Remark \ref{rcompvm}(\ref{rcompvm-4})). We prove that 
for any $s\ge 0$ there exists a closed scheme $Z\subset X$ such that $\dim(X)-\dim(Z)\ge s+1$ and for $U=X\setminus Z$ we have $\oo_U\in \dk(U)_{w\ge -s}$. According to the aforementioned relation of $\dkx$-weights of $\oo_X=\id_{X,*}(\oo_X)$ to $\kk$-groups, it follows that $\kk_{-s}$ is ``supported in codimension $s$''; this statement appears to be completely new (see Remark \ref{rweib}(\ref{isupport})). We also establish some more criteria for $f_*(\oo_Y)$ to belong to  $\dkx_{\wchow\ge -n}$ (for $f$ being of finite type). In particular  (see Theorem \ref{tweib}(II.\ref{3.3.1.II.2})), it suffices to verify that the groups $\kk_{-i}^{BM,Y}(P\times_X Y)$ vanish for $i<-n$ only in the case where  $P$ is regular and affine and $\dim(P)\le \ovd-i$, where $\ovd$ is the dimension of an $X$-compactification of $Y$.\footnote{Recall that  $X$-compactifications of any finite type separated $Y/X$ exist, and their dimensions are equal.  Moreover, $\ovd$ equals $\dim(Y)$ whenever $X$ is of finite type over a field or over $\spe \z$; see Remark \ref{rdimf} below for more detail.} Furthermore, it suffices to compute these $\kk$-groups for a finite number of  ``test schemes''  that can be (more or less) explicitly described in terms of $f$; see Remark \ref{rweib}(\ref{iextest}) for more detail. Note also that the reader willing to
avoid the complicated arguments of \S\ref{sgabblem} may replace most of them (including the construction of the aforementioned ``test schemes'')  by a much easier Proposition \ref{pldh} in the case where $X$ is of finite type over a field.
On the other hand, in \S\ref{sgengab} the $K$-motivic ``envelope'' formulations of \S\ref{sgabblem} were extended to a wide range of ``motivic'' categories; in particular (applying them to \'etale motives as studied in \cite{cdet}) one obtains certain ``envelope'' statement that appear to be quite new over ``general'' schemes (yet we do not treat this matter in detail).

Let us  now describe the contents  of the paper. Some more information of this sort can be found at the beginnings of sections.

In \S\ref{sprelim} we 
  describe some basics on triangulated categories $\dk(-)$ of $K$-motives over (excellent separated Noetherian finite dimensional) schemes. Since the properties of these categories are quite similar to the ones of ``more or less usual'' relative Voevodsky motives (as studied in several papers of Ayoub, Cisinski, and D\'eglise, we only sketch their proofs. We also recall those aspects of the theory of weight structures that will be needed below.

In \S\ref{sprwchow} we introduce our Chow weight structures on $\dk(-)$ (starting from its restriction to the subcategory of compact motives) and study their properties. We also formulate several criteria for the motif $f_*(\oo_Y)$ for $f\colon Y\to X$ to be of weights $\ge -n$ in $\dkx$; they are formulated in terms of negative homotopy invariant $K$-groups of certain $Y$-schemes. A remarkable distinction of this section from the study of the Chow weight structures in \cite{hebpo}, \cite{brelmot}, and \cite{bonivan} is that we study (in detail) the weights of non-compact motives (also); this enables us not to restrict our criteria to the case where $f$ is of finite type.\footnote{Note  also that certain Chow weight structures on "big" motivic categories were the main subject of  \cite{bokum}; however, 
 the reason to consider them was to avoid the usage of resolution of singularities statements.}

In \S\ref{slength} we study these criteria along the weight bounds for $f_*(\oo_Y)$ in more detail. For this purpose we apply the resolution of singularities results of Gabber for putting the objects $f_*(\oo_Y)$ and $f_!(\oo_Y)$ (for $f:Y\to X$ being of finite type) into the envelope of 
 $v_{b,*}\oo_{V_b}[m_b]$ for certain finite type morphisms $v_b\colon V_b\to X$ with regular domains and $m_b\ge 0$; the pairs $(V_b,m_b)$ may be (more or less) explicitly described. These results can be extended to a wide range of ``relative motivic'' categories; so our methods yield (in particular) a more precise version of the constructibility of ``direct images'' of  (tensor) unit $h$-motives as proved in \cite[\S6.2]{cdet} following the pattern of \cite[\S XIII.3]{illgabb} (and it implies a new property of the corresponding constructible complexes of \'etale sheaves).   We bound below the weights of $f_*(\oo_Y)$; 
  this yields vanishing results closely related to the (``homotopy invariant'' version of the)  question  of Ch. Weibel on the vanishing of the $K$-theory of $Y$ in degrees less than   $-\dim(Y)$. One may say that our weight bounds are certain ``categorifications'' of this vanishing question.

In \S\ref{scompvoevet} we prove that the ``weights'' of $K$-motives can be expressed in terms of the ones of the associated Beilinson motives. 
Then we proceed to study the weights of the latter (as well as of  some ``similar'' motivic categories);  we mostly study   the case where the base scheme $X$ is  the spectrum of a field or a variety. Under this restriction the weight zero $X$-motives come from regular projective  $X$-schemes; so the corresponding weight spectral sequences $T_{\wchow(X)}(H,M)$ relate the (co)homology $H_*$ of $X$-motives to that of $X$-schemes of this sort. Note here that (according to the general theory of weight structures) the non-vanishing of 
 $E_2^{pq}T_{\wchow(X)}(H,M)\neq 0$ for some $q\in \z$ and $p>n$ implies that $M$ ``contains weights less than $-n$''. 
When $X=\spe k$ this yields a close connection of the  weights of $X$-motives and schemes  with Deligne's weights of their \'etale (co)homology;  for  $X$ being a variety one can use the weights from \cite{bbd} and \cite{huper} here. We have an inequality between these two versions of weight bounds; it is conjecturally an equality, and we prove this conjecture for certain ``(almost) maximally singular'' $X$-schemes. We also illustrate our notion of {\it maximally singular schemes} (this is a certain ``motivic'' characterization of the singularities of $X$) by some easy examples.  Lastly; we describe some simple substitute of the ``motivic resolution of singularities'' arguments of \S\ref{slength} 
in the case where $X$ is a variety; this suggests a close relation of our weight filtrations and spectral sequences to the {\it singularity} ones considered in \cite{pasdesc} and \cite{ciguil} in the characteristic $0$ case.


The authors are deeply grateful to prof. F. D\'eglise and prof. S. Kelly for their illuminating remarks. 

\section{Some preliminaries: notation, $K$-motives, and weight structures}\label{sprelim}

The results of these section are (more or less) easy consequences of the formalism of (relative) motivic categories and weight structures (as developed in \cite{cd} and \cite{bws}, respectively).

In \S\ref{snotata}  we introduce some notation and conventions for (mostly, triangulated) categories. We also recall some basics on ``localizing coefficients'' in triangulated categories and recall the main result of \cite{bsnull} (that describes the envelope of a set of object of $\cu$ in terms of cohomological functors from $\cu$).

In \S\ref{skmot} we introduce our version of the categories  $\dk(-)$ of (relative) $K$-motives. Our results are essentially the $\kgl$-module versions of the corresponding results of \cite{cd} and \cite{cdet}; so we just skip most of the proofs.

In \S\ref{sbws} we recall  some basics  properties on weight structures and prove a few new statements. In \S\ref{swss} we introduce the (more complicated)  notions of weight complexes and weight spectral sequences; a reader only interested in \S\ref{sprwchow}--\ref{slength} may skip this section.

\subsection{Some notation, definitions, and auxiliary statements}\label{snotata}
  

For a category $\cu$ the symbol $\cu^{op}$ will denote its opposite category.

 For a category $\cu$, $X,Y\in\obj\cu$, 
$\cu(X,Y)$ is the set of  $\cu$-morphisms from  $X$ to $Y$.
We will say that $X$ is  a {\it
retract} of $Y$ if $\id_X$ can be factored through $Y$. Note that if $\cu$ is triangulated or abelian 
then $X$ is a  retract of
$Y$ if and only if $X$ is its direct summand.

For categories $\cu,\du$ we write 
$\du\subset\cu$ if $\du$ is a full 
subcategory of $\cu$.

For any $\du\subset\cu$
the subcategory $\du$ is called {\it Karoubi-closed} in $\cu$ if it
contains all retracts of its objects in $\cu$. We will call the
smallest Karoubi-closed subcategory of $\cu$ containing $\du$  the {\it
Karoubi-closure} of $\du$ in $\cu$; sometimes we will use the same term
for the class of objects of the Karoubi-closure of a full subcategory
of $\cu$ (corresponding to some subclass of $\obj\cu$).

The {\it Karoubi envelope} $\kar(\du)$ (no lower index) of an additive
category $\du$ is the category of ``formal images'' of idempotents in $\du$.

In this paper all complexes will be cohomological, i.e., the degree of
all differentials is $+1$; respectively, we will use cohomological
notation for their terms. $K(B)$ will denote the homotopy category of complexes over an additive category $B$; for $n\in \z$ the notation $K(B)^{\le n}$ is used to denote the class of complexes isomorphic (i.e., homotopy equivalent to) complexes concentrated in degrees $\le n$.

For a $\cu$-morphism $f\colon X\to Y$  the symbol $\co(f)$ denotes the third vertex of the  triangle $X\to Y\stackrel{f}{\to} \co(f)\to X[1]$
 (so, $\co(f)$ is well-defined up to a non-canonical isomorphism). 

The symbols $\cu$ and $\du$ below will always denote some triangulated categories.
 We will use the
term ``exact functor'' for a functor of triangulated categories (i.e.,
for a  functor that preserves the structures of triangulated
categories).

For any
distinguished triangle $A\to B\to C$  in $\cu$ we will call $B$ a ($\cu$-) {\it extension} of $C$ by $A$. 
A class $D\subset \obj \cu$ will be called {\it extension-closed} if $ D$ contains $0$ as well as
all extensions of its elements by its elements.
In particular, an extension-closed $D$ is strict (i.e., contains all
objects of $\cu$ isomorphic to its elements).

The smallest extension-closed $D$ containing a given $D'\subset \obj \cu$ will be called the {\it extension-closure} of $D'$.

The smallest extension-closed Karoubi-closed subclass of $\obj \cu$ containing $D$
(resp. containing $\cup_{i\ge 0}D[i]$, resp.  containing $\cup_{i\le 0}D[i]$) 
 will be called the {\it envelope} (resp. the {\it left envelope}, resp. the {\it right envelope}) of $D$.

Below $\au$ will always  denote some abelian category.


We will call a covariant (resp. contravariant)
additive functor $H\colon\cu\to \au$ 
 {\it homological} (resp. {\it cohomological}) if
it converts distinguished triangles into long exact sequences. 

For $X,Y\in \obj \cu$ we will write $X\perp Y$ if $\cu(X,Y)=\ns$.
For $D,E\subset \obj \cu$ we  write $D\perp E$ if $X\perp Y$
 for all $X\in D,\ Y\in E$.
For $D\subset \cu$ the symbol $D^\perp$ denotes the class
$$\{Y\in \obj \cu\mid X\perp Y\ \forall X\in D\}.$$
Sometimes we will use the notation $D^\perp$ to denote the corresponding
 full subcategory of $\cu$. Dually, ${}^\perp{}D$ is the class
$\{Y\in \obj \cu\mid Y\perp X\ \forall X\in D\}$. 

We will say that some $C_i\in\obj\cu$, $i\in I$, \emph{Hom-generate} $\cu$ if for $X\in\obj\cu$ we have: $\cu(C_i[j],X)=\ns\ \forall i\in I,\ j\in\z\implies X=0$ (i.e., if $\{C_i[j]\mid j\in \z\}^\perp$ contains only zero objects).

$M\in \obj \cu$ will be called compact if the functor $\cu(M,-)$
commutes with all small coproducts 
(we will only consider compact objects in those categories that are closed with respect to arbitrary small coproducts).

We will say that a triangulated category $\cu$ (closed with respect to arbitrary small coproducts) is {\it compactly generated} if the (triangulated) subcategory of compact objects in it is essentially small and its objects Hom-generate $\cu$.

For a set 
of objects $C_i\in\obj\cu$, $i\in I$, we will use the notation $\langle C_i\rangle$ to denote the smallest strictly full triangulated subcategory containing all $C_i$; for
$D\subset \cu$ we will write $\langle D\rangle$ instead of $\langle \obj D\rangle$. 
We will call the  Karoubi-closure of $\langle C_i\rangle$ in $\cu$ the triangulated category  generated by $C_i$ (recall that it is triangulated indeed).

If $\cu$ is closed with respect to all small coproducts and $\du\subset \cu$ ($\du$ is a triangulated category that may be equal to $\cu$)
  we will say that 
the objects $C_i$ of $ \du$ generate $\du$ {\it as a localizing subcategory} of $\cu$ if  
$\du$ is the smallest full strict triangulated subcategory of $\cu$ that contains $C_i$ and is closed with respect to all small coproducts (it easily follows that $C_i$ also Hom-generate $\du$).

We 
 list the main properties of  the ``localization of coefficients'' functors for compactly generated triangulated categories.

\begin{pr}\label{plocoeff}

Let $\cu$ be  a triangulated category that is  compactly generated by its (full) triangulated subcategory $\cu'$;
let
$\sss\subset \z$ be a set of prime numbers. Denote $\z[\sss\ob]$ by $\lam$; denote
 by $\cu_{\sss-tors}$ the localizing subcategory of $\cu$ generated by cones of $c\stackrel{\times s}{\to} c$ for $c\in \obj \cu,\ s\in \sss$.

Then the following statements are valid.

\begin{enumerate}
\item\label{icg1}  $\cu_{\sss-tors}$ is (also) Hom-generated by $c'\stackrel{\times s}{\to} c'$ for $c'\in \obj \cu',\ s\in \sss$.


\item\label{icg2} The Verdier quotient category $\cu[\sss\ob]=\cu/\cu_{\sss-tors}$ exists (i.e., the morphism groups of the target are sets) and is closed with respect to  small coproducts.  Moreover, the localization functor $l\colon\cu\to \cu[\sss\ob]$ respects all coproducts, converts compact objects into compact ones, and $\cu[\sss\ob]$ is generated by $l(\obj \cu')$ as its own localizing subcategory. 

\item\label{icg3} For any $c\in \obj \cu$, $c'\in \obj \cu'$, we have $\cu[\sss\ob](l(c'),l(c))\cong \cu(c,c')\otimes_\z \lam$.

\item\label{icg4}$\cu[\sss\ob]$ is an  $\lam$-linear category.



\end{enumerate}
\end{pr}
\begin{proof}
See Proposition A.2.8 and Corollary A.2.13 of \cite{kellyth} (cf. also Proposition 5.6.2(I) of \cite{bpure} and Appendix B of \cite{levconv}).

\end{proof}

\begin{rema}\label{rlocoeff}
\begin{enumerate}
\item\label{icompfu}
Sometimes we will have to  ``increase $\sss$''. For an $\sss'\supset \sss$ we certainly have 
obvious exact (localization)  comparison functor $\cu[\sss\ob]\to \cu[\sss'{}\ob]$ that respects compact objects and coproducts.  Note that one can obtain it by setting the ``new starting category'' being equal to $\du=\cu[\sss\ob]$.


\item\label{icompenv} 
In the case $\sss'=\p\setminus \{l\}$ (for $l\in \p\setminus \sss$) the corresponding comparison functor will be denoted by $c_{\du}^l$.  
Certainly, Proposition \ref{plocoeff}(\ref{icg3}) yields that the restriction of $c_{\du}^l$ to the subcategory $\du'$ of compact objects is the ``naive $\zll$-linearization'' functor, i.e., it is the exact functor that tensors the morphism groups by $\zll$.

\end{enumerate}
\end{rema}

The following obvious modification of the main result of \cite{bsnull} appears to be quite useful for the ``control of envelopes''.

\begin{pr}\label{pbsnull}
Let $\cu$ be a   small $\lam$-linear triangulated category, $D\cap\{M\}\subset \obj \cu$. Then $M$ belongs to the envelope  of $D$ if and only if for any homological functor $F\colon\cu\to \ab$ 
we have $F(M)=\ns$ whenever the restriction of $F$ to $D$ is zero. Moreover, if $\lam\neq \q$ then it suffices to verify this condition under the assumption that the target of $F$ is $\z_{(l)}$-modules  for some $l\in \p\setminus\sss$. 
\end{pr}
\begin{proof} The first part of the assertion is immediate from Theorem 0.1 of~\cite{bsnull} (applied to the category $\cu^{op}$). To obtain the ``moreover'' part one should note (similarly to ibid.) that all the functors $F(-)\otimes_z  \z_{(l)}$ for $l\in \sss$ are 
 cohomological, and we have $F(M)=\ns$ whenever 
$F(M)\otimes_z  \z_{(l)}=\ns$   for  all $l\in \sss$. \end{proof}


Throughout the paper we will only consider  schemes that are excellent separated 
of finite Krull dimension; we will call schemes that satisfy all of these conditions nice ones (for the sake of brevity).
All the morphisms we will consider will be separated; they will also usually be of finite type.

\begin{rema}\label{rrnice}
\begin{enumerate}
\item\label{rrnice-1}. Moreover, we will usually   fix a set of primes $\sss$ (and set $\lam=\z[\sss\ob]$ as in Proposition \ref{plocoeff} above). We will say that a (nice) scheme $X$ is {\it $\lam$-nice} 
if  all the primes in $\p\setminus \sss$ are invertible on it (so, the characteristics of all the residue fields of $X$ belong to $\sss\cup\ns$; in particular, all nice schemes are $\q$-nice with $\sss=\p$).
By default,
all the schemes we consider will be $\lam$-nice also (so, we will assume that a scheme is $\lam$-nice if  it will not be said explicitly that is just nice).


\item\label{rrnice-2}. Our reason for concentrating on $\lam$-nice schemes  is that this restriction is necessary for the application of Gabber's resolution of singularities result (see Remark \ref{rconstrgabber} below) for the ``control of the compactness'' of $\lam$-linear motives (see  Theorem \ref{tcd}(II.\ref{irmotgen})).
\end{enumerate}
\end{rema}

All morphisms of schemes we consider will be separated. We will say that a morphism is smooth only if it is also of finite type. 


The following versions of $K$-theory will be one of the main subjects of this paper: for a scheme $X$ we will use the notation $KH_*(X)$ for the   (Weibel's) homotopy invariant $K$-theory of $X$ (cf. \cite{ciskth}, \cite{kellykth}); 
$\kk(X)=KH(X)\otimes_\z \lam$.

We will say that a projective system $S_i$, $i\in I$,  of schemes is {\it essentially affine} if the transition morphisms $g_{ji}\colon S_j\to S_i$ are affine 
whenever $i\ge i_0$ (for some $i_0\in I$). 

A presentation of a scheme $X$  as $\cup X_l^\alpha$, where $X_l^\alpha$, $1\le l\le n$ (it will be convenient for us to use this numbering convention throughout the paper), are pairwise disjoint locally closed  subschemes of $X$ and  each $X_l^{\alpha}$ is open in $\cup_{i\ge l} X_i^{\alpha}$,
will be called a {\it stratification} of $X$.\footnote{This somewhat weak notion of a stratification was used in some previous papers of the first author.}
The corresponding embeddings $X_l^{\alpha}\to X$ will be denoted by $j_l^\al$. 
We we say that this stratification ($\al$) is regular whenever all  $X_l^{\alpha}$ are regular.

The symbol $X\re$ will denote the reduced scheme associated to a scheme $X$. Note that $X$ and $X\re$ are ``equivalent from the motivic point of view'' (see Theorem \ref{tcd}(II.\ref{itr}) below); so one may assume that all the schemes we consider are also reduced (and consider all the Cartesian diagrams of this paper in the category of reduced schemes).

\subsection{On $K$-motives over a base}\label{skmot}


For any 
nice scheme $S$ 
 consider the symmetric motivic ring spectrum $\kglp_{S}
 \in \obj SH(S)$ defined in 
\S13.3.1 of \cite{cd}  (that relies on the main result of \cite{rso}); this spectrum is a ``minor modification'' of Voevodsky's $\kgl_S$. Following \S13.3.2 of  ibid., we take the category 
 $\dkz(S)$ of modules 
over $\kglp_S$; 
 here we consider $\kglp_S$-modules in $SH(S)$ endowed with the (``projective'')  structure of  a Quillen model category (see \S5.3.35  of ibid.), and  define $\dkz(S)$ as the homotopy category of the model category obtained. 

We will not give detailed proofs of (most of) the statements below since very similar statements were already proved in ibid. (along with other papers of Cisinski and Deglise).

\begin{theo}\label{tcd}

Let $X,Y$ be any
 (nice) schemes, $f\colon Y\to X$ be a (separated) scheme morphism, $i\in \z$.

I. Then the following statements are valid.

 \begin{enumerate}
\item\label{icatz}
The category $\dkz(X)$ is a tensor
 triangulated category; its unit object 
 will be denoted by $\z_X$. $\dkz(X)$ is closed with respect to all (small) coproducts, and the tensor product respects them.

\item\label{ifunast} We have exact functors $f^*\colon\dkz(X)\to \dkz
(Y)$ and $f_*\colon\dkz(Y)\to \dkz(X)$; $f^*$ is left adjoint to
 $f_*$. 

Any of them (when $f$ varies) yields a  $2$-functor from the category of 
(nice) schemes
with  separated morphisms to the $2$-category of triangulated categories.

\item\label{ifunshr} If $f$ is of finite type, then we also have adjoint functors
$$f_!\colon\dkzy \leftrightarrows \dkzx :\!\!f^!.$$
Similarly, these two types of functors yield 
$2$-functors from the category of
(nice) schemes
with  separated finite type morphisms to the $2$-category of triangulated categories. 

\item\label{iexch} 
For a Cartesian square
of (separated)  
morphisms 
\begin{equation}\label{ebchcd}
\begin{CD}
Y'@>{f'}>>X'\\
@VV{g'}V@VV{g}V \\
Y@>{f}>>X
\end{CD}\end{equation}
such that $g$ and $g'$ are of finite type
we have $g^!f_*\cong f'_*g'{}^!$ and $g'_!f'{}^*\cong f^*g_!$.

\item\label{iupstar}  $f^*$ is symmetric monoidal; $f^*(\z_X)=\z_Y$.

\item\label{ipur} $f_*\cong f_!$ if $f$ is proper; $f^!=f^*$ if $f$ is an open immersion.


\item\label{iglu} 
If $i\colon Z\to X$ is a closed immersion, $U=X\setminus Z$, $j\colon
 U\to X$ is the complementary open immersion, then the motivic image functors yield a {\em gluing datum} for $\dk(-)$  in the sense  of \S1.4.3 of \cite{bbd}; cf. 
Proposition 1.1.2(10) of \cite{brelmot}). In particular, for any $M\in \obj \dkzx$ the pairs of morphisms 
\begin{equation}\label{eglu1} 
j_!j^*(M) \to M
\to i_*i^*(M)\end{equation} 
and \begin{equation}\label{eglu2} 
i_*i^!M \to M \to j_*j^*M\end{equation} 
can be (uniquely and) functorially 
completed to distinguished triangles. Moreover,
$i^*j_{!} = 0$, $i^{!}j_* = 0$, and the adjunctions transformations
$i^* i_* \to 1_{\dkz(Z)} \to i^{!}i_{!}$ and $j^*j_*\to 1_{\dkz(U)} \to j^{!}j_{!}$
are isomorphisms.

\item \label{ipure}  $f^!\cong f^*$ 
 if $f$ is smooth (of finite type).

\item \label{ipura} 
If $i\colon S'\to S$ is a closed immersion of regular (nice)  schemes 
 then $\z_{S'}\cong i^!(\z_S)$.

\item\label{itre}
If $X,Y$ are regular, and $\mathcal{O}_Y$ is a free   $\mathcal{O}_X$-module of finite rank $d$, then for any $M\in \obj \dkzx$ there exists a morphism $v\in \dkzx( f_*f^*(M),M)$ such that its composition with
the unit morphism $M\to f_*f^*(M)$ 
equals $d\id_M$.

\item\label{imotgen}
The full subcategory $\dkzc(X)\subset \dkx$ of compact objects is 
 triangulated.  $\dkzc(X)$ is generated by $g_!(\z_{X'})$  for $g\colon X'\to X$ running through all smooth  (finite type) morphisms. 

\item\label{imotgenf} $f^*$ preserves the compactness of objects; this is also true for $f_!$ if $f$ is of finite type.

\item\label{icoprod} $f^*$ and $f_*$ commute with arbitrary small coproducts; the same is true for  $f_!$ and $f^!$ if $f$ is of finite type.

\item\label{icont}
Let a scheme $S$ be  the limit of an essentially affine (see \S\ref{snotata}) filtering projective system of  schemes $S_i$ for $i\in I$ (certainly, we assume that all of these schemes are nice) .  Denote the corresponding transition morphisms $S_j\to S_i$ (resp. $S\to S_i$) by $g_{ji}$ (resp. by $h_i$).
Then $\dkzc(S)$ is isomorphic to the $2$-colimit  of the categories $\dkzc(S_i);$ in this isomorphism  the corresponding connecting functors are given by $g_{ji}^*$ and by $h_i^*$, respectively.

Furthermore, for any  $i_0\in I$, $M\in \obj \dkzc(S_{i_0})$, and $N\in \obj \dkz(S_{i_0})$, the natural map
$$
\inli_{i\ge i_0} \dkz({S_{i}})(g_{ii_0}^*(M),g_{ii_0}^*(N))\to \dkzs(h_{i_0}^*(M),h_{i_0}^*(N))
$$
is an isomorphism.

\item\label{ikmor} We have 
 $\dkz(\z_X[i],\z_X)\cong KH_{i}(X)$; 
these isomorphisms (when $X$ varies) are compatible with the motivic functors of the type $f^*$. 

\end{enumerate}

II. Assume in addition that $X$ and $Y$ are $\lam$-nice.  For any  $S$ that is also $\lam$-nice we define the 
  category $\dk(S)$ as $\dkz(S)[\sss\ob]$ (see Proposition \ref{plocoeff});  the localization functor  $\dkz(S)\to \dk(S)$ will be denoted by $l_S$. 

\begin{enumerate}
\item\label{icat} $\dk(X)$ is a tensor
 triangulated category; the object $\oo_X=l_X(\z_X)$
is a unit one for this tensor product.

Furthermore, $\dkx$ is closed with respect to all (small) coproducts.

\item\label{ikmorr} We have 
 $\dk(\oo_X[i],\oo_X)\cong \kk_{i}(X)$; 
these isomorphisms (when $X$ varies) are compatible with the motivic functors of the type $f^*$. 
In particular, if $X$ is regular and $i<0$ then $\oo_X[i]\perp \oo_X$.

\item\label{ianr}
The natural analogues of assertions I.\ref{ifunast}--\ref{icont} for the categories $\dk(-)$ along with their subcategories $\dkc(-)$ of compact objects   (and for $\lam$-nice schemes)  are also valid.

\item\label{irmotgen}
If $f$ is of finite type, then $f_*$ and $f^!$ respect the compactness of objects (also; cf. assertion I.\ref{imotgenf}).

\item \label{itr}
If $f$ is a finite universal homeomorphism, then $f^*$, $f_*$, $f^!$,  and $f_!$  are equivalences of categories. 
Moreover, $ f^!(\oo_Y)\cong  f^*(\oo_Y)= \oo_X$ and $f_*(\oo_X)= f_!(\oo_X)\cong \oo_Y$.

\item\label{igenc}
If $X$ is of finite type over a field, then $\dkcx$ (as a triangulated category) is generated by $\{ p_*(\oo_P)\}$, where $p\colon P\to X$ runs through all  projective morphisms such that $P$ is regular. 

In particular,  if $X$ is the spectrum of a perfect field itself, then we consider (all) smooth projective $p$ here. 

Moreover, if $\lam=\q$ then it suffices (in the first of these statements) to assume that  $X$ is of finite type over an excellent noetherian scheme of  dimension at most $3$.




\end{enumerate}
\end{theo}
\begin{proof}
I. Corollary 13.3.3 of \cite{cd} states that  $\dkz(-)$ is a {\it motivic triangulated category}
(see Definition {2.4.45} of ibid.). Moreover, it is also  oriented (see  Remark 13.2.2  and Example 12.2.3(3) of ibid.). 
Hence Theorem {2.4.50} of ibid. (along with  Definitions {2.4.45} and  1.1.21 of ibid.) 
implies our assertions   \ref{icat}--
\ref{iglu}.  Moreover, we also obtain that assertion \ref{ipure} is fulfilled 
  ``up to 
Tate twists'' (cf. Remark \ref{rcons}(\ref{itate}) below). Since $\dkz(-)$  is ``periodic'' (see (K4) in \S13.2.1 of  \cite{cd}), the Tate twists are automorphisms of $\dkz(-)$; this finishes the proof of this assertion. 

Assertion \ref{ipura} follows easily from Theorem 13.6.3 of ibid.

Assertion \ref{itre} can be easily established 
 similarly to Theorem 14.3.3 of \cite{cd}, using Proposition 13.7.6 of ibid.


 To prove assertion \ref{imotgen} we note that $g_!$ is left adjoint to $g^*$ whenever
 $g\colon Z\to X$ is a smooth morphism. 
 Hence the definition of $\dkzx$ (see  the adjunction in \cite[\S13.3.2]{cd} and the definition of $SH(X)$) 
yields that all  objects of the form $g_!(\z_Z)$ (for a finite type $g$) 
are compact,  and they Hom-generate $\dkzx$ (cf. Remark 4.4 of \cite{cdint}). Hence the triangulated subcategory of $\dkzx$ that is generated by all $g_!(\z_Z)$ consists of compact objects. Lastly, it contains all compact objects of  $\dkzx$ by  Lemma 4.4.5 of \cite{neebook} (cf. also Lemma A.2.10 of \cite{kellykth}).

I.\ref{imotgenf}. $f^*$ respects the compactness of objects according to assertions I.\ref{iexch} and I.\ref{imotgen}.
The same is true for $f_!$ (if $f$ is of finite type) by Remark \ref{rcons}(\ref{imgbmz}) below.

By Proposition 1.3.20 of \cite{cd}, these fact imply that $f_*$ and $f^!$ (for a finite type $f$) respect coproducts. To conclude the proof of assertion \ref{icoprod} it remains to note that $f^*$ and $f_!$ (if $f$ is of finite type)  respect coproducts since they possess right adjoints.

Next,  
assertion \ref{icont} can be established 
 similarly to Proposition 4.3 of \cite{cdint} (see also Proposition 2.7 of ibid.).

\ref{ikmor}. 
The adjunction used in   \S13.3.2 of \cite{cd} yields that $\dkz(\z_X[i],\z_X)\cong \sht(X) (\sinf_X(X_+),\kglp_X)$.  
It remains to apply Theorem 2.20 of \cite{ciskth}. 

II. Assertions \ref{icat}--\ref{ianr} easily follow from the corresponding statements in part I of our theorem if we combine assertions I.\ref{imotgen}, \ref{icoprod} with Proposition \ref{plocoeff}.

Assertion \ref{irmotgen} can be proved similarly to Corollary 6.2.14 of \cite{cdet}; see Remark \ref{rconstrgabber} below for more detail.

The first part of assertion \ref{igenc} can be proved similarly to Proposition 7.2 of \cite{cdint}, whereas in the $\lam=\q$-case one should combine Corollary 4.4.3 of \cite{cd} 
  with Theorem  1.2.5 of \cite{tem}; cf. \S2.4 of \cite{bondegl}. 


 
\end{proof}

\begin{rema}\label{rcons}
 $\dk(-)$ is the main motivic category of this paper.

We will  need the following observations related to it below.
\begin{enumerate}


\item\label{imgbm} 
For a finite type (separated) $f\colon Y\to X$ we set $\mgbmx(Y)=f_!(\oo_Y)$  (this is a certain {\it Borel--Moore} motif of $Y$; cf.  \cite{bondegl} 
and \S I.IV.2.4 of \cite{lemm}).
For $g\colon X'\to X$ being a morphism of ($\lam$-nice)  schemes (that is separated but not necessarily of finite type) one easily sees that $g^*(\mgbmx(Y))\cong \mgbm_{X'}(Y\times_{X} X')$.

Next, for $U\subset Y$ being an open subscheme, $Z=Y\setminus U$, the distinguished triangle (\ref{eglu1})
easily yields the natural distinguished triangle 
 \begin{equation}\label{emgys}
\mgbmx(U) 
{\to}\mgbmx(Y) 
{\to} \mgbmx(Z). 
\end{equation} 
(cf. 
\S1.3.8(BM3) of \cite{bondegl}).  Certainly, this triangle yields the corresponding long exact sequence for any cohomology theory $H$ defined on 
$\dkx$.
We define $H^{BM}_i(Y)$ as $H(\mgbmx(Y)[-i])$ (so, we ``lift'' the index $i$ and ``change the sign of the degree in the usual way''; the reason for doing this is that we want this notation to be compatible with the usual one for $K$-theory).

\item\label{imgbme}
More generally, let $Y_l^\al$ be the components of some  stratification $\al$ of  $Y\re$. 
Then combining obvious induction with (\ref{eglu1}) and Theorem \ref{tcd}(II.\ref{itr})   we obtain that any $M\in \obj \dk(Y)$ belongs to the extension-closure of $\{\overline{j}_{l,!}\overline{j}_{l}^*(M)\}$, where $\overline{j}_l\colon Y_l^\al\to Y$
are the corresponding morphisms. 
It easily follows that $\mgbmx(Y)$ belongs to the extension-closure of $\{\mgbmx(Y_l^\al)\}$. This observation is especially useful for us 
if $\al$ is a regular stratification. 


\item\label{imgbmz} Certainly, we also have distinguished triangles similar to (\ref{emgys}) in $\dkzx$ (for any nice $X$). Hence $\dkzcx$ contains $g_!(\z_Z)$ for $g\colon Z\to X$ being an arbitrary finite type morphism. This allows to conclude the proof of Theorem \ref{tcd}(I.\ref{imotgenf}).

\item\label{imgbmles}
Now we apply 
 part \ref{imgbm} of this remark to (our version of) $K$-theory; this corresponds to $H\colon M\mapsto \dkx(M,\oo_X)$. We fix a 
 quasi-projective $Z/X$ and choose a certain smooth $Y/X$ containing it as a closed subscheme, $U=Y\setminus Z$. Then we obtain a long exact sequence
\begin{equation}\label{ekthles}
\dots\to \kk_i^{BM,X}(Z)\to \kk_{i}(Y)\to \kk_{i}(U)\to \dots
\end{equation}
 here we use the fact that $g_!$ is left adjoint to $g^*$ if $g$ is smooth, and consider the cohomology theory represented by $\oo_X$.

Now let $g\colon X'\to X$ be a morphism of schemes (that is separated and not necessarily of finite type).
Then, considering the cohomology theory represented by $g_*(\oo_{X'})$ and applying the adjunction
$g^*\dashv g_*$, we get a long exact sequence 
\begin{equation}\label{ekthlesgen}
\hspace{-3cm}
\begin{tikzpicture}
\node (a0) at (1,0) {$\dots$};
\node (a1) at (3.5,0) {$\kk_{i+1}(Y\times_X X')$};
\node (a2) at (7,0) {$\kk_{i+1}(U\times_X X')$};
\node (b0) at (0,-2.5) {$\kk_i^{BM,X'}(Z\times_X X')$};
\node (b1) at (4,-2.5) {$\kk_{i}(Y\times_X X')$};
\node (b2) at (7,-2.5) {$\kk_{i}(U\times_X X')$};
\node (c0) at (0,-4) {$\dkx(\mgbmx(Z),g_*(\oo_{X'})[-i])$};
\draw[->]
(a0) edge (a1)
(a1) edge (a2)
(b0) edge (b1)
(b1) edge (b2);
\draw[->]
(a2) edge[out=0,in=180] (b0);
\draw[->]
(b0) edge node[auto] {$\cong$} (c0);
\end{tikzpicture}
\end{equation}
It certainly follows that $\kk_i^{BM,X'}(Z\times_X X')$ vanishes for all $i< -n$, where $n$ is a fixed integer, if and only if $\kk_{i}(Y\times_X X')\cong \kk_{i}(U\times_X X')$
for $i< -n$ and  $\kk_{-n}(Y\times_X X')$ surjects onto $\kk_{-n}(U\times_X X')$. 

\item\label{isuspect} We suspect that $\kk_i^{BM,X}(Z)$ 
is naturally 
 isomorphic to the 
 $\kk$-theory of $Y$ with the support on $Z$; possibly we will treat this question in a subsequent paper (perhaps, using the results of \cite{papiro} or Theorem 1.18 of \cite{navariemroch}). 
At least, 
 part \ref{imgbmles} of this remark yields the following: 
 $\kk_i^{BM,X'}(Z\times_X X')$ vanishes for all $i\le -n$ if and only if the same property is fulfilled for 
the $\kk$-theory of $Y\times_X X'$ with the support on $Z\times_X X'$.

\item\label{imgbmed} We will also need a certain (Verdier) dual to part \ref{imgbme} of this remark. 
Let $Y_l^\al$ be the components of some  stratification $\al$ of  $Y\re$. Then 
 (\ref{eglu2}) combined with Theorem \ref{tcd}(II.\ref{itr})   yields that any $M\in \obj \dk(Y)$ belongs to the extension-closure of $\{\overline{j}_{l,*}\overline{j}_{l}^!(M)\}$ (for the morphisms $\overline{j}_{l}\colon Y_l^\al\to Y$).

Now assume that $Y$ and all   $Y_l^\al$ are regular. 
Then parts (I.\ref{ipura},II.\ref{ianr}) of the theorem yield that $\overline{j}_l^!(\oo_Y)\cong \oo_{Y_l^\al}$ for any $l$.
Hence 
 $f_*(\oo_Y)$ (for any separated $f\colon Y\to X$) belongs to the extension-closure of $\{f_{l,*}(\oo_{Y_l^\al})\}$, where $f_l=f\circ \overline{j}_l$. It certainly follows that  $f_{1,*}(\oo_{Y_1^\al})$ belongs to the 
extension-closure of $\{f_!(\oo_Y)\}\cup \{(f_{l,*}(\oo_{Y_l^\al}): l\ge 2\}[1]$.

Lastly, assume that $f$ is of finite type. Then for any $N\in \obj \dk(Y)$ the object $f^!(N)$ belongs to the extension-closure of $\overline{j}_{l,*}f_l^!(N)$. Now take $N=\oo_X$  and assume that all  $Y_l^\al$ are regular and quasi-projective over $X$. Then combining parts I.\ref{ipure}, I.\ref{ipura}, and II.\ref{ianr} of Theorem \ref{tcd}
we obtain that $f_l^!(N)=f^!(\oo_X)\cong \oo_{Y_l^\al}$. Thus $f^!(\oo_X)$ belongs to extension-closure of $\{\overline{j}_{l,*}(\oo_{Y_l^\al})\}$.


\item\label{iaxioms} 
Below we will 
 need only those properties of $\dk(-)$ that are listed in our Theorem. 
Thus one may consider it as a list of ``axioms'' for a system of triangulated categories.

In particular, the authors do not claim that all possible constructions of the categories $\dk(-)$ possessing these properties are isomorphic. 
 Moreover, one can probably consider the following generalization of  our setting: for $R$ being an arbitrary torsion-free coefficient ring one may define $\dk(S)$ as the homotopy category of the category of modules over $\kglp_S\otimes R$ (in $SH(S)$). Yet such a generalization will not affect our main results significantly.



\item\label{itate}
We will not need use the tensor structure much in this paper. Yet we note that for Beilinson motives (i.e., for Voevodsky motives with rational coefficients that were the central subject of \cite{cd}, \cite{hebpo}, and \cite{brelmot}, and will also be considered in \S\ref{scompvoevet} below) there is a certain particular case of tensor products that is very important (this is also the case for $cdh$-motives considered in \cite{cdint} and \cite{bonivan}).

In $K$-motives we have $f_!(\oo_{\p^1(X)})\cong \oo_X\bigoplus \oo_X$  for any ($R$)-nice $X$. Yet for Beilinson motives (as well as for $R$-linear $cdh$-motives over characteristic $p$ nice schemes, where $R$ is unital ring such that  $p$ is invertible in $R$ whenever it is positive) we have 
$f_!(\oo_{\p^1(X)})\cong \oo_X\bigoplus \oo_X\lan -1\ra $ for a certain $\otimes$-invertible object $\oo_X\lan -1 \ra$ (that is often denoted by  $\oo_X(-1)[-2]$; it is not isomorphic to $\oo_X$) instead (so, Beilinson motives are not ``periodic'').  Tensor products by    $\oo_X\lan -1\ra$ and by its tensor powers (including negative ones; these product functors are called Tate twists) commute with all the motivic image functors, 
  whereas the corresponding ``axioms'' of Beilinson motives differ from their $K$-analogues in part I.\ref{ipure} and I.\ref{ipura} of  Theorem \ref{tcd} by  certain 
(locally constant) Tate twists; see  the beginning of \S\ref{sgengab} below for more detail.

\end{enumerate}
\end{rema}

\subsection{Weight structures: reminder}\label{sbws}

\begin{defi}\label{dwstr}
\begin{enumerate}[I.]
\item For a triangulated category $\cu$, a pair of classes
$$
\cu_{w\le 0},\; \cu_{w\ge 0}\subset\obj \cu
$$
will be said to define a {\it weight structure} $w$ for $\cu$ if 
they  satisfy the following conditions:
\begin{enumerate}[(i)]
\item $\cu_{w\ge 0}$ and $\cu_{w\le 0}$ are 
 Karoubi-closed in $\cu$
(i.e., contain all $\cu$-retracts of their objects).

\item {\bf Semi-invariance with respect to translations:}
$$
\cu_{w\le 0}\subset \cu_{w\le 0}[1],\quad\cu_{w\ge 0}[1]\subset
\cu_{w\ge 0}.
$$

\item {\bf Orthogonality:}
$$
\cu_{w\le 0}\perp \cu_{w\ge 0}[1].
$$

\item {\bf Weight decompositions:}
for any $M\in\obj \cu$ there
exists a distinguished triangle
\begin{equation}\label{wd}
B\to M\to A\stackrel{f}{\to} B[1]
\end{equation} 
such that $A\in \cu_{w\ge 0}[1],\  B\in \cu_{w\le 0}$.
\end{enumerate}
\item The full category $\hw\subset \cu$ whose object class is
$\cu_{w=0}=\cu_{w\ge 0}\cap \cu_{w\le 0}$ 
 will be called the {\it heart} of 
$w$.


\item $\cu_{w\ge i}$ (resp. $\cu_{w\le i}$, resp.
$\cu_{w= i}$) will denote $\cu_{w\ge
0}[i]$ (resp. $\cu_{w\le 0}[i]$, resp. $\cu_{w= 0}[i]$).


\item We will  
call
$\cu^b=(\cup_{i\in \z} \cu_{w\le i})\cap (\cup_{i\in \z} \cu_{w\ge i})$ the class of {\it bounded} 
objects of $\cu$. We will say that $w$ is bounded if $\cu^b=\obj \cu$.

Besides, we will say that elements of $\cup_{i\in \z} \cu_{w\ge i}$ are {\it bounded below}.

\item Let $\cu$ and $\cu'$ 
be triangulated categories endowed with
weight structures $w$ and
 $w'$, respectively; let $F\colon\cu\to \cu'$ be an exact functor.
\end{enumerate}
The functor $F$ will be said to be {\it left weight-exact} 
(with respect to $w,w'$) if it maps
$\cu_{w\le 0}$ into $\cu'_{w'\le 0}$; it will be called {\it right weight-exact} if it
maps $\cu_{w\ge 0}$ to $\cu'_{w'\ge 0}$. $F$ is called {\it weight-exact}
if it is both left 
and right weight-exact.

\end{defi}

\begin{rema}\label{rstws}

1. A weight decomposition (of any $M\in \obj\cu$) is (almost) never canonical;
still  (any choice of) a pair $(B,A)$ coming from  (\ref{wd}) will be often denoted by $(w_{\le 0}M,w_{\ge 1}M)$. 
More generally, for any $m\in \z$ shifting a weight decomposition of  $M[-m]$  by $[m]$ we obtain a  distinguished triangle $ w_{\le m}M\to M\to w_{\ge m+1}M $
with some $ w_{\ge m+1}M\in \cu_{w\ge m+1}$, $ w_{\le m}M\in \cu_{w\le m}$; we will call it an {\it $m$-weight decomposition} of $M$.


2.  A  simple (and yet useful) example of a weight structure comes from the stupid
filtration on 
$K(B)$ (for an arbitrary additive category $B$; see \S\ref{snotata}). 
We take $K(B)_{w\le 0}=K(B)^{\ge 0}$ (in the notation described above), and  take  $K(B)_{w\ge 0}$ being the similarly defined $K(B)^{\le 0}$. We call this weight structure the {\it stupid} one; see Remark 1.2.3(1) of \cite{bonspkar} for more detail.

3. In the current paper we use the ``homological convention'' for weight structures; 
it was previously used in \cite{hebpo}, 
   \cite{brelmot},  and in successive papers,
  whereas in 
\cite{bws} and in \cite{bger} the ``cohomological convention'' was used. In the latter convention 
the roles of $\cu_{w\le 0}$ and $\cu_{w\ge 0}$ are interchanged, i.e., one considers   $\cu^{w\le 0}=\cu_{w\ge 0}$ and $\cu^{w\ge 0}=\cu_{w\le 0}$. For example,  a complex $M\in \obj K(B)$ whose only non-zero term is the fifth one 
 has weight $-5$ (with respect to the stupid weight structure) in the homological convention, and has weight $5$ in the cohomological convention. Thus the conventions differ by ``signs of weights''. 
  
 
\end{rema}

Now  we recall some basic 
properties of weight structures. 

\begin{pr} \label{pbw}
Let $\cu$ be a triangulated category endowed with a weight structure $w$, $M\in \obj \cu$, $i,j\in \z$. Then the following statements are valid.

\begin{enumerate}

\item \label{idual}
The axiomatics of weight structures is self-dual, i.e., for $\du=\cu^{op}$
(so $\obj\cu=\obj\du$) there exists the (opposite)  weight
structure $w'$ for which $\du_{w'\le 0}=\cu_{w\ge 0}$ and
$\du_{w'\ge 0}=\cu_{w\le 0}$.

\item\label{iextw}  $\cu_{w\le i}$, $\cu_{w\ge i}$, and $\cu_{w=i}$
are Karoubi-closed and extension-closed in $\cu$ (and so, additive). 

\item\label{iwtrun}
Assume that $j<i$. Then for any choices of  weight decompositions corresponding to the  rows of the square
 $$\begin{CD} w_{\le j} M@>{}>>
M \\
@VV{c_{ji}}V@VV{\id_M}V\\ 
w_{\le i} M@>{}>> M \end{CD}$$ 
there exists a unique  morphism $c_{ji}$ making it commutative. Moreover, $\co(c_{ji})\in \cu_{[j+1,i]}$.


\item\label{iort} 
 $\cu_{w\ge i}=(\cu_{w\le i-1})^{\perp}$ and $\cu_{w\le i}=\perpp \cu_{w\ge i+1}$.

\item\label{iextcub} 
 The class $\cu_{[i,j]}$ equals the extension-closure of  $\cup_{i\le k\le j}\cu_{w=k}$

 \item \label{isum}
    If $A\to B\to C\to A[1]$ is a $\cu$-distinguished triangle and
$A,C\in \cu_{w= 0}$, then $B\cong A\oplus C$.

\item \label{iwdext} If $A\to C\to B$ is a $\cu$-distinguished triangle then for any fixed  $w_{\le i-1}A, w_{\le i-1}B$ there exists  a weight decomposition of $C$ such that $w_{\le i-1}C$ is an extension of $w_{\le i-1}B$ by  $w_{\le i-1}A$.


\item \label{iadj} 
Let $\du$ be a  triangulated category endowed with a weight structure $v$; let $F\colon\cu \leftrightarrows \du:\!G$ be exact 
adjoint functors. Then $F$ is left weight-exact if and only if $G$ is right weight-exact.

 \item\label{iortprecise}
Assume that $M$ belongs to the envelope (see \S\ref{snotata}) of some class of $M_j\in \obj \cu$ (for $j\in J$); we fix a choice of  $w_{\le i-1}M_j$ (for these objects).
Then $M\in \cu_{w\ge i}$ if and only if  $w_{\le i-1}M_j\perp M$.

In particular, if  $J=J_1\cup J_2$ such that $M_j\in \cu_{w\ge i}$ for any $j\in J_1$ and  $M_j\in \cu_{w\le i-1}$ for any $j\in J_2$, then  it suffices to check whether $M_j\perp M$ for all $j\in J_2$.

\end{enumerate}
\end{pr}

\begin{proof} 

Assertions \ref{idual}--\ref{isum}
 are 
 contained in Theorem 2.2.1 of \cite{bger} (whereas their proofs relied on  \cite{bws}) and assertion \ref{iwdext} is a part of Lemma 1.5.4 of \cite{bws}
(pay attention to Remark \ref{rstws}(3)!). 

Assertion \ref{iadj} is just Proposition 1.2.3(9) of \cite{brelmot}.

It remains to verify assertion \ref{iortprecise}. Obviously, the ``in particular'' part of the assertion is just a particular case of the general statement preceding it; so we prove the latter. Furthermore, the orthogonality axiom of weight structures yields that $N\perp M$ if $N\in \cu_{w\le i-1}$ and $M\in \cu_{w\ge i}$; so the ``only if'' part of the statement is clear.

We prove the converse implication.
We start with the following easy observation: for any choice of an $i-1$-weight decomposition triangle $w_{\le i-1}M\to M\to w_{\ge i}M$ we have $M\in \cu_{w\ge i}$ if and only if $w_{\le i -1}M\perp M$. Indeed, the latter condition yields that $M$ is a retract of $w_{\ge i}M$ (whereas the converse implication is immediate). 

Hence it suffices to verify that any object $M$ in the envelope of   $\{M_j\}$ possesses a shifted weight decomposition (as above) such that $w_{\le i-1}M$ belongs to the envelope of $\{w_{\le i-1}M_j\}$. To this end it certainly suffices to combine assertion \ref{iwdext} with  the following statement: for $A$,  some fixed  $w_{\le i-1}A$ and $C$ being a retract  of $A$ there exists  a weight decomposition of $C$ such that $w_{\le i-1}C$ is a retract of $w_{\le i-1}A$.  The latter result  can easily be established using the corresponding arguments from the proof of Lemma 5.2.1 of ibid. (at least, in the case where $\cu$ is Karoubian; cf. the remark below).

\end{proof}

\begin{rema}\label{rvirt}
Whereas the case of a Karoubian $\cu$ is certainly sufficient for the purposes of the current paper, we note that there exists an alternative (and more elegant) proof of  assertion \ref{iortprecise}. It relies on the properties of the so-called virtual $t$-truncations of cohomological functors (see \S2.3 of \cite{bger} and Appendix A.3 of \cite{brelmot}). We consider the functor $F=\tau_{\ge 1-i}(\cu(-,M))\colon\cu^{op}\to \ab$; note that it is cohomological (also).  If $w_{\le i-1}M_j\perp M$  (for all $j$) then $F(M_j)=0$. Hence in this case we have $F(M)=\ns$. It remains to note that the latter implies that $M\in \cu_{w\ge i}$.


\end{rema}

In \S\ref{snclength} and \S\ref{scompvoevet} below we will need some properties of weight structures ``extended'' from subcategories of compact objects.

\begin{pr}\label{pextws}
Let $\cu'\subset \cu$ be triangulated categories  such that $\cu$ contains all small coproducts of its objects, $\cu'$ is
is essentially small, and  the objects of $\cu'$ are compact in $\cu$.  Let $w'$ be a bounded weight structure on $\cu'$.
Then the following statements are valid.

\begin{enumerate}
\item\label{iextws}
The sets $\cu_{w\ge 0}=\cu'_{w'\le -1}{}^{\perp_{\cu}}$ and $\cu_{w\le 0}= {}^{\perp_{\cu}}(\cu_{w\ge 0}[1])$ yield a weight structure on $\cu$.

\item\label{iextwh}
$\hw$  
equals the $\cu$-Karoubi-closure of the category of all $\cu$-coproducts of  objects of $\hw'$.

\item\label{ibort} Assume that $\cu'_{w'\le 0}$ is the envelope of some set $\{C_i\}$ of its objects. Then $\cu_{w\ge 0}=\{C_i[-1]\}^{\perp_{\cu}}$. 

\item\label{iextwe} The embedding $\cu'\to \cu$ is weight-exact (with respect to $w$ and $w'$). 

\item\label{iextweadj} Let $v$ be a weight structure on a triangulated category $\du$;
 let $F\colon\cu \leftrightarrows \du:\!G$ be adjoint exact functors. 
Then the following statements are equivalent: 
\begin{enumerate}
\item\label{i1}  $F( \cu'_{w'\le 0})\subset \du_{v\le 0}$.


\item\label{i2} $G$ is right weight-exact.

\item\label{i3} $F$ is left weight-exact.

\end{enumerate}

\item\label{iextwcommcoprod}
For $(\cu,w)$ and $(\du,v)$ as above and any exact $G\colon\du\to \cu$ that commutes with (small) coproducts we have the following: $G$ is left weight-exact whenever $G(\du'_{v'\le 0})\subset \cu_{w\le 0}$.

\item\label{iextwcoprod}
For any $M_i\in \obj \cu$ (for $i$ running through some index set) we have $\coprod M_i\in \cu_{w\ge 0}$ (resp. $\coprod M_i\in \cu_{w\le 0}$) if and only if  $ M_i\in \cu_{w\ge 0}$ (resp. $ M_i\in \cu_{w\le 0}$) for all $i$.

\end{enumerate}

\end{pr}
\begin{proof}

We will apply the following obvious observation and denote it by the symbol (*): for any objects $M_i\in \obj \cu$  and $M\subset \obj \cu$ the class $\{M_i\}\perpp\subset M\perpp$ whenever $M$ lies in the envelope of $\{M_i\}$.

 The orthogonality axiom  of weight structures immediately gives $\cu'_{w'=0}\perp_{\cu} \cu'_{w'=0}[i]$  for all $i>0$; 
 in other words, the class $\cu'_{w'=0}$ is {\it negative} in $\cu$ 
 (see Remark 2.2.2(3) of \cite{bsnew}). Since the elements of $\cu'_{w'=0}$ are compact in $\cu$, $\cu'_{w'=0}$ is also {\it class-negative} in $\cu$ (see Definition 1.2.2(10) and Remark 2.3.2(2) of ibid.). Applying Corollary 2.3.1(1) of ibid. we obtain the existence of a class $L\subset \obj \cu$ such that the couple $(L,(\cup_{i<0}\cu'_{w'=0}[i]){}^{\perp_{\cu}})$ is a weight structure (on $\cu$). Next, combining 
the observation (*) with Proposition \ref{pbw}(\ref{iextw},\ref{iextcub}) we obtain that  the class $\cu'_{w'\le -1}{}^{\perp_{\cu}}$ equals $(\cup_{i<0}\cu'_{w'=0}[i]){}^{\perp_{\cu}}$. Invoking Proposition \ref{pbw}(\ref{iort}) we conclude that the couple described in assertion \ref{iextws} is a weight structure indeed. 

Applying Corollary 2.3.1(1) of ibid. once again we also obtain assertion \ref{iextwh}.

Assertion \ref{ibort} immediately follows from our observation (*). 

Assertion \ref{iextwe} is an immediate consequence of the description of $w$ along with Proposition \ref{pbw}(\ref{iort}).

Given assertion \ref{iextws}, assertion \ref{iextweadj} is provided by Remark 2.1.5(3) of \cite{bpure}.

 Assertion \ref{iextwcommcoprod} is given by Corollary 2.3.1(1) of \cite{bsnew} also.

Lastly, loc. cit. also says that $w$ is {\it smashing} in the sense of Definition 1.2.2(7) of ibid., and the latter fact immediately implies assertion \ref{iextwcoprod}.

\end{proof}

\begin{rema}\label{rextws}
1. Theorem 5 of \cite{paukcomp} states that one can obtain a weight structure on $\cu$ ``starting from'' the right envelope of any set of objects of $\cu$ (instead of
$\cu_{w\le -1}$ in assertion \ref{iextws}); cf. also  Theorem 4.4.3 of \cite{bpure}(1)  for a more general statement.  Yet we do not need these more general existence results  in the current paper.

2. One can obtain several more properties of weight structures obtained this way by  combining loc. cit. with Proposition 2.5.1 of ibid. 

\end{rema}

\subsection{On weight complexes and weight spectral sequences}\label{swss}

Now we recall some of the properties of  weight complexes and weight spectral sequences.\footnote{The term "weight complex" originates from \cite{gs}; yet the functor of Gillet and Soul\'e was (essentially) extended to Voevodsky motives over a field only in \cite{mymot}, whereas the current general definition was given in \cite{bws}; cf. also \S2.4 of \cite{bpure}. Furthermore, our definition of weight spectral sequences (essentially and) vastly generalizes the one of Deligne; see Remark 2.4.3 of \cite{bws}, \S3.6 of \cite{brelmot},  and  Proposition \ref{phomcoh} below.}  The only place preceding \S\ref{scompvoevet} where they will be needed is Remark \ref{rwwc}; so that the reader mostly interested in sections \ref{sprwchow}--\ref{slength} may ignore the current section.

\begin{pr} \label{pwc}
Let $\cu$ be a triangulated category endowed with a weight structure $w$, $M\in \obj \cu$, $n\in \z$. Then the following statements are valid.

\begin{enumerate}

\item\label{iwc0} For all $i\in \z$ fix some choices of $w_{\le i}M$ and
denote $\co(c_{i-1,i})[-i]$ (see Proposition \ref{pbw}(\ref{iwtrun}) by $M^{-i}$. Define (a choice of) the weight complex $t(M)$ for $M$ as the complex whose terms are $M^j$ (for $j\in \z$) and the connecting morphisms are given as the corresponding compositions $M^{j}\to (w_{\le -j-1} M)[j+1]\to M^{j+1}$.

Then $t(M)$ is a complex indeed (i.e., the square of the boundary is zero); all $M^i$ belong to $\cu_{w=0}$ (so, we are able to consider $t(M)$ as an object of $K(\hw)$).

\item\label{iwc2} $M$ determines 
$t(M)$ up to a 
 homotopy equivalence.

\item\label{iwc3} If $M$ is bounded below, then  $M\in \cu_{w\ge -n}$ if and only if $t(M)\in K(\hw)^{\le n}$.

\item\label{iwcenv} If $M$ belongs to the envelope of certain $M_i\in \obj \cu$ then $t(M)$ belongs to the $K(\hw)$-envelope of $t(M_i)$.

\item\label{iwctow} Let $d_i\colon M_{\le i-1}\to M_{\le i}$ for $i\in \z$ be a chain of $\cu$-morphisms such that $\co(d_i)\in \cu_{w=i}$ for all $i$,  $M_i=0$ for $i\ll 0$, and $M_i\cong M$ (with $d_{i+1}$ being isomorphisms) for $i\gg 0$. Then $M_{\le i}$ give certain choices of $w_{\le i}M$ and $d_i$ yield the corresponding $c_{i-1,i}$.

\end{enumerate}
\end{pr}
\begin{proof} All of these statements 
except the last one are essentially 
contained in Theorem 3.3.1 of \cite{bws}  (yet pay attention to Remark \ref{rstws}(3)!). Assertion \ref{iwctow} is precisely Theorem 2.2.1(14) of \cite{bger}.
\end{proof}

Now we recall some basics on (general) weight spectral sequences. 
It will be more convenient for us in this paper to consider them for homological functors (including  perverse \'etale homology of motives); certainly, dualization is not a problem (cf. \S2.4 of \cite{bws}).

\begin{pr}\label{pwss}
Let $\au$ be an abelian category; $H\colon\cu\to \au$ be any 
 functor.

I. For any $m\in \z$ the object $(W_{m}H)(M)=\imm (H(w_{\le m}M)\to H(M))$
does not depend on the choice of $w_{\le m}M$; moreover, it is $\cu$-functorial in $M$.

II. Now let $H\colon\cu\to \au$ be a homological functor; for any $r\in \z$ denote $H\circ [r]$ by $H_r$.



Then the following statements are valid.

1. There exists a ({\it weight})  spectral sequence $T=T_w(H,M)$ with $E_1^{pq}=
H_q(M^p)$ such that the differentials for $E_1T_w(H,M)$ come from $t(M)$.
  It converges to $ H_{p+q}(M)$ if $M$ is bounded. 
	
2. $T_w(H,M)$ is $\cu$-functorial in $M$ (and does not depend on any choices) starting from $E_2$.

3.  If $M$ is bounded, then the step of filtration given by ($E_{\infty}^{l,m-l}:$ $l\ge k$)
 on $H_{m}(M)$ (for some $k,m\in \z$) equals  $(W_{-k}H_{m})(M)$.

\end{pr}

\begin{proof}

Immediate from Proposition 2.1.2 and Theorem 2.3.2 of \cite{bws}  
 (cf. also Proposition 1.3.2 of \cite{bmm}).
\end{proof}

\begin{coro}\label{cdetect}
Let $\cu,\au,H,M$ be as in (part II of) the proposition, $n\in \z$, and consider the following conditions:

1. $E_2^{pq}T_w(H,M)\neq 0$ for some $q\in \z$ and $p>n$.

2. $(W_{-n-1}H_q)(M)\neq 0$  for some $q\in \z$. 

Then  the following statements are valid.

1. If condition 1 or 2 is fulfilled then $M\notin \cu_{w\ge -n}$.

2. Assume that $M$ is bounded below. Then condition 2 implies condition 1.


\end{coro}
\begin{proof} 1. It suffices to note that one can take $w_{\le i} M=0$ for all $i<-n$ when computing weight filtrations and weight spectral sequences (using weight complexes).

2. Assume that $M\in \cu_{w\ge k}$ for some $k\in\z$. We should check that $(W_{-n-1}H_q)(M)= 0$ for all $q\in \z$ and $p>n$ whenever
$E_2^{pq}T_w(H,M)= 0$ (for  $p>n$).
 
Fix some choices of $w_{\le i}M$ for all $i\in\z$; we take  $w_{\le i}M=0$ for $i<k$. 
According to Proposition \ref{pbw}(\ref{iwtrun}) the corresponding object $M'=w_{\le - n}M$ belongs to $\cu_{[k,-n]}$ and $w_{\le j}M$ yield some 
 choices of $w_{\le j}M'$ for $j\le -n$. 
Hence the corresponding choice of $t(M')$ is the stupid truncation of $t(M)$ in degrees $\ge n$. 

Now we consider the weight spectral sequence  $T_w(H,M)$. According to Proposition \ref{pwss}, 
the vanishing of $E_2^{pq}T_w(H,M)=E_2^{pq}T_w(H,M')$ for $p>n$ yields that 
$(W_{-n-1}H_q)(M')=\imm (H(W_{-n-1}M')\to H(M'))=0$ (for all $q$). Thus $(W_{-n-1}H_q)(M)= 0$ also.

\end{proof}

\section{On  Chow-weight structures for $K$-motives; relating the ``motivic length'' with negative $K$-groups}\label{sprwchow}

In this section we introduce the Chow weight structures for $\dk(-)$ and relate them to (negative) $\kk$-groups.

In \S\ref{swchow} we introduce and study the Chow weight structures on the subcategories $\dk^c(-)\subset \dk(-)$ of compact objects. Our exposition closely follows the arguments of \cite{bonivan}. 

In \S\ref{sgenlength} we use a simple calculation to establish the relation between the weights of $f_*(\oo_Y)$ (for a separated finite type $f\colon Y\to X$) to the (negative degree, Borel--Moore) $\kk$-groups of $Y$-schemes. We prove several equivalent conditions for $f_*(\oo_Y)\in \dkx_{\wchow\ge -n}$.

In \S\ref{snclength} we use Proposition \ref{pextws} to extend the Chow weight structures and their properties from $\dk^c(-)$ to $\dk(-)$. This allows us to lift the finiteness of  type restriction on $f$ that was imposed in the criteria of Theorem \ref{testw}.

\subsection{On the Chow weight structure for  compact $K$-motives}\label{swchow} 

Similarly to \cite{brelmot} and \cite{bonivan}, the properties of $K$-motives listed in Theorem \ref{tcd} yield the existence of certain {\it Chow} weight structures for $\dkc(-)\subset \dk(-)$. 
We will start with the properties of the ``compact'' Chow weight structure.

Sometimes we will need the definition of ``global'' Chow motives over a scheme.

\begin{defi}\label{dchow}
For a scheme $S$ we define the category $\chows$ of Chow motives over $S$ as the Karoubi-closure of $\{\mgbm_S(P)\}$ in $\dks$; here $P$ runs through all finite type 
 regular schemes that are 
projective over $S$.
\end{defi}

Note that the ``Beilinson-motivic'' versions of Chow motives (as defined here)  have played a crucial role in \cite{hebpo} and in \cite[\S2.1--2.2]{brelmot}; cf. also \S2.3 of \cite{bonivan}. 

\begin{theo}\label{twchow}

Let $X$ be an $\lam$-nice  scheme. Then the following statements are valid.

I. There exists a bounded weight structure $\wchowc=\wchowc(X)$ for $\dkcx$ 
that possesses the following descriptions.

\begin{enumerate}
\item\label{iwchow1} $\dkcx_{\wchowc\le 0}$
is  the right envelope of $\{\mgbm_X(P)\}$,
$\dkcx_{\wchowc\ge 0}$ is the  left envelope of $p_*(\oo_P)$ for 
$p\colon P\to X$ running through all 
  compositions of a smooth projective morphism with a finite universal homeomorphism whose domain is regular and with an immersion. 

\item\label{iwchow2} Moreover, $\dkcx_{\wchowc\le 0}$ is the right envelope of  $\{\mgbmx(T)\}$ for $T$ running through all finite type $X$-schemes; $\dkcx_{\wchowc\ge 0}$ is the left envelope of 
  $\{t_*(\oo_T)\}$ for $t\colon T\to X$ running through all finite type morphisms with {\bf regular} domains.

\item\label{iwchow3} $\dkcx_{\wchowc \ge 0}$ is the right envelope of  $\{\mgbmx(T)\}$ for $T$ running through all quasi-projective $X$-schemes; 
$\dkcx_{\wchowc\ge 0}$ is the left envelope  of  $t_*(\oo_T)$ for $t\colon T\to X$ running through all quasi-projective morphisms with regular domains.

\item\label{iwchow4}  If $X$ 
 of finite type over a field, 
then 
$\dkcx_{\wchowc\le 0}$ is the right envelope of $\obj \chow(S)$, $\dkcx_{\wchowc\ge 0}$ is the left envelope of $\obj\chow(S)$.

\end{enumerate}

II. 
 Let $f\colon Y\to X$ be a (separated) scheme morphism. 
Then the following statements are valid for the ``compact versions'' of the motivic functors.
\begin{enumerate} 

\item\label{iwechow5} $\oo_X\in \dkcx_{\wchowc\le 0}$.

\item\label{iwechow6} If the reduced scheme $X_{red}$ associated to $X$ is regular then $\oo_X\in \dkcx_{\wchowc= 0}$.

\item\label{iwechow1}
$f^*$  is left weight-exact.

\item\label{iwechow2} If $f$ is of finite type then $f^!$ and $f_*$  are right weight-exact, and $f_!$ is left weight-exact.

\item\label{iwechow4} If $f$ is proper then $f_*= f_!$ is also weight-exact.

\item\label{iwechow3} Moreover, $f^*\cong f^!$ is weight-exact   if $f$ is either smooth or a finite universal homeomorphism.

\item\label{iwechow7} Moreover,   $f^*$ is weight-exact if it is
the (inverse) limit of  an essentially affine system such that the transition morphisms are compositions of smooth morphisms and finite universal homeomorphisms.

\item\label{iwechow8} Let $\xxx$  denote the set of (Zariski) points of $X;$ for a $K\in \xxx$ the corresponding morphism $K\to X$ is denoted by $j_K$.
Then $M\in \dkcx_{\wchowc\le 0}$ if and only if for any $K\in \xxx$ we have $j_K^*(M)\in \dkc(K)_{\wchowc\le 0}$.

\item\label{iwechow9}
Define $j_K^!$ using the ``standard method'' (cf. \S2.2.12 of \cite{bbd}) as follows: we decompose $j_K$ as $K\stackrel{j_K^0}{\to} \overline{K}\stackrel{i}{\to} X$ ($\overline {K}$ is the closure of $K$ in $X$) and set $j_K^!=i^!\circ j_K^{0*}$.

Then $M\in \dkc(X)_{\wchowc\ge 0}$ if and only if for any $K\in \xxx$ (for $\xxx$ as above) we have $j_K^!(M)\in \dkc(K)_{\wchowc\ge 0}$.

\end{enumerate}

\end{theo}
\begin{proof}
I. The methods applied in the proof of Theorem 2.1.3, Theorem 2.2.1(2d)  (along with Remark 2.2.2(1)), and Proposition 2.3.2 (along with Remark 2.3.3(1)) of \cite{bonivan} can be carried over to our context (of $\lam$-linear $K$-motives over $\lam$-nice schemes) without any difficulty (if we apply Theorem \ref{tcd}). 


II. Once again, the arguments used in (\S2.2 of)  \cite{bonivan} 
yield all of the statements without any difficulty.

\end{proof}

\begin{rema}\label{rexplwchow}
\begin{enumerate}
\item\label{rexplwchow-1} The arguments used in  \cite{bonivan} for the proof of the corresponding analogue of  (the existence statement in) Theorem \ref{twchow} were not ``very explicit''; in particular, they do not bound the weights of $\oo_X$ from below (for a non-regular $X$). In the next section we will describe certain more explicit arguments (that can also be applied to $\dm(-)$).

\item\label{rexplwchow-2} The corresponding envelopes in \cite{bonivan} ``differed by stabilization'' via $\otimes \oo_X\lan i \ra$ ($=\oo_X\lan -1 \ra^{\otimes -i}$ for $i\in \z$; see Remark \ref{rcons}(\ref{itate})). The natural analogues of the properties of $\wchowc$ described in this section are also valid for Beilinson motives (over arbitrary nice base schemes) since they possess all the properties needed  to prove them.  

	
\item\label{rexplwchow-3} One of the main benefits of weight structures is that they relate the objects of $\cu$ to the ``more simple'' objects of the corresponding $\hw$ (via weight complexes and weight spectral sequences). So, it certainly makes sense to describe ${\underline{Hw}}^c_{\chow(X)}$ ``more explicitly'' (note that this will also yield a description of the whole ${\underline{Hw}}_{\chow(X)}$; see Proposition \ref{pwchownc}
below and Proposition \ref{pextws}(\ref{iextwh})). 
Now, Theorem \ref{tcd}(II.\ref{igenc}) (combined with Theorem 4.3.2(II) of \cite{bws}) implies that ${\underline{Hw}}^c_{\chow(X)}=\chow(X)$ (see Definition \ref{dchow}) 
whenever either $X$ is of finite type over a field or $\lam=\q$ and $X$ is of finite type over an excellent noetherian scheme of  dimension at most $3$. 
Now, one can ``compute'' morphism groups between two objects of the type $\mgbmx(P)$ where $P$ is regular and projective over $X$ using the method of the proof of \cite[Lemma 1.1.4.(I.1)]{brelmot}.
Computing the composition of morphisms operation for these motives seems to be the most difficult problem here; yet one may probably solve it  by applying the arguments of \cite{jin}.
\end{enumerate}	
	
\end{rema}

\subsection{Motivic weight bounds  in the terms of negative $\kk$-groups: the finite type version}\label{sgenlength}

Now we combine the previous theorem with Proposition \ref{pextws}(\ref{ibort}) (and apply these statements to certain motives of $X$-schemes). 
We start from the case of $X$-schemes of finite type.

\begin{theo}\label{testw} Let $f\colon Y\to X$ be a (separated) finite type morphism of  $\lam$-nice schemes, $n\ge 0$. Then the following conditions are equivalent.

\begin{enumerate}

\item\label{itest1} $f_*(\oo_{Y})\in \dkc(X)_{\wchowc\ge -n}$.

\item\label{itest2} For any 
(separated) finite type
 morphism $P\to X$ the groups $\kk_i^{BM,Y}(P\times_X Y)$ (see Remark \ref{rcons}(\ref{imgbm}, \ref{imgbmles})) vanish for $i<-n$.

\item\label{itestaff} For any 
 morphism $P\to X$ of finite type with regular affine domain the groups $\kk_i^{BM,Y}(P\times_X Y)$  vanish for $i<-n$.

\item\label{itest3} For any  morphism $P\to X$ being the composition of a smooth projective morphism with a finite universal homeomorphism whose domain is regular
 and with an immersion the groups $\kk_i^{BM,Y}(P\times_X Y)$ vanish for $i<-n$.


\item\label{itest4} For any smooth $P/X$ and any open 
embedding $j\colon U\to P$ we have the following: $\kk_{i}(P\times_X Y)=\ns$ for $i<-n$ 
 and  $\kk_{-n}(P\times_X Y)$ surjects onto $\kk_{-n}(U\times_X Y)$.


\end{enumerate}
\end{theo}
\begin{proof}
If condition \ref{itest1} is fulfilled then we certainly have 
$\dkcx_{\wchowc \le n-1}\perp f_*(\oo_{Y})$. Hence for any separated finite type $P/X$ and $i<n$ we have $\mgbmx(P)[i]\perp  f_*(\oo_{Y})$ (see Theorem \ref{twchow}(I.\ref{iwchow2})). 
It remains to apply 
the isomorphism $\kk_i^{BM}(Z\times_X Y)\cong \dkx(\mgbmx(Z),f_*(\oo_{Y}[-i]))$  (mentioned in (\ref{ekthlesgen})) to deduce condition  \ref{itest2}. 

Certainly, condition  \ref{itest2} yields conditions  \ref{itest3} and \ref{itestaff}. Conversely, condition \ref{itestaff} implies condition \ref{itest2} according to Remark \ref{rcons}(\ref{imgbme}).

Similarly to the implication  \ref{itest1} $\implies$ \ref{itest2}, 
condition  \ref{itest3} yields that 
$\mgbmx(P)[i]\perp  f_*(\oo_{Y})$ for any  $P$ of the corresponding type. Hence it implies condition \ref{itest1} by Theorem \ref{twchow}(I.\ref{iwchow2})  (along with Proposition \ref{pbw}(\ref{iort}). 

 Lastly, since $\kk_{i}(P\times_X Y)\cong \kk_{i}^{BM,Y}(P\times_X Y)$ (basically by definition), 
condition  \ref{itest2} implies condition  \ref{itest4}
 according to Remark \ref{rcons}(\ref{imgbmles}). Next,  condition  \ref{itest4} certainly yields that the corresponding facts are true for $P$ being smooth quasi-projective over $X$. Hence the remark cited yields condition \ref{itest1} if we combine it with Theorem \ref{twchow}(I.\ref{iwchow3})
 (along with Proposition \ref{pbw}(\ref{iort}).

\end{proof}

\begin{rema}\label{rcones}
\begin{enumerate}
\item\label{rcones-1} The ``initial'' case $X=Y$, as well as the cases   where $X$ is the spectrum of a field (that may be perfect or just equal to $\com$)  and (the ``final'' case) $X=\spe \lam$ in the theorem (already) appear to be quite interesting. 
Note that by the virtue of the results of the following subsection we can consider arbitrary separated morphisms (not necessarily of finite type) of $\lam$-nice schemes in this remark.

Now we define  $c(Y/X)$ as the minimal $n\in \z$ such that $f_*(\oo_{Y})\in \dkx_{\wchow\ge -n}$.  Theorem \ref{twchow} (along with Proposition \ref{pwchownc}
in the case when $f$ is not of finite type, and with the following part of this remark) yields the following properties of this characteristic: $c(Y/X)\ge 0$; $c(Y/X)= 0$ if $X_{red}$ is regular;  $c(Y'/X)\le  c(Y/X)$ if $Y'$ is smooth over over $Y$;  $c(Y/X)\ge  c(Y/S)$ if $X$ is separated over $S$.

\item\label{rcones-2} Applying   the argument used in the proof of our theorem 
for $P=X$  one also obtains 
  $f_*(\oo_{Y})\notin \dkx_{\wchow\ge 1}$ (since $\kk_0(Y)\neq \ns$).

Yet note that criteria quite similar to the ones above can be proved for the question  whether $\co (f'_*(\oo_{Y'}) \to f_*(\oo_{Y}))$  belongs to $ \dkx_{\wchow\ge -n}$ 
(where the connecting morphism is induced 
by a factorization of $f$ through some finite type  
 $f'\colon Y'\to X$).  
 This (more general) formulation can also be interesting for $n<0$. 

\item\label{rcones-3} Certainly, conditions  \ref{itest3} and \ref{itestaff} 
 can also be reformulated similarly to condition \ref{itest4}. 

 Note also that one may replace the Borel--Moore $K$-theory in 
 these conditions by the corresponding $K$-theory with support; see Remark \ref{rcons}(\ref{isuspect}). 

\item\label{rcones-4} Below we will show that in condition \ref{itest2} one can restrict himself  to a  finite set of  ``test schemes'' $P/X$. Now we will only make two simple remarks  related to this claim.

Firstly, if $X$ is of finite type over a field, it suffices to consider $P$ running through regular projective $X$-schemes (see Theorem \ref{twchow}(I.\ref{iwchow4})). 
Note also that in the case where $X$ is the spectrum of a field itself, one may pass to its perfect closure (using Theorem \ref{twchow}(II.\ref{iwechow7})). After that 
one can consider $P$ being smooth over (the new) $X$, and so we have   $\kk_i^{BM,Y}(P\times_X Y)=\kk_i (P\times_X Y)$.

On the other hand, (\ref{emgys}) implies that one can ``cut test schemes into pieces''. Moreover, one can pass to direct limits here; this yields certain ``Borel--Moore stalks'' (that are closely related to certain coniveau spectral sequences); see 
\cite{bondegl}.
\end{enumerate}
\end{rema}

\subsection{On 
``$\kk$-weight bounds'' for $X$-schemes that are not of finite type}\label{snclength}

Now we prove the ``non-compact'' version of Theorem \ref{twchow}. 
It easily implies that Theorem \ref{testw} can be extended to $Y$ being not necessarily of finite type over $X$.

\begin{pr}\label{pwchownc}
For any ($\lam$-nice)  scheme $S$ let $\wchow=\wchow(S)$ denote the ``extension'' of  $\wchowc(S)$ from $\dkcs$ to $\dks$  constructed via the method of Proposition \ref{pextws}(\ref{iextws}).   Let $f\colon Y\to X$ be a (separated) scheme morphism, $M\in \obj \dkx$. 
Then the following statements are valid.

\begin{enumerate} 

\item\label{iwexchow5} $\oo_X\in \dkx_{\wchow\le 0}$.

\item\label{iwexchow6} If 
$X\re$ (see the end of \S\ref{snotata}) is regular then $\oo_X\in  \dkx_{\wchow= 0}$.

\item\label{iwexchow1}
$f_*$ is right weight-exact; 
$f^*$  is left weight-exact.

\item\label{iwexchow2} If $f$ is of finite type then $f^!$ is right weight-exact and $f_!$ is left weight-exact.

\item\label{iwexchow4} If $f$ is proper then $f_*= f_!$ is also weight-exact.

\item\label{iwexchow3} Moreover, $f^*\cong f^!$ is weight-exact   if $f$ is either smooth or a finite universal homeomorphism.

\item\label{iwexchow9} $f_*(\oo_Y)$ is bounded below.

\item\label{iwexchowstr} Let $\al$ be a stratification of $X$. Then for the corresponding $j_l\colon X_l^\al\to X$ we have the following statements:
$M\in  \dkx_{\wchow\le  0}$ if and only if $j_l^*(M)\in \dk(X_l^\al)_{\wchow \le 0}$ for all $l$; $M\in  \dkx_{\wchow\ge  0}$ if and only if $j_l^!(M)\in \dk(X_l^\al)_{\wchow \ge 0}$ for all $l$.

\item\label{iwexchowoc}  Let $x_i\colon X_i\to X$ be an open cover of $X$. Then 
$M\in  \dkx_{\wchow\le  0}$ (resp.  $M\in  \dkx_{\wchow\ge  0}$) if and only if $x_i^*(M)$ belongs to $ \dk(X_i)_{\wchow \le 0}$ 
(resp.  to  $ \dk(X_i)_{\wchow \ge 0}$) for all $i$.

\item\label{iwexchow7} $f^*$ is weight-exact also in the case where $f$ is
the (projective) limit of  an essentially affine system such that the transition morphisms are compositions of smooth morphisms and finite universal homeomorphisms.

\item\label{iwexchow8} Let $\xxx$  and  $j_K^!$ (for $K\in \xxx$) be as in Theorem \ref{twchow}(II.\ref{iwechow8}). 
Then $M\in \dkx_{\wchow\ge 0}$ if and only if for any $K\in \xxx$ we have $j_K^!(M)\in \dk(K)_{\wchow\ge 0}$. 

\end{enumerate}

\end{pr}
\begin{proof}
All of these statements are easy consequences of Theorem \ref{twchow} along with Proposition \ref{pextws}.

In particular, Proposition \ref{pextws}(\ref{iextwe})  (combined with Theorem \ref{twchow}) yields assertions \ref{iwexchow5} and \ref{iwexchow6}. 

Next, combining   Proposition \ref{pextws}(\ref{iextweadj})  with the adjunctions $f^* \dashv f_*$ and $f_! \dashv f^!$ (for $f$ being a finite type morphism) we obtain the proof of assertions \ref{iwexchow1} and \ref{iwexchow2}.
Assertions  \ref{iwexchow4} and \ref{iwexchow3} follow from the previous two immediately when we take into account the  isomorphisms of functors mentioned in their formulations.

Next, $\oo_Y$ is bounded below (in $\dky$) since $\wchowc(Y)$ is bounded. Hence $f_*(\oo_Y)$ is bounded below also, and we obtain assertion \ref{iwexchow9}.
Next, $\oo_Y$ is bounded below (in $\dky$) since $\wchowc(Y)$ is bounded. Hence $f_*(\oo_Y)$ is bounded below also, and we obtain assertion \ref{iwexchow9}.

The ``only if''  statements in  assertion \ref{iwexchowstr} are immediate from assertions \ref{iwexchow1}
 and \ref{iwexchow2}, respectively. 

To obtain the converse implications one should recall that the classes $\dkx_{\wchow\le  0}$ and $\dkx_{\wchow\ge  0}$ are extension-closed. Thus it suffices to recall that $M$ belongs to the extension-closure of $\{j_{l!}j^*_l(M)\}$ and also to  the extension-closure of $\{j_{l*}j^!_l(M)\}$ (see Remark \ref{rcons}(\ref{imgbme}, \ref{imgbmed})) and apply the remaining statements in assertions \ref{iwexchow1} and \ref{iwexchow2}.

 \ref{iwexchowoc}. Note that $j_i^!=j_i^*$ and take a stratification of $X$ such that each of $X_l^{\al}$ lies in one of the subschemes $X_i$. Then the assertion follows from the previous one easily.

To prove assertion \ref{iwexchow7} it certainly suffices to verify that $f^*$ is right weight-exact. So, for any $M\in \dkx_{\wchow \ge  0}$,  
any finite type $Y'/Y$ and $i<0$ we should check that $\mgbm_Y(Y')[i]\perp f^*(M)$ (see Proposition \ref{pextws}(\ref{ibort}) and Theorem  \ref{twchow}(I.\ref{iwchow2})). Now, we can assume that $Y'$ is defined ``at a finite level'' (i.e.,  for $f=\prli f_i\colon Y_i\to X$  there exists  $i_0$  such that $Y'=Y\times_{Y_{i_0}}Y_0'$ for some finite type $Y'_0/Y_{i_0}$; see  
 Theorem 8.8.2(ii) of \cite{ega43}).  Hence for the corresponding $h_{i_0}\colon Y\to X_{i_0}$ we have  $\mgbm_Y(Y')\cong h_{i_0}^*(\mgbm_{Y_{i_0}}Y_0')$ (see Remark \ref{rcons}(\ref{imgbm})); for the transition morphisms $g_{ii_0}\colon Y_i\to Y_{i_0}$ we also have $\dky(\mgbm_Y(Y')[i],f^*(X))\cong \inli_{i\ge i_0} \dm_{Y_i}(\mgbm_{Y_{i}} Y_i\times_{Y_{i_0}}Y_0', g_{ii_0}^*(M))$ (see Theorem \ref{tcd}(I.\ref{icont}, II.\ref{ianr}).  
 Now, the weight-exactness of $g_{ii_0}$ implies the vanishing of (all terms of) this limit; this concludes the proof of the assertion.

The ``only if'' part of assertion \ref{iwexchow8} follows immediately from the previous assertion (along with assertion \ref{iwexchow2}).
The converse implication  can be proved similarly to its ``compact motivic'' version (i.e., to the corresponding ``half'' of Proposition 2.2.3 of \cite{bonivan}); note that Theorem \ref{tcd}(I.\ref{icont}, II.\ref{ianr}) contains more information than its analogue used in ibid. 

\end{proof}

\begin{coro}\label{ctestwnc}
Theorem \ref{testw} is also valid for $Y$ being not of finite type (and separated) over $X$.
\end{coro}
\begin{proof} The proof of Theorem \ref{testw} can be extended to this context without any difficulty if we apply Proposition \ref{pwchownc} and use Proposition \ref{pextws}(\ref{ibort}) instead of
 Proposition \ref{pbw}(\ref{iort}).

\end{proof}

\begin{rema}\label{rwwc}
\begin{enumerate}
\item\label{rwwc-1} Now we recall that $g_*(\oo_X)$ is bounded below; hence it belongs to $\dkx_{\wchow\ge -n}$ if and only if $t(g_*(\oo_X))\in K(\hw_{\chow}(X))^{\le n}$ (see Proposition \ref{pwc}(\ref{iwc3})).  This criterion does not seem to be very ``practical'' in general (especially if $g$ is not of finite type). Yet certainly there are some cases where  $t(g_*(\oo_X))$ can be computed ``explicitly''. Possibly, we will consider this question in a subsequent paper. 

\item\label{rwwc-2} Once again (following Remark \ref{rcons}(\ref{itate})) we recall that in the cases $\lam=\q$ and $\lam=\zop$ (for any $p\in \z$) there exist ``reasonable'' $\lam$-linear ``Voevodsky-type'' motivic categories  $\dm(-)$ over arbitrary $\lam$-nice schemes. All of the statements of this section 
easily carry over to these categories; the main distinctions is that one should ``take Tate twists into account'' (see the aforementioned remark once again), whereas in Theorem \ref{tcd}(II.\ref{ikmorr}) one should take the corresponding $\lam$-linear motivic cohomology instead of ($\lam$-linear homotopy invariant) $K$-theory. We did not treat this setting in detail since this sort of motivic cohomology (of singular schemes) seems to be ``less popular'' than $K$-theory, and we didn't want to restrict ourselves to two possibilities for $\lam$ only.    
\end{enumerate}
\end{rema}

\section{``Studying weights'' using (general) Gabber's resolution of singularities results}\label{slength}

The main goal of  this section is to establish some weight bounds on objects of the type $f_*(\oo_Y)$ and $\mgbmx(Y)$ for $Y$ being (usually) of finite type over $X$. We also put objects of this type in certain ``explicit'' envelopes (that yield related envelope statements for some choices of $m$-weight decompositions for these motives, where $m$ runs through integers). We also obtain similar results for other ``motivic categories'' possessing properties similar to that of $K$-motives. Since we want our results to be as general is possible, we have to invoke some definitions and results of O. Gabber (as described in \cite{illgabb}). Note also that most of the formulations of this section are motivated by the example described in Remark \ref{rcompvm}(\ref{rcompvm-4}) below.

In \S\ref{sdimf} we introduce a certain modification of Gabber's dimension functions; the reason for this is that the properties of the corresponding function $\de=\de^B$ is ``more satisfactory'' for our purposes than the properties of  Krull dimension.

In \S\ref{sgabblem} we formulate and prove the aforementioned envelope statements for $K$-motives.

In \S\ref{sgenweibelconj} we combine these results with the 
ones of the previous section to obtain a collection of weight bounds and the corresponding bounds on ``the negativity'' of certain (Borel--Moore) $\kk$-groups. 
In particular, we prove that $K_{-s}$ is ``supported in codimension $s$'' (for any $s\ge 0$).

In  \S\ref{sgengab} we extend the results of \S\ref{sgabblem} to (any system of) motivic triangulated categories $\md(-)$ satisfying certain 
(``axiomatic'') assumptions. In particular, this yields a collection of results on \'etale sheaves that appear to be new.

\subsection{On dimension functions}\label{sdimf}

Since we want our results to be valid for a wide range of schemes, we will need a certain ``substitute'' of the Krull dimension function $\dim(-)$.\footnote{The reason is that we want some notion of dimension that would satisfy the following property: if $U$ is open dense in $X$ then its ``dimension'' should equal the ``dimension'' of $X$.} So we need certain ``pseudo-dimension functions'' that we define in terms of Gabber's dimension functions (as introduced in \S XIV.2 of \cite{illgabb} and applied to motives in \cite{bondegl}). 

\begin{defi}\label{ddf}
Let $B$ be a $\lam$-nice scheme; denote by $B^k$ its irreducible components whose dimensions are $d_k$.
Let $y$ be the spectrum of a field that is essentially of finite type over $B$ 
(i.e., it is the generic point of an irreducible scheme of finite type over $B$) and $b$ be its image in $B$.


\begin{enumerate}
\item\label{ddf-1} Throughout this section we will say that a scheme is a   $B$-scheme only if it is separated and of finite type over $B$. 

\item\label{ddf-2} 
If $b\in B^k$ then we define   $\de^k(y)
=d_k-\operatorname{codim}_{B^k}b+\operatorname{tr.\,deg.}k(y)/k(b)$, where $k(y)$ and $k(b)$ are the corresponding residue fields.

 For a general $y$ we define $\de(y)$ as the minimum over all $B^k$ containing $b$ of the numbers $\de^k(y)$.

\item\label{ddf-3} For $Y$ being an $\sz$-scheme we define $\de^{B}(Y)=\de(Y)$ as the maximum over points of $Y$ of $\de(y)$. 
\end{enumerate}
\end{defi}

The results of \cite[\S XIV.2]{illgabb} easily yield that the function $\de$ satisfies the following properties.

\begin{pr}\label{pdimf}
Let $X$ and $U$ be $B$-schemes.

1.  $\de(X)$ is a finite integer that is not smaller than $\dim(X)$. 

2. $\de(U)\le \de(X)+d$ whenever
there exists a 
 $B$-morphism $u\colon U\to X$ generically of dimension at most $d$. 
Moreover, we have an equality here whenever $d=0$ and  $u$ is  dominant, and a strict inequality if 
 the image of $u$ is nowhere dense. 

3. Furthermore, if $c>0$, $U\subset X$, and any irreducible component of $U$ is of codimension at least $c$ in some irreducible component of $X$ (containing it) then $\de(U)\le \de(X)-c$.
\end{pr}
\begin{proof} 
For any irreducible component $B^k$ of $B$ we combine  Proposition XIV.2.2.2 of \cite{illgabb} (cf. also the  proof of Corollary XIV.2.2.4  of loc. cit.) 
with Corollary XIV.2.5.2 of ibid. to obtain 
that the restriction of $\de^k$ to the points of any $B^k$-scheme $Y$ yields a dimension function on it (in the sense of Definition XIV.2.1.10 of ibid.). 
 It follows immediately that $\de^k$ satisfy the obvious analogues of our assertions. 
Combining these statements for all $k$ we easily obtain the result.

\end{proof}

\begin{rema}\label{rdimf}
\begin{enumerate}

\item\label{idimfkrull} If $B$ is of finite type over a field or over $\spe \z$ (more generally, it suffices to assume that $B$ is a Jacobson scheme all of whose components are equicodimensional; see Proposition 10.6.1 of \cite{ega43}) then our method yields a function $\de$ such that $\de(Y)=\dim(Y)$ for any $Y$ that is of finite type over $B$. Thus the reader satisfied with this restricted setting may replace all $\de$-dimensions 
 in this section by Krull dimensions. Moreover, some of the arguments may be simplified. Another easier case is the one of irreducible $B$. In both of these cases our $\de$ is a ``true'' dimension function, i.e., $\de(X)-\de(Z)=\codim_X(Z)$ for $Z\subset X$ being any irreducible $B$-schemes.

\item\label{idcomp} Recall that any $B$-scheme $Y$ possesses a $B$-compactification $\overline{Y}$
 by Nagata's theorem (i.e., $Y$ is open dense in $\overline{Y}$ and $\overline{Y}$ is proper over $B$; see Theorem 4.1 of \cite{conag}). Since $\de(\overline{Y})=\de(Y)$, we obtain $\de(Y)\ge \dim(\overline{Y})$.

Note also that the (Krull) dimensions of all possible $B$-compactifications of $Y$ are equal. 

\item\label{indepdimf} 
It is easily seen for $B'$ being a $B$-scheme that the values of the similarly defined function $\de^{B'}$ is not greater than the (corresponding) ones of $\de$.

\item\label{ilocdimf} 
Actually, in all the statements involving $\de$ before Corollary \ref{cgabber} we will use Proposition \ref{pdimf} as a certain ``axiomatics'' of $\de$. 
In particular, if (a possibly, reducible) $B$ admits a dimension function in the sense of \cite[Definition XIV.2.1.10]{illgabb} then one may use the corresponding ``extension'' of this function 
to $B$-schemes (see Corollary XIV.2.5.2 of ibid.)  instead of $\de$.  Recall also that any $\lam$-nice $S$ possesses an open cover by subschemes admitting dimension functions; see Corollary XIV.2.2.4  of \cite{illgabb}. This statement may be used to enhance slightly Theorem \ref{tweib}(I.\ref{3.3.1.I.2}) below, and to extend the results of \cite{bondegl}. It may also make sense to combine the results of this section with usage of  somewhat more general ("generalized") dimension functions described in Definition 4.2.1  of \cite{bonspkar}.


\item\label{idimfopt} Actually, in the most of the formulations below taking $B=X$  seems to be ``optimal''.

\end{enumerate}

\end{rema}


\subsection{On the ``motivic Gabber's lemma'' and its applications}\label{sgabblem}

The arguments of this section are mostly based on \S6.2 of \cite{cdet}; yet our ``basic resolution of singularities diagram'' is the one used in \cite{kellykth}. 
Another distinction from the arguments is that the usage of Verdier localizations certainly does not ``give control'' over those envelopes that are not shift-stable; so we apply Proposition \ref{pbsnull} instead. 

Throughout the subsection we will assume that all the schemes we consider are of finite type over some fixed (nice) $B$. 

\begin{defi}\label{dnicepair}
\begin{enumerate}

Let $X$ be 
finite type separated $B$-scheme, $d= \de( X)$, $r\ge 0$.

\item\label{idp} 

Given a closed 
embedding  $i_Z\colon Z\to W$ of  schemes of finite type over $X$ and an open $U\subset X$ we consider
the following commutative diagram:
$$
\begin{tikzcd}
Z\arrow{r}{i_Z}\arrow{dr}[swap]{\pi} & W\arrow{d}{a} &
U_W\arrow{l}[swap]{j_W}\arrow{d}{a_U} \\
& X & U\arrow{l}{j};
\end{tikzcd}
$$
here the right-hand square is Cartesian.

We define the following object in $\dk(X)$:
$$
 \ph_X^U(Z\to W)= \pi_* i_Z^* j_{W,*}(\oo_{U_W})
=\pi_* i_Z^* j_{W,*}a_U^*(\oo_U);
$$
we will often omit $X$ and $U$ in this notation. Respectively, when using this notation we will suppose that the 
corresponding $X,Z,W,U$ satisfy the aforementioned conditions. $Q$ will usually denote $X\setminus U$ and $Q_W=W\times _X Q$; denote the embeddings $Q\to X$ and  $Q_W\to W$ by  $i$ and $i_W$, respectively, and denote the composition $\al\circ i_W$ by $\pi_Q$.

In the case $W=X,\ Z=Q$ we will often replace the notation $\ph_X^U(Q\to X)$ just by $\ph(X)=\ph_X^U(X)$. 

\item\label{idr} 
We define the following classes of objects in $\dkx$: $N_r(X)=N^{B}_r(X)$ is the $\dkx$-envelope of $\{\mgbmx(V)[m]\}$ for $V$ running through regular finite type $X$-schemes, $m\ge 0$,  and $\de(V)+m\le r$; $VN_r(X)=VN^{B}_r(X)\subset N_r(X)$ is the envelope of $\{\mgbmx(V)[m]\}$ with $(V,m)$ satisfying all the conditions for $N_r(X)$  along with $\de(V)<r$.

\item\label{idre}
 We will say that 
$j$ is {\it $r$-reasonable} if the corresponding object
$\ph_{X}^{U}(X)$ belongs to $ N_r(X)$.
We will say that $j$ is just {\it reasonable} if it is $d$-reasonable; then  we will call $(U,X)$ a {\it reasonable pair}.

More generally, for $c>0$ we will say that $(U,X)$ is {\it $c$-coreasonable} if there exists $T\subset Q$ 
such that $\dim X-\dim T\ge c$ and 
the pair 
$(U,X\setminus T)$ is reasonable.

\item\label{idrs}
 We will say that a reduced 
scheme $X'$ is {\it reasonable} 
 if 
$(V',X')$ is reasonable for any open dense regular  $V'\subset X'$.
\end{enumerate}

\end{defi}

\begin{rema}\label{rreas}
\begin{enumerate}

\item\label{ivn}
We certainly have $N_{r-1}(X)[1]\subset VN_r(X)\subset N_r(X)$ (assuming that $N_{r-1}(X)$ is empty for $r=0$).

\item\label{ifuni}
Certainly, any element of $N_r(X)$ (or of $VN_r(X)$) belongs to the envelope of  $\{\mgbmx(V_b)[m_b]\}$
for some finite set of $(V_b,m_b)$ as above. 
 Yet our arguments yield some more information on possible $(V_b,m_b)$ corresponding to a reasonable embedding (see part \ref{idre} of the definition). 
We start  explaining this from certain easy and useful 
arguments; see the continuation in Remark \ref{rweib}(\ref{iextest}) below.

So, let $f\colon Y\to X$ and $g\colon X\to S$ be (separated) finite type morphisms of $B$-schemes. Then $g_!$ certainly sends $VN_r(X)\subset N_r(X)$
into $VN_r(S)\subset N_r(S)$.

Next, for a regular finite type $V/X$ we have $f^*(\mgbmx({V})[m])\cong \mgbm_Y({V\times_X Y})[m]$. Thus for $V_b$ being the components of some regular stratification of  $(V\times_X Y)\re$  the object  $f^*(\mgbmx({V})[m])$ belongs to  the $\dk(Y)$-envelope of $\{\mgbm_Y(V_b)[m]\}$ (see Remark \ref{rcons}(\ref{imgbme}).
Hence $f^*(VN_r(X)) \subset VN_r(Y)$ and  $f^*(N_r(X))\subset N_r(Y)$ whenever $f$ is quasi-finite.

\item\label{iproj} Hence one may assume that all $V_b$ corresponding to part \ref{idre} of the definition are $Q$-schemes (since  $i_!i^*(\ph_X^U(X))\cong \ph_X^U(X)$). 

\end{enumerate}
\end{rema}

Now we prove a collection of properties of our notions. 

\begin{lem}\label{lnice}


 Adopt the notation of Definition \ref{dnicepair} (so, $d= \de( X)$).

\begin{enumerate}
 \item\label{idtriangle}
 There is a  distinguished triangle 
$j_!(\oo_U)\stackrel{h}{\to} j_*(\oo_U)\to \ph^U_X(X)$. 

Thus if we assume that $\ph^U_X(X)$ belongs to the envelope of 
$ \{\mgbmx(V_b)[m_b]\}$ for some finite type morphisms $v_b\colon V_b\to X$ and $m_b\in \z$, then  $j_*(\oo_U)$ belongs to the envelope of $\{\mgbmx(U)\}\cup  \{\mgbmx(V_b)[m_b+1]\}  $ and  $\mgbmx(V)$ belongs to the envelope of $\{j_*(\oo_U)\}\cup  \{\mgbmx(V_b)[m_b]\}$.

\item\label{idsmfunct}
Let $g\colon X'\to X$ be a smooth morphism; 
 denote by $W',U',Z'$ the   pullbacks of the corresponding schemes with respect to $g$. 
Then $g^*(\ph_X^U(Z\to W)) \cong \ph_{X'}^{U'}(Z'\to W')$. 

\item\label{idfunct} $\ph_X^U(Z\to W)$  is
contravariantly functorial in $(Z,W)$, i.e., any  commutative square $C$
\begin{equation}\label{efunct} 
 \begin{CD}
 Z'@>{}>>W'\\
@VV{k_Z}V@VV{k}V \\
Z@>{}>>W
\end{CD}\end{equation}
of finite type $X$-schemes yields a 
 morphism $\ph(C)\colon\ph(Z\to W) \to \ph(Z'\to W')$, and this construction respects compositions of morphisms of pairs.

\item\label{idisom} If the morphism $k$ in (\ref{efunct}) is \'etale and $k_Z$ is an isomorphism then $\ph(C)$ 
is an isomorphism.

\item\label{idtr} Let $Z'$ be a closed subscheme of $Z$; denote the embedding of $\tilde{Z}=Z\setminus Z'$ into $ W$ 
by $\tilde{i}$. 
 Then there is a distinguished triangle $D\to \ph(Z\to W)\to \ph(Z'\to W)\to D[1]$ for $D=a_*\tilde{i}_! \tilde{i}^*j_{W,*}(\oo_{U_W})$,
where the second arrow is the one coming from  assertion
\ref{idfunct}.
Moreover,    
$D$ is isomorphic to $a_*\tilde{j}_!(\ph_{W\setminus Z'}^{U_W}(\tilde{Z}\to W\setminus Z'))$ for 
$\tilde{j}$ being the embedding $W\setminus Z'\to W$ whenever $U_W$ is disjoint from $Z$, whereas  $D\cong (a\circ \tilde{i})_* (\oo_{\tilde{Z}})$ whenever $\tilde{Z}\subset U_W$. 

\item\label{idlem} Assume that $W=X$ and let $m\colon M\to X$ be an embedding.
Denote by $\overline{M}$ the closure of $M$ in $X$,
and let $M\xrightarrow{\widetilde{m}} \overline{M}
\xrightarrow{\overline{m}} X$ be the corresponding factorization.
Then $E=\pi_*\pi^*m_*(\oo_M)$ is isomorphic to
 $\overline{m}_!(\ph_{\overline{M}}^M 
(Z\cap \overline{M}\to \overline{M}))$. 
In particular,
$E\cong \overline{m}_!\widetilde{m}_* (\oo_M)={m}_* (\oo_M)$
whenever $M\subset Z$.


\item\label{idglu2}
Let $T$ be a closed subscheme of $Q$; 
denote its embedding into $X$ by $t$.
Then $\ph_X^U(T\to X)\cong t_{!}t^* (\ph_X^U(X))$.
Moreover, for $o\colon O\to X$ being the embedding complementary to $t$ there is a distinguished triangle
\begin{equation}\label{edglu2}
o_{!}(\ph_O^{U\cap O}(Q\setminus T\to O))\to \ph_X^U(X)\to \ph_X^U(T\to X)
\end{equation}

\item\label{iaddefect}
Assume that $U=\sqcup U_b$; denote by $c_b$ the embeddings of the $X$-closures $\overline{U_b}$ into $X$, $Q_b=\overline{U_b}\setminus U_b$. 
Then $\ph_{X}^{U}(X)\cong \bigoplus_b c_{b,!}(\ph_{\overline{U_b}}^{U_b}(Q_b\to \overline{U_b})) $.

\item\label{idsplit} 
Assume that $a$ is proper, $U$ and $U_W$ are regular,  $\mathcal{O}_{U_W}$ is a free   $\mathcal{O}_{U}$-module of 
dimension $e$. 
Then the morphism 
$\ph(X)\to\ph(Q_W\to W)$ 
(coming from the corresponding commutative diagram via assertion~\ref{idfunct})
 is split injective (i.e., is a coretraction) in the localized category $\dkx[e\ob]$. 


\item\label{idcover} Let $O_b$ be a Zariski cover of $X$. Then $j$ is $r$-reasonable whenever all  the embeddings $j\times_X O_b$ are.

\item\label{idreg}
Assume that $j$ is dense, $W$ is regular, $\de(W)=r$,  
  $\pi$ is proper, $Z\subset Q_W$, 
 and all closed
nowhere dense reduced subschemes of $W$ are reasonable. 
Then $\ph_X^U(Z\to X)\in VN_r(X)$. 

\item\label{indep} If $j$ is reasonable,  $U$ is regular and dense in $X$, then $X$ is reasonable.  


\item\label{idshift} For any $s> 0$ we have  $N_{r-s}(X)[-s]\subset VN_{r-s+1}(X)[-s]\subset   N_r(X)$.

\end{enumerate}

\end{lem}
\begin{proof}

\ref{idtriangle}. By Theorem~\ref{tcd}(I.\ref{iglu}) we have a
distinguished triangle
$$
j_!j^*j_*(\oo_U) \to j_*(\oo_U) \to i_*i^*j_*(\oo_U),
$$
and there is a canonical isomorphism $j^*j_*(\oo_U)\cong\oo_U$.
It remains to note that $\ph^U_X(X) = i_*i^*j_*(\oo_U)$ by definition.

\ref{idsmfunct}. Easy, using smooth base change
(Theorem~\ref{tcd}(I.\ref{iexch}, I.\ref{ipur})).

\ref{idfunct}.
We have a commutative diagram
$$
\begin{tikzcd}
 Z'\arrow{r}{i'_Z}\arrow{d}[swap]{k_Z} & W'\arrow{d}{k}
& U'_W\arrow{l}[swap]{j'_W}\arrow{d}{k_U}
 \\
Z\arrow{r}{i_Z}\arrow{dr}[swap]{\pi} & W\arrow{d}{a} &
U_W\arrow{l}[swap]{j_W}\arrow{d}{a_U} \\
& X & U\arrow{l}{j}
\end{tikzcd}
$$
where the right-hand squares are pullbacks.
Let's start with a unit morphism of the adjunction
$k_U^*\dashv k_{U,*}$:
$$
\oo_{U_W} \to k_{U,*}k_U^*(\oo_{U_W}) = k_{U,*}(\oo_{U'_W}).
$$
Applying $j_{W,*}$, we get a morphism
$$
j_{W,*}(\oo_{U_W}) \to j_{W,*}k_{U,*}(\oo_{U'_W}) \cong k_*j'_{W,*}
(\oo_{U'_W}).
$$
Using the adjunction $k^*\dashv k_*$, we get a morphism
$$
k^* j_{W,*}(\oo_{U_W}) \to j'_{W,*}(\oo_{U'_W}).
$$
Applying ${i'_Z}^*$, we get
$$
k_Z^*i_Z^*j_{W,*}(\oo_{U_W}) \cong
{i'_Z}^*k^* j_{W,*}(\oo_{U_W})
\to {i'_Z}^*j'_{W,*}(\oo_{U'_W}).
$$
Using the adjunction $k_Z^*\dashv k_{Z,*}$, we get a morphism
$$
i_Z^*j_{W,*}(\oo_{U_W}) \to k_{Z,*}{i'_Z}^*j'_{W,*}(\oo_{U'_W}).
$$
Finally, applying $\pi_*$, we get
$$
\ph(Z\to W) = \pi_*i_Z^*j_{W,*}(\oo_{U_W}) \to
(\pi\circ k_Z)_*{i'_Z}^*j'_{W,*}(\oo_{U'_W}) = \ph(Z'\to W').
$$
It is clear that this construction respects composition.

 \ref{idisom}. This is  the $K$-motivic version of  Lemma 6.2.11 of \cite{cdet}, and its proof carries over to our setting without any difficulty.
Indeed, note that (in the notation of the proof of part~\ref{idfunct}) 
in our setting the morphism
$$
k^* j_{W,*}(\oo_{U_W}) \to j'_{W,*}(\oo_{U'_W})
$$
is an isomorphism.
Then we apply ${i'_Z}^*$ to it, use the adjunction $k_Z^*\dashv k_{Z,*}$,
and apply $\pi_*$; thus the result is again an isomorphism.

\ref{idtr}.
Note that (in the setting of the proof of part \ref{idfunct})
the morphisms $k$ and $k_U$ are identities.
Therefore the morphism $\ph(Z\to W) \to \ph(Z'\to W)$
is obtained by applying $\pi_*$ to the unit morphism
$M\to k_{Z,*} k_Z^* M$, where $k_Z\colon Z'\to Z$ is the closed
embedding, and $M = i_Z^* j_{W,*}\oo_{U_W}$.
Now, the gluing datum of Theorem~\ref{tcd} (I.\ref{iglu}) gives 
a distinguished triangle
$$
\widetilde{k}_{Z,!}\widetilde{k}_Z^* M \to M \to k_{Z,*} k_Z^* M,
$$
where $\widetilde{k}_Z\colon \widetilde{Z}\to Z$ is the open
embedding.
After applying $\pi_*$, we get the desired triangle.
The ``moreover'' part of the assertion easily follows from the fact that
$\widetilde{i}^*j_{W,!}=0$ (combined with assertion
\ref{idtriangle}). 

\ref{idlem}. 
We have 
$m_*(\oo_M)\cong \overline{m}_! \widetilde{m}_*(\oo_M)$. Consider the
 corresponding Cartesian square in the category of reduced schemes 
 $$\begin{CD}
 Z\cap  \overline{M}@>{\overline{m}_Z}>>Z\\
@VV{\pi_{\overline{M}}}V@VV{\pi}V \\
\overline{M}@>{\overline{m}}>>X
\end{CD}$$
Then Theorem \ref{tcd}(I.\ref{iexch}, II.\ref{ianr}, II.\ref{itr})
yields
$\pi^*\overline{m}_!\cong \overline{m}_{Z,!} \pi_{\overline{M}}^*$.
Hence
$$
E\cong \pi_*\overline{m}_{Z,!}  \pi_{\overline{M}}^* \widetilde{m}_*
(\oo_M) =\overline{m}_!(\ph_{\overline{M}}^M
(Z\cap \overline{M}\to \overline{M})).
$$
Lastly, if $M\subset Z$, then $Z\cap  \overline{M}=\overline{M}$
and
$\pi_{\overline{M}}=\id_{\overline{M}}$; therefore
$\ph^M_{\overline{M}}(Z\cap\overline{M}\to\overline{M})
= \widetilde{m}_*(\oo_{M})$.
It follows that
$E\cong \overline{m}_!\widetilde{m}_* (\oo_M)={m}_* (\oo_M)$.

\ref{idglu2}. 
The first part of the assertion  is immediate from $t_{*}t^*\cong t_{!}t^*i_*i^*$.

Now we consider the 
distinguished triangle 
$ o_{!}o^*(\ph_X^U(X)) \to \ph_X^U(X)\to t_{!}t^*(\ph_X^U(X))$ coming from (\ref{eglu1}).
 It yields (\ref{edglu2}) since  $o^*(\ph_X^U(X))\cong \ph_O^{U\cap O}(Q\setminus T\to O)$ according to assertion \ref{idsmfunct}.

\ref{iaddefect}. Easy; note that $j_*(\oo_U)\cong \bigoplus_b j_{b,*}(\oo_{U_b})$ (for the corresponding embeddings $j_b\colon U_b\to X$). 

\ref{idsplit}.
The argument is similar to the proof of Lemma 6.2.12 of 
ibid. We have a commutative diagram
$$
\begin{tikzcd}
 Q_W\arrow{r}{i_W}\arrow{d}[swap]{a_Q} & W\arrow{d}{a}
& U_W\arrow{l}[swap]{j_W}\arrow{d}{a_U}
 \\
Q\arrow{r}{i} & X &
U\arrow{l}[swap]{j}, \\
\end{tikzcd}
$$
where both squares are pullbacks.
Note that $\ph(X) = \ph(Q\to X) = i_* i^* j_*(\oo_U)$ and
\begin{align*}
\ph(Q_W\to W) &= (a\circ i_W)_* i_W^* j_{W,*} a_U^*(\oo_U) \\
&= (i\circ a_Q)_* i_W^* j_{W,*} a_U^*(\oo_U) \\
&= i_* i^* a_* j_{W,*} a_U^*(\oo_U) \\
&= i_* i^* j_* a_{U,*} a_U^*(\oo_U).
\end{align*}
Here we used the proper base change formula (Theorem \ref{tcd}(I.\ref{iexch}, I.\ref{ipur})).
It is easy to see that the morphism $\ph(Q\to X) \to \ph(Q_W\to W)$ constructed in part~\ref{idfunct}
is obtained by applying $i_* i^* j_*$ to the unit morphism $\oo_U\to a_{U,*}a_U^*(\oo(U))$.
Now we can apply Theorem \ref{tcd}(I.\ref{itre}).

\ref{idcover}. We should check the reasonability of $(V,X)$ for any open dense regular $V\subset X$.
 Choose a stratification $\al$ of $X$ such that any of $X_l^\al$ lies inside one of the $O_b$. 
Now, according to Remark \ref{rcons}(\ref{imgbme}) it suffices to verify that  
$j_{l,!}j_l^*( \ph_X^V(X\setminus V\to X))\in VN_r(X)$ for all $l$. Hence it remains to combine assertion \ref{idsmfunct} with Remark \ref{rreas}(\ref{ifuni}).

\ref{idreg}.  Remark \ref{rreas}(\ref{iproj}, \ref{ifuni}) yields that it suffices to verify the assertion with $X$ replaced by $W$; so we assume  $W=X$.

Choose a regular stratification $\al$ 
of $X$ such that
$X_1^\al=U$ and any of $X_l^\al$ for $l>1$  either lies in $Z$ or is disjoint from  $Z$. Then $\de(X_l^\al)<d$   for  $l\ge 2$ (see Proposition \ref{pdimf}), and $j_{*}(\oo_{U})$ belongs to the envelope of   $\{\oo_X\}\cup\{j_{l,*}(\oo_{X_l^\al})[1]:\ l>1\}$ (see Remark \ref{rcons}(\ref{imgbmed}). 
Thus it suffices to verify that the objects $\pi_*\pi^*(\oo_X)=\mgbmx(Z)$ and $\pi_*\pi^*j_{l,*}(\oo_{X_l^\al}[1])$ for $l>1$  belong to $VN_r(X)$. 

Now,
$\mgbmx(Z)\in N_{r-1}(X)\subset VN_r(X)$ immediately from Remark \ref{rcons}(\ref{imgbme}). 

Next we apply assertion \ref{idlem}. So we present all $j_l$ as the compositions $X_l^\al \stackrel{\overline{j}_l}{\to}\overline{X_l^\al}\stackrel{i_l}{\to} X$, where $ \overline{X_l^\al}$
are the closures of $X_l^\al$ in $X$. Recall that the embeddings  $\overline{j}_l$ are reasonable according to our assumptions.

If $X_l^\al$ (for $l>1$) lies in $Z$ then $\pi_*\pi^*j_{l,*}(\oo_{X_l^\al}[1]) 
\cong i_{l,!}\overline{j}_{l,*}(\oo_{X_l^\al}[1])$ according to (the ``in particular'' part of) assertion \ref{idlem}. 
Now, $\overline{j}_{l,*}(\oo_{X_l^\al})\in N_{r-1} (\overline{X_l^\al})$  according to (the ``thus'' part of) assertion \ref{idtriangle}; hence it remains to apply Remark  \ref{rreas}(\ref{ifuni}, \ref{ivn}). 
 
Lastly, assume that $X_l^\al$ is disjoint from $Z$ for some $l>1$; denote $Z\cap \overline{X_l^\al}$ by $Z_l$. Then $\pi_*\pi^*j_{l,*}(\oo_{X_l^\al}[1])\cong i_{l,!}\ph_{\overline{X_l^\al}}^{X_l^\al}(Z_l\to \overline{X_l^\al})[1]$ according to assertion \ref{idlem}. 

Since $\overline{j}_l$  is reasonable, we have $\ph_{\overline{X_l^\al}}^{X_l^\al}(\widetilde{X_l^\al}\to \overline{X_l^\al})\in  VN_{r-1} (\overline{X_l^\al})$, where $\widetilde{X_l^\al}=\overline{X_l^\al}\setminus X_l^\al$.
Since $Z_l$ is a closed subscheme of $\widetilde{X_l^\al}$,  it remains to combine (the first part of) assertion \ref{idglu2} with Remark  \ref{rreas}(\ref{ifuni}, \ref{ivn}).

\ref{indep}. It certainly suffices to prove the following (for any open dense regular $U\subset X$): if $U$ is regular and $j_1\colon U_1\to U$ is an open dense embedding then $j$ is reasonable if and only if $j_1'=j\circ j_1$ is. 

Now we argue somewhat similarly to the proof of assertion \ref{idreg}. Choose a regular stratification $\al$ of $U$ such that $U_1^\al=U_1$. Certainly, 
$\de(U_l^\al)<d$ for $l>1$. Denote $j\circ j_l$ by $j_l'$ for all $l$. Then $j_*(\oo_U)$ belongs to the envelope of $\{j'_{l,*}(\oo_{U_l})\}$
and $j'_{l,*}(\oo_{U_l}) $ belongs to the envelope of $\{j_*(\oo_U)\}\cup \{j'_{l,*}(\oo_{U_l}):\ l\ge 2\}[1]$. Hence the arguments used in the proof of assertion \ref{idreg} can easily be applied to conclude the proof.

\ref{idshift}. 
These inclusions easily follow the fact that $\mgbmx(V)[-s]$ is a retract of   $\mgbmx(\mathbb G_m^s(V))$ and $\de(\mathbb G_m^s(V))=\de(V)+s$ (where $\mathbb G_m^s$ is the $s$th Cartesian power of the multiplicative group scheme). 

\end{proof}

Now we are able to prove the central (technical) statement of this section.

\begin{pr}\label{pmotgablem}\label{pgab}

Any reduced $B$-scheme is reasonable. 

\end{pr}
\begin{proof}

We will verify the reasonability of any (finite type) reduced $B$-scheme $X$ of $\de$-dimension at most $d$ for some $d\ge 0$.  
Note that the statement is trivial for $d=0$ (since then $\dim(X)=0$ according to Proposition \ref{pdimf}). 

Moreover, a finite type (separated) $X/B$, $\de(X)=d$, is reasonable whenever it is $d+1$-coreasonable.
Thus 
 by an obvious  induction on $(c,d)$ our proposition  reduces to the following claim   (somewhat similarly to \S6.2 of \cite{cdet} and \S XIII.3 of \cite{illgabb}):  
$X$ is $c+1$-coreasonable for some integer $c$, $0\le c\le d$,    provided that any 
finite type reduced $B$-scheme  is reasonable if its $\de$-dimension is less than $d$ and is $c$-coreasonable if its $\de$-dimension equals $d$.

Now,  it suffices to verify our claim in the case where $X$ is irreducible. 
Indeed, let $X=\cup X_k$, where $X_k$ are irreducible components of $X$.
If all $X_k$ are $c+1$-coreasonable
 then one can choose  closed $T_k\subset X_k$ of $\de$-dimension less than $d-c$ and non-empty open regular $U_k\subset X_k\setminus T_k$ 
such that all the pairs $(U_k,X_k\setminus T_k)$ are reasonable and $U_k$ are pairwise disjoint (in $X$). 
Then $(\sqcup U_k,X\setminus (\cup_k T_k))$ is easily seem to be reasonable. Thus $X$ is $c+1$-coreasonable according to Lemma \ref{lnice}(\ref{indep}).

Now we prove our claim for an irreducible $X$ using Proposition \ref{pbsnull}. 
So, we need a small subcategory of $\dkx$ containing all ``relevant objects''.
We take $\cu$ being a small full triangulated 
 subcategory of $\dkx$ containing all ``geometrically defined'' objects of this category; those are  obtained  from objects of the form $\oo_D$ for $D$ being separated of finite type over $X$ by applying compositions of motivic image functors coming from finite type separated $X$-morphisms 
  (cf. Remark \ref{rconstrgabber} below).
	We fix a homological functor $F\colon\cu \to R-\modd$ whose restriction to 
	$VN_d(X)$ is zero; here we take $R=\q$ if $\lam=\q$ and $R=\zll$ for some  $l\in \p\setminus\sss$ in the opposite case. 
	We should check that 
	$F(\ph_X^U(T\to X))=\ns$.

	 Following 
\S2 of \cite{kellykth},  in the case $\lam\neq \q$ we apply \cite[Theorems IX.1.1 and 
 Theorem II.4.3.2]{illgabb} 
to obtain a commutative diagram 
\begin{equation}\label{ekg} 
  \begin{CD}
 W'@>{k}>>X'\\
@VV{q}V@VV{p}V \\
W@>{r}>>X
\end{CD}\end{equation}
whose connecting morphisms are separated of finite type, 
$p$ is proper and generically finite of degree prime to $l$, $k$ is a Nisnevich cover, and $W'$ is regular.
Certainly, the schemes $X'$ and $W'$ are
generically finite over $X$; this is also true for $W$ (see Definition II.1.2.2 of ibid.). Thus $X,W,$ and $W'$ are of $\delta$-dimension at most $d$ according to Proposition \ref{pdimf}. 
In the case $\lam=\q$ we replace the usage of Theorem II.4.3.2 of \cite{illgabb} by that of Theorem II.4.3.1 of loc. cit.; we obtain the same data with the only difference being that the generic degree of $p$ may be arbitrary. 

We fix some open dense regular $U\subset X$ such that $U'=U\times_X X'$ is regular and $\mathcal{O}_{U'}$ is a free   $\mathcal{O}_{U}$-module.  
According to Lemma \ref{lnice}(\ref{indep}), it suffices to verify that the embedding $j\colon U\to X$ is $c+1$-coreasonable.

Now we apply the inductive assumption (as described above) to $X$ and $X'$.  We obtain the existence of  closed $T\subset X$ and $T'\subset X'$ of $\de$-dimension at most $d-c$ such that  $X\setminus T$ and $X'\setminus T'$  are reasonable. Combining this fact 
with Lemma \ref{lnice}(\ref{idcover}, \ref{idfunct}) we obtain the following: it suffices to verify the existence of an open reasonable neighborhood $V$ of any generic point $t$ of $T$ in $X$. We will check this fact for some fixed $t$ of this sort; for this purpose
we will prove the existence of $V$ such that $\ph_V^{U\cap V}(Q\cap V\to V)\in VN_d(V)$.

Now we can ``modify''  (\ref{ekg}) similarly to  \S6.2 of \cite{cdet}; note that the (easy) method of the proof of Lemma 4.2.14 of \cite{cd} justifies this action without any difficulty. Since we can replace $X$ by any open neighborhood $V$ of  $t$ 
(replacing all other  schemes in (\ref{ekg}) by the corresponding pullbacks), we can assume that $T'$ is finite over $T$ and that the fiber of $k$ over $T'$ is an isomorphism (recall that $k$ is a Nisnevich cover!).
Now, $V\cap (X\setminus T)$ is reasonable 
according to part \ref{idsmfunct} of the lemma.
 Combining this fact with  
the triangle (\ref{edglu2}), part \ref{indep} of the lemma, and Remark \ref{rreas}(\ref{ifuni}) we reduce the statement in question to
$\ph_X^U(T\to U)\in VN_d(X)$.

Let $T_{W'}$ denote $T'\times_{X'} W'$, and denote the image $q(T_{W'})$ 
 (noting that the restriction of $q$ to $T_W'$ is certainly finite) by $T_{W}$.
We should verify the vanishing of $F(\ph_X^U(T\to X))$.
Now, Lemma \ref{lnice}(\ref{idfunct}) provides us with a chain of morphisms
$$
\ph(T\to X)\to \ph(p\ob(T)\to X')\to \ph(T'\to X')\to \ph(T_{W'}\to W').
$$

The results of applying $F$ to these morphisms are injective according to the lemma. Being more precise,  the induced homomorphism from $F(\ph(T\to X))$  into $F( \ph(p\ob(T)\to X'))$ is split injective according to part \ref{idsplit} of the lemma. Next, $F(\ph(p\ob(T)\to X'))$ injects into $F (\ph(T'\to X'))$ according to part \ref{idtriangle} of the lemma; here we use the fact that $F(D)=0$ since $D\in VN_d(X')$  according to the inductive assumption combined with Remark \ref{rreas}(\ref{ivn}, \ref{ifuni}). 
Lastly,  $F(\ph(T'\to X')) $ maps isomorphically onto $F( \ph(T_{W'}\to W'))$ according to part \ref{idisom} of the lemma.

 It remains to note that the composition map $F(\ph(T\to X))\to F(\ph(T_{W'}\to W'))$
  factors through  $F(\ph(T_{W}\to W))$, whereas  the latter group is zero 
according to part \ref{idreg} of the lemma combined with the inductive assumption. 

\end{proof}

This statement has several nice implications. 

\begin{theo}\label{tgabber}

  Let $f\colon Y\to X$ be a finite type (separated) morphism for $X$ being a separated finite type $B$-scheme,
	and denote $\de(Y)$ by $d$.
	Then the following statements are valid.

\begin{enumerate}
\item\label{tgabber-1} If  $Y$ is  regular then there exists a morphism $\mgbm_{X}(Y)\to f_{*}(\oo_Y)$ whose cone belongs to $VN_r(X)$.

\item\label{tgabber-2} For any $r\ge 0$. $n\in \z$, $-r-1\le n\le r$, $N_r(X)$ coincides with the envelope $N_r^n(X)$ of $\{\mgbmx(V)[m]:\ m\le n\}\cup \{v_*(\oo_V)[m]:\ m> n\}$ for $v\colon V\to X$ running through finite type (separated) morphisms with regular domain with $\de(V)+|m|\le r$.

\item\label{tgabber-3} $\mgbmx(Y)$ 
 and $f_*(\oo_Y)$ belong to $N_d(X)$.

\item\label{tgabber-4} We have $f_*(N_r(Y))\subset N_r(X)$, whereas $f^!(N_r(X))\subset  N_r(Y)$ whenever $f$ is quasi-finite. 
\end{enumerate}
\end{theo}

\begin{proof}
1.  Recall that 
$Y$ possesses an $X$-compactification $\overline{Y}$ (see Remark \ref{rdimf}(\ref{idcomp})); denote the corresponding proper morphism $\overline{Y}\to X$ by $\overline{f}$. 
Then Lemma \ref{lnice}(\ref{idtriangle}) gives a distinguished triangle $\mgbm_{X}(Y)\to f_{*}(\oo_Y)\to \overline{f}_!(\ph_{\overline{Y}}^Y(Y)) $; 
 hence the statement follows from Proposition \ref{pgab}.

2. For $r=0$ we have  $\de(V)=0$ in the definitions of $N_r(X)$ and $N_r^n(X)$; thus $V$ is proper over $X$. 
Hence the obvious inductive argument enables us to assume that  $N_{r'}(X)=N_{r'}^{n'}(X)$ for any $r'<r$ and  
 $-r'-1\le n'\le r'$.

Next, it certainly suffices to verify that $N_r(X)=N^r_r(X)$ and that  the classes $N_r^n(X)$ coincide for all $n$.

We certainly have $N_r(X)\subset N_r^r(X)$. The converse implication is immediate from Lemma \ref{lnice}(\ref{idshift}).


Thus is remains to prove by descending induction on $n$ 
 that $N_r^{n-1}(X)=N_r^{n}(X)$ for any 
$n$ between $r$ and $-r$. This statement is also equivalent to $\mgbmx(V)[n]\in N_r^{n-1}(X)$ and
$v_*(\oo_V)[n]\in N_r^{n}(X)$ 
whenever $V$ is a regular scheme of finite type over $X$ and $|n|+\de(V)\le r$.

Now, assertion 1 yields that $v_*(\oo_V)[n]$ belongs to the $\dkx$-envelope of $\{\mgbmx(V)[n]\}\cup VN_{\de(V)}[n]$. Hence to verify that $v_*(\oo_V)[n]\in N_r^{n}(X)$  it suffices to check whether 
$VN_{\de(V)}[n]\subset  N_r^{n}(X)$.  
Next, the inductive hypothesis (with respect to $n$) yields that $N_r^{n}(X)=N_r(X)$, whereas the latter class contains $VN_{\de(V)}[n]$ (if $\de(V)+|n|\le r$);  here we apply Remark \ref{rreas}(\ref{ivn}) for $n\ge 0$ and  Lemma \ref{lnice}(\ref{idshift}) for $n<0$.

Now we  verify that $\mgbmx(V)[n]\in  N_r^{n-1}(X)$ 
The statement is obvious if $\de(V)=0$; see Remark \ref{rdimf}(\ref{indepdimf}). 

Now we  check our statement for $\de(V)>0$; so, $|n|<r$.  Assume that $n>0$; then $\mgbmx(V)[n-1]\in N^{r-1}(X)$.  Applying the inductive hypothesis with respect to $r$ we obtain  that $\mgbmx(V)[n-1]\in N^{r-1}_{n-2}(X)$. Hence in this case it suffices to note that $N^{r-1}_{n-2}(X)[1]\subset  N^r_{n-1}(X)$. 
Similarly, if  $n>0$ then $\mgbmx(V)[n+1]\in N^{r-1}(X)$; thus $\mgbmx(V)[n]\in  N^{r-1}_{n}(X)[-1]\subset  N^r_{n-1}(X)$. 

It remains to consider the case $n=0$. Applying  assertion 1 once again, we note that $\mgbmx(V)$ belongs to the $\dkx$-envelope of $\{v_*(\oo_V)\}\cup VN^{\de(V)}[-1]$. Thus it suffices to check that $VN^{r}[-1]\subset N^r_{-1}(X)$. Now, the class $VN^{r}[-1]$ equals the envelope of  $N^{r-1}(X)\cup \{\mgbmx(V)[-1]:\ \de(V)<r\}$ by definition. The latter envelope certainly lies in   $N^r_{-1}(X)$.

3. The first part of the assertion is an immediate consequence of Remark \ref{rcons}(\ref{imgbme}). Applying it for $X=Y$ we obtain that $\oo_Y\in N^d(Y)$. Hence 
$\oo_Y\in N^d_{-d-1}(Y)$ (see assertion 2). It obviously follows that $f_*(\oo_Y)\in N^d_{-d-1}(X)$, and is remains to apply assertion 2 once again.

4. Similarly to the argument above, we use the equalities $N^r(Y)=N^r_{-r-1}(Y)$ and $N^r(X)=N^r_{-r-1}(X)$. They imply the first part of the assertion immediately (cf. Remark \ref{rreas}(\ref{ifuni})). Combining these equalities with Remark \ref{rcons}(\ref{imgbmed}) we also obtain the second part of the assertion. 
\end{proof}

\begin{rema}\label{rconstrgabber}
Note that the arguments of this section  
do not rely on Theorem \ref{tcd}(II.\ref{irmotgen}). Thus Theorem \ref{tgabber}(\ref{tgabber-3}) 
can be used as an important ingredient of its proof  (cf. the proof of Corollary 6.2.14 of \cite{cdet}). 
 This was our reason for considering (a small skeleton of) the category of ``geometrically generated $K$-motives'' instead of  $\dkcx$  in the proof of Proposition \ref{pgab}.

\end{rema}


We describe some  obvious consequences from the theorem.

\begin{coro}\label{cgabber}
 Let $f\colon Y\to X$ be a finite type morphism for $X$ being a $\lam$-nice scheme;
let $d$ denote the dimension of an $X$-compactification of $Y$. 
Then $\mgbmx(Y)$ 
 and $f_*(\oo_Y)$ belong to the $\dkx$-envelope 
 of $\{\mgbmx(V)[m]:\ m\le n,\ \dim(V)+|m|\le d\}\cup \{v_*(\oo_V)[m]:\ m> n,\ \dim(V)+|m|\le d\}$, where $V$ runs through all regular finite type $X$-schemes satisfying this inequality. 
\end{coro}
 \begin{proof} Once again,  the symbol $\overline{Y}$ will denote an $X$-compactification of $Y$; it suffices to verify the assertion for $X=\overline{Y}$. 
Next, we take $B$ equal to $\overline{Y}$ also; then for the corresponding $\de$ we certainly have $\de(\overline{Y})=\dim (\overline{Y})=d$ (immediately from the definition of $\de$). Hence the assertion follows from Theorem \ref{tgabber}(\ref{tgabber-2}, \ref{tgabber-3}) combined with the inequality $\de(-)\ge \dim(-)$ (see Proposition \ref{pdimf}).
\end{proof}

\subsection{An application to bounding weights and to the vanishing of ``too negative'' $\kk$-groups}\label{sgenweibelconj}

Now we improve Corollary  \ref{ctestwnc} (that includes Theorem \ref{testw} as a particular case) by combining Proposition \ref{pmotgablem} with Proposition \ref{pbw}(\ref{iortprecise}).

\begin{theo}\label{tweib}
Let $f\colon Y\to X$ be a separated 
morphism of $\lam$-nice reasonable schemes, $\dim Y=d$, $0\le r\le d$.
Then the following statements are valid.

I. 
\begin{enumerate}
\item\label{3.3.1.I.1} $f_*(\oo_Y)\in \dkx_{\wchow\ge -d}$.

\item\label{3.3.1.I.2} There exists a closed subscheme $Z\subset Y$ such that $\dim(Z)\le r-1$ and for the morphism $f'\colon Y\setminus Z\to X$ we have $f'_*(\oo_{Y\setminus Z})\in \dkx_{\wchow\ge r-d}$.

\item\label{3.3.1.I.3} Assume that $f$ is of finite type; let $\overline{Y}$ be an $X$-compactification of $Y$  (see Remark \ref{rdimf}(\ref{idcomp})) and denote its dimension by $\overline{d}$. Then $f_*(\oo_Y)\in \dkx_{\wchow\le \ovd}$ and $\mgbmx(Y)\in \dkx_{\wchow\ge - \ovd}$.


\end{enumerate}

II.
\begin{enumerate}
\item\label{3.3.1.II.1} For any smooth $V/X$ and any open 
embedding $j\colon U\to V$ we have the following: $\kk_{s}(V\times_X Y)=\ns$ for $s<-d$ 
 and  $\kk_{-d}(V\times_X Y)$ surjects onto $\kk_{-d}(U\times_X Y)$. 

\item\label{3.3.1.II.2} In the setting of assertion I.3 
 we have the following: 
for any 
(separated) finite type
 morphism $P\to X$ the groups $\kk_{-s}^{BM,Y}(P\times_X Y)$ (see Remark \ref{rcons}(\ref{imgbmles}))
vanish for all $s>r$ if (and only if) this statement is valid under the additional assumptions that $P$ is regular and affine, $\dim(P)\le \ovd-s$, and $s\le d$. 
\end{enumerate}


\end{theo}

\begin{proof}
I. All of the assertions easily follow the fact that for any $\lam$-reasonable scheme $S$ and $v\colon V\to S$ being a morphism of finite type with a regular domain we have $\mgbms(V)\in \dk(S)_{\wchow\le 0}$ and 
$v_*(\oo_V)\in \dk(S)_{\wchow\ge 0}$; see Proposition \ref{pwchownc}(\ref{iwexchow6}, \ref{iwexchow1}, \ref{iwexchow2}).

\ref{3.3.1.I.1}. Indeed, applying  Corollary \ref{cgabber} to the morphism $\id_Y$, $n=-d-1$ we deduce that $\oo_Y=\mgbm_Y(Y)$ belongs to  $\dk(Y)_{\wchow\ge -d}$. It remains to apply 
Proposition \ref{pwchownc}(\ref{iwexchow1}) once again. 

\ref{3.3.1.I.2}. We apply Theorem \ref{tgabber}(\ref{tgabber-2}, \ref{tgabber-3}) to the morphism   $\id_Y$ obtaining  that $\oo_Y$ belongs to the envelope 
of  $\{v_{m,*}(\oo_{V_m}) [m]:\ -d\le m\le d\}$ with $v_m\colon V_m\to Y$ being finite type morphisms with regular domains and $\de^Y(V_m)\le d-|m|$. We take $Z$ being the scheme-theoretic image in $Y$ of $\cup_{m>d-r}V_m$.  Then $\dim(Z)\le r-1$ according to 
Remark \ref{rdimf}(\ref{idcomp}). Next, for $U=Y\setminus Z$ and the open embedding $u\colon U\to X$ we have $\oo_U\cong u^*(\oo_Y)$; thus $\oo_U$ belongs to the $\dk(U)$-envelope of $\{(v_{m}\times_Y u)_{*}(\oo_{V_m\times_Y U}) [m]:\ -d\le m\le d\}$, 
 whereas $V_m\times_Y U$ is empty for $m>d-r$. Once again, it remains to apply 
Proposition \ref{pwchownc}(\ref{iwexchow6}, \ref{iwexchow2}).

\ref{3.3.1.I.3}. Immediate from Corollary \ref{cgabber} applied to $f$ with 
 $n$ being equal to $\ovd$ and $-\ovd-1$, respectively.

II.
\ref{3.3.1.II.1}. This is just a combination of assertion I.\ref{3.3.1.I.1} with Corollary \ref{ctestwnc}. 

\ref{3.3.1.II.2}. The aforementioned corollary yields that the first condition in the assertion (i.e., the one without the restriction on $\dim(P)$)
is fulfilled if and only if $f_*(\oo_Y)\in \dkx_{\wchow\ge -r}$. 

Now we apply Proposition \ref{pbw}(\ref{iortprecise}) for $M=f_*(\oo_Y)$, $i=-r$. According to Corollary \ref{cgabber}, 
$M$ belongs to the envelope of $\{\mgbmx(V)[m]:\ m\le -r-1,\ \dim(V)-m\le \ovd\}\cup \{v_*(\oo_V)[m]:\ m\ge -r,\ \dim(V)+|m|\le \ovd\}$, where $V$ runs through all regular finite type $X$-schemes satisfying this inequality. 
Hence (combine the recollection in the beginning of the proof of assertion I with Proposition \ref{pbw}(\ref{iortprecise})) to verify that $f_*(\oo_Y)\in \dkx_{\wchow\ge -r}$ it suffices to check that $\mgbmx(V)[m]\perp M$ whenever $V$ is  regular scheme of finite type over $X$,  $\dim(V)-m\le \ovd$, and $m\le -r-1$.
Now, this condition is automatic for $m>d$ according to assertion I.1. Thus is remains to recall that $\dkx(\mgbmx(V)[m],M)\cong \kk_{m}^{BM,Y}(V\times_X Y)$ (see \ref{ekthlesgen}).

\end{proof}

\begin{rema}\label{rweib}
\begin{enumerate}
\item\label{icompkelly}
It appears that no analogues of part  II.\ref{3.3.1.II.1} of the theorem were ever formulated in the literature; yet prof. S. Kelly has kindly informed the authors that the methods of  \cite{kellykth} may be used to prove a somewhat similar statement.

Yet part II.\ref{3.3.1.II.2} of the theorem is completely new. Note that one may certainly replace the Borel--Moore $K$-theory in it by relative $K$-theory and relate it to other conditions described in Theorem \ref{testw}. Furthermore, one can slightly improve the statement by considering $\de^X$-dimensions instead of the Krull ones in it.

\item\label{isupport} Part I.\ref{3.3.1.I.2} of the theorem also appears to be 
 new. In particular, it says that for any $n\ge 0$ the property of $\oo_X$ not to belong to $\dkx_{\wchow\ge -n}$ is ``supported in codimension $>n$''; thus the same is true for the non-vanishing of $\kk_s(X)$ for $s<-n$ (along with the non-vanishing of 
 $\kk_{-s}^{BM,X}(P)$ for a separated finite type $P/X$). These facts are (more or less) trivial for $n=0$; yet for $n>0$ they do not seem to follow from Kelly's results (since the inverse limit of $X\setminus Z$ for $Z$ running through all closed subschemes of codimension greater than $n$ does  not exist as a scheme a general).

\item\label{iextest} 
Now we explain that instead of checking that $\kk_{-s}^{BM,Y}(P\times_X Y)=\ns$ for all $(s,P)$ as in assertion II.\ref{3.3.1.II.2} it suffices to verify it for a certain set of ``test pairs'' $(P_b,s_b)$ that may be explicitly described. 

For this purpose it certainly suffices to verify the corresponding ``explicit version'' of Corollary \ref{cgabber}, i.e., that $f_*(\oo_Y)$ belongs to the envelope of
$\{\mgbmx(V_b)[m_b]:\ m_b\le n\}\cup \{v_*(\oo_{vp_c})[m_c]:\ m> n\}$ for certain ``explicit $(V_b,m_b)$'' (and $V_b,\vp_c$ being regular of finite type over $X$, $\dim(V_b)+|m_b|$ and $\dim(V_c)+|m_c|$  being at most $\ovd$).
Now we note that the proofs of the corresponding parts of Lemma \ref{lnice} are quite ``constructible''. Thus the proof of Proposition \ref{pgab} can be used to construct explicit pairs of a similar sort in the cases $\lam =\zll$ (for some prime $l$) and $\lam=\q$ since the arguments from \cite{illgabb} used in it are constructible also.
It easily follows that $(V_b,m_b)$ can be explicitly described in this case.


 To obtain  an ``explicit version'' of Proposition \ref{pgab} 
when  $\lam$ is distinct from $\q$ and $\zll$  one can start from applying the argument used in the proof of the proposition  for 
$\lam'=\z_{(l_0)}$ for some $l_0\in p\setminus \sss$. 
The (inductive) construction of the corresponding pairs  for the object $\ph_{X}^{U}( X)$ involves a finite number of applications of Lemma \ref{lnice}(\ref{idsplit}); denote the corresponding values of $e$ by $e_i$. Then our arguments actually yield that  the (``explicit'') envelope statement in question for $\ph_{X}^{U}( X)$ is actually fulfilled in the category of motives with $\lam[(\prod e_i)\ob]$-coefficients.

Next one can apply Corollary 0.2 of \cite{bsnull}. It states that for $M,M_j\in\obj \dkcx$ 
the object $M$ belongs to the envelope of $\{M_j\}$ whenever  $c_{\dkcx}^l(M)$ (see Remark \ref{rlocoeff}(\ref{icompenv}))  belongs to  the envelope of  $\{c_{\dkcx}^l(M_j)\}$ for any $l\in \p\setminus \sss$.\footnote{ Recall that target of  $c_{\dkcx}^l$  is isomorphic to the $\zll$-linear version of $\dkcx$.}
Thus for $L$ being the set of prime divisors of $\prod e_i$ belonging to $\p\setminus \sss$ the object
  $\ph_{X}^{U}( X)$
 belongs to the envelope of  $ \{\mgbmx(V_{b,X})^{l_0}[m_b^{l_0}]\}\cup(\cup_{l\in L}\{\mgbmx(V_{b,X}^l[m_b^l]\}$ 
 for the corresponding $V_{b,X}^l$. 

This ``explicit test statement'' is completely new. Note also that in the next section we will relate the Chow-weights of $K$-motives with rational coefficients (and so, also the vanishing of the corresponding $K$-groups) to \'etale (co)homology of schemes.

\item\label{icompkth} The authors suspect that all the aforementioned properties of $\kk(-)$ are valid for the "usual" (i.e., "non-homotopy-invariant") $K$-theory. Unfortunately, the ("motivic") methods of the current paper cannot be used to prove this conjecture (in the general case) since at the moment the "theory of motivic categories" is understood well enough only for motivic categories satisfying the homotopy invariance property (see \S\ref{sgengab} below). Yet applying the "classical" comparison argument one immediately obtains  the  proof of  the restriction of our conjecture (for $\zop$-linear $K$-theory) to nice schemes such that $p$ is nilpotent on them; cf. (the proof of) \cite[Theorem 3.5]{kellykth}. 

\item\label{iopt}
The authors certainly do not claim that this method of constructing $(V_b,m_b)$ is ``optimal''. One of the reasons for this is that one can ``often'' replace the usage of the aforementioned results of \cite{illgabb} by  much simpler (though less general) resolution of singularities results (in particular, the ones established in \cite{tem}).
So one can  find a smaller set of  $(V_b,m_b)$  (at least) in some cases; cf. Proposition \ref{pldh}  below for an ``easy'' formulation of this sort (and  Remark \ref{rvd}(3) for a very simple example).

\item\label{ibiratest} One may also obtain certain ``explicit test schemes'' for bounding below the Chow-weights of  $\co (h_*(\oo_{X''}) \to g_*(\oo_{X'}))$ as mentioned in Remark \ref{rcones}(\ref{rcones-2}); here one should combine the arguments above with Proposition \ref{pbw}(\ref{iwdext}). This corresponds to bounding from below the non-vanishing of certain relative Borel--Moore $K$-theory (and may be unfolded into the statements concerning some ``bi-relative'' homotopy invariant $K$-groups).

\item\label{ikerzs}
It would be interesting to combine (somehow) the arguments of the current paper with the recent work \cite{kerzs}.

\end{enumerate}

\end{rema}

\subsection{Some ``axiomatic'' generalizations} \label{sgengab}
We will now explain that our arguments can be applied to a wide range of ``motivic'' categories.
Our main examples (other than $K$-motives) are Beilinson motives, $cdh$-motives (see Remark \ref{rcons}(\ref{itate})), and $h$-motives as studied in \cite{cdet} (yet cf. Remark \ref{rcompvm}(\ref{rcompvm-1}) below).
We describe a certain axiomatic setup that would enable us to treat all of these settings simultaneously. 

So, for  a scheme $B$ as above we will use the notation $G=G(B)$ to denote the category of 
separated finite type $B$-schemes (with morphisms being separated $B$-scheme morphisms).
We assume  the existence of a function $\de$ from $\obj G$ into non-negative integers that fulfills all the properties listed in Proposition \ref{pdimf}.

 Next, we
  consider a $2$-functor $\md$ from $G$ into the $2$-category of (compactly generated) tensor triangulated categories that are closed with respect to small coproducts.
	We assume that all the categories $\md(X)$ for $X\in \obj G$ are $\lam$-linear, where $\lam=\z[\sss\ob]$ and $\sss$ contains all the primes that are not 
	invertible on $B$.
	All the notation 
	for $\md(-)$
	will be similar to that for $K$-motives (though we will sometimes 
	put an 
	index $\md$ to avoid ambiguity). So, for $v\colon V\to X$ being a $G$-morphism we define $\mgdx(V)=v_!(\ood_V)$.
	We will assume that the $\md$-versions of 
	Theorem \ref{tcd}(I.\ref{ifunast}--\ref{iglu}) 
	are fulfilled. 
	
	Moreover, we assume that for the splitting
	$\mgdx(\p^1(X))\cong \ood_X\bigoplus \ood_X\lan -1\ra $ coming from the natural morphisms $X\to \p^1(X)\to X$ the object  $\ood_X\lan -1\ra $ is $\otimes$-invertible in $\md(X)$ for any $X\in \obj G$, and for any $i\in \z$ the functor $-\lan i \ra =-\otimes (\ood_-)^{\otimes i}$ ``commutes'' with all the functors of the type $f^*$, $f_*$, $f_!$, and $f^!$ for $f$ being a $G$-morphism. 
	Furthermore, in the setting of assertions I.(\ref{ipure}--\ref{ipura}) we assume that $f^!\cong f^*\lan s \ra$ if $f$ is everywhere of relative dimension $s$, and  
	$i^!(\ood_X)$ is isomorphic to $\ood_Z\lan -c \ra$ if 
$i$ is everywhere of codimension $c$. 

Lastly, we assume that {\it homotopy invariance} holds for $\md$, i.e., that for any $X\in G$ and $f$ being the projection $\afo(X)\to X$
the counit morphism $f_!f^!(\ood_{X})\to \ood_X$ is invertible.


In order to describe the extension of the results of \S\ref{sgabblem} to this general setting we need certain non-periodic generalizations of the corresponding definitions. 


\begin{defi}\label{dnp}
\begin{enumerate}
\item\label{dnp-1} In the notation of Definition \ref{dnicepair} (including $X$ being a $B$-scheme) we define $\phd_X^U(Z\to W)= \pi_* i_Z^* j_{W,*}(\ood_{U_W})$. 

In the case $W=X,\ Z=Q$ we  will replace the notation $\phd_X^U(Q\to X)$ just by $\phd_X^U(X)$.

\item\label{dnp-2} For $r\in \z$ we define the following classes of objects in $\dkx$: $\tn_r(X)$ is the $\dkx$-envelope of $\{\mgdx(V)[m]\lan - m \ra\}$ for $V\to X$ running through $G$-morphisms with $V$ being  regular, 
  $m\ge 0$,  and $\de(V)+m\le r$; $\tvn_r(X)$ is the envelope of $(\tn_{r-1}(X)[1]\cup (\tn_{r-1}(X))\lan -1\ra)$. 

\item\label{dnp-3} We will say that 
$j$ is {\it $\md$-reasonable} if $\phd_{X}^{U}(X)\in \tvn_{\de(X)}(X)$. 

\item\label{dnp-4} For any $n\in \z$, $-r-1\le n\le r$ we consider the envelope $\tn_r^n(X)$ of $\{\mgdx(V)\lan -\max(m,0)\ra [m]:\ m\le n\}\cup \{v_*(\ood_V)\lan  -\max(m,0)\ra [m]:\ m> n\}$ for $v\colon V\to X$ running through finite type (separated) morphisms with regular domain and $\de(V)+|m|\le r$ (cf. Theorem \ref{tgabber}(\ref{tgabber-2})).
\end{enumerate}

\end{defi}

Now we will sketch the proof of the analogues of the results of \S\ref{sgabblem}; we will formulate them in the order in which they can actually be proven.

\begin{theo}\label{tnp}
Let $r\ge 0$; let $X$ be a $B$-scheme.

I. Then the following statements are valid.
 \begin{enumerate}
\item\label{inice1}
The (obvious) $\md$-analogues of parts \ref{idtriangle}--\ref{iaddefect} of Lemma \ref{lnice},  Remark \ref{rreas}(\ref{ifuni}), and   Lemma \ref{lnice}(\ref{idcover}) are fulfilled.

\item\label{idshiftn} $\tvn_r(X)\subset \tn_r(X)$.  Moreover, for any $s> 0$ we have 
 $\tvn_{r-s}(X)\lan - s\ra \subset   \tvn_r(X)$,  $\tn_{r-s}(X)\lan - s\ra [-s]\subset \tvn_{r-s+1}(X)\lan 1-s\ra[-s] 
\subset   \tn_r(X)$, and $\tn_{r-s}\lan -s\ra[-1]\subset \tvn_r$. 


\item\label{ireas} The $\md$-analogues of   Lemma \ref{lnice}(\ref{idreg}, \ref{indep}) are also valid. 

\end{enumerate}
II Assume in addition that $\md$ satisfies the following {\it (local) splitting property}: if  $g\colon Y'\to Y$ is a degree $e>0$  finite 
morphism of $\lam$-nice schemes then there exists an open $U\subset Y$ such that for the morphism $g_U=g\times_Y U:Y_U\to U$ the adjunction morphism $\ood_{U}\to g_{U,*}g_{U}^*(\ood_{U}) =g_{U,*}(\ood_{Y_U})$ becomes split injective (i.e.,  a coretraction) in the localized category $\md(U)[e\ob]$. 

Then the 
$\md$-analogue of  Theorem \ref{tgabber} is fulfilled. 

\end{theo}
\begin{proof}
I.\ref{inice1}. Note that the $\md$-version of Theorem \ref{tcd}(I.\ref{iglu}) implies that the functors $f^*$, $f_*$, $f^!$,  and $f_!$  are invertible whenever $f$ is a  nil-immersion; see Proposition 2.3.6(1) of \cite{cd}.
Hence the proofs of the corresponding statements carry over to the $\md$-context without any difficulty.

\ref{idshiftn}. Similarly to the proof of Lemma \ref{lnice}(\ref{idshift}), 
the non-trivial inclusions can be easily obtained  from the following observations:
 $\de(\mathbb G_m(V))=\de((\p^1)(V))=\de(V)+1$, $\mgdx(\p^1(V))\cong \mgdx(V)\lan - 1 \ra \bigoplus  \mgdx(V)$, and $\mgdx(\mathbb G_m(V))\cong \mgdx(V)\lan -1 \ra \bigoplus  \mgdx(V)[1]$ (the latter is an easy consequence of the homotopy invariance property).

\ref{ireas}. The arguments used in the proofs of 
 Lemma \ref{lnice}(\ref{idreg}, \ref{indep}) can be carried over to the $\md$-setting. The only notable distinction here is that when we consider regular stratifications of $X$ and $U$, respectively, we should assume (for simplicity) that the components are irreducible for $l>1$. Then for the corresponding embeddings $j_l\colon Y_l^\al\to Y$ (for $Y$ equal to $X$ or $U$, respectively) we would have $j_l^!(\ood_Y)\cong \ood_{Y_l}\lan -\codim_{Y}Y_l\ra$. Still one can easily ``handle'' these Tate twists using the previous assertion.

II. 
It is easily seen that the local splitting property (along with assertion I) enables one to carry over all the arguments used in the proof of Proposition \ref{pgab} and Theorem \ref{tgabber} to the $\md$-context.

\end{proof}

\begin{rema}\label{rcompvm}
\begin{enumerate}
\item\label{rcompvm-1} In contrast with most of the results of this paper, Theorem \ref{tnp} does not depend on the existence of 
any (Chow-type) weight structures. Thus it can be applied to $h$-motives as considered in \cite{cdet}. Since the latter are closely related to constructible complexes of 
 \'etale sheaves, we obtain some new statements on the latter, that essentially enhance Gabber's constructibility results (see \cite[\S XIII]{illgabb}). 
Note here that the methods of these papers allow us to avoid inverting 
the corresponding $\sss$ (in the coefficient ring $\lam$).
One may also weaken some other of our restrictions on $B$. 


\item\label{rcompvm-2} Recall that in ibid. it was proved that the corresponding functors $f^*$, $f_*$, $f_!$, and $f^!$ respect the constructibility of objects. We could have used the method of ibid. (that were also applied in \cite{cdet})  to deduce the $\md$-version of this result from the $\md$-versions of Theorem \ref{tcd}(I.\ref{icatz}, \ref{imotgen}, \ref{icoprod}) (combined with the assumptions on $\md$
introduced above).  

Note also that one can assume  $\md(-)$ to be  defined on all pro-open subschemes of $\obj G(B)$; then 
the $\md$-version of Theorem \ref{tcd}(I.\ref{icont}) can be used to reduce the local splitting property (that was described in Theorem \ref{tnp}(II)) to the following statement:  if $g\colon F'\to F$ is a finite 
morphism of spectra of fields in pro-$\obj G(B)$
  of  degree $e$ 
	then 
the image of the  unit morphism $\ood_F\to g_*g^*(\ood_{F}) \cong g_*(\ood_{F'})$ in $\md(F)[e\ob]$ splits.

\item\label{rcompvm-3} Now we illustrate our definitions in a rather simple (and yet quite important) case. 

Assume that $\de=\de^B$ (see Definition \ref{ddf}(\ref{ddf-3})), whereas 
$B$ is the spectrum of a perfect field. Then $\de$  equals the Krull dimension function. 
Next, if  we assume in addition that $\md$ is the Beilinson motives $2$-functor (as studied in \cite{cd}) and $X$ is finite over $B$, then $\md(X)$ is the unbounded version of  the $\q$-linear Voevodsky motives over $X$; see \S16.1 of \cite{cd}. 
Moreover,
 for any smooth connected variety $V/X$ the object  $\mgbmx(V)\lan \de(V)\ra$  is the Voevodsky motif $\mg_X(V)$ of $V$ over $X$; see \cite[(2.3.4.a)]{bondegl}. 

\item\label{rcompvm-4} We also describe a simple case when $j_!(\ood_U)$, $j_*(\ood_U)$, and
$\phd_{X}^{U}(X)$ can be put into  ``very explicit'' envelopes as prescribed by Theorem \ref{tnp}. Assume that $X$ is regular and that $Q=X\setminus U$  is the union $\cup_{j\in J} Q_j$, where all $Q_j$ and all of the intersections $Q_{J'}=\cap_{j\in J'} Q_j$ for $J'\subset J$ are regular closed subschemes of $X$. Then one can easily verify the following: 
$\mgdx(U)$ belongs to the envelope of $\{\mgdx(Q_{J'})[-\#J'] :\ J'\subset J\}$, $j_*(\ood_U)$ belongs to the envelope of $\{\mgdx(Q_{J'})\lan - \codim_{X}Q_{J'} \ra [\#J'] :\ J'\subset J\}$, and the motif $\phd_{X}^{U}(X)$ belongs to the envelope of 
${\mgdx(U)} \cup \{\mgdx(Q_{J'})\lan - \codim_{X}Q_{J'} \ra [\#J'] :\ J'\subset J, J'\neq \emptyset\}$. Note also that 
all the embeddings $Q_{J'}\to X$ are proper morphisms (and so, $\mgdx(Q_{J'})\lan s\ra$ is a ``$\md$-Chow motif'' over $X$ for any  $s\in \z$ and $J'\subset J$).
 
Hence these envelope statements demonstrate ``explicitly'' that $\phd_{X}^{U}(X)\in \tvn_{\de(X)}(X)$ and that the motives $\mgdx(U)$ and $j_*(\ood_U)$ belong to $\tn_{\de(X)}^n(X)$ for any $n$ between $-\de(X)-1$ and $\de(X)$ (so, we have no need to consider different envelopes for distinct $n$ here). Certainly, this simple example motivated our definitions and formulations; it is quite remarkable that we are able to prove the corresponding results in our very general setting. 

\end{enumerate}
\end{rema}

\section{``Studying weights'' using $ldh$-descent, Voevodsky motives, and their \'etale (co)homology}\label{scompvoevet}

In this section we will mostly consider the case $\lam=\q$ (and $\sss=\p$); see Remark \ref{rqcoeff} below. So, it will be sufficient for us to assume all the schemes we consider to be nice (since this is equivalent to being $\q$-nice).

In \S\ref{scompdm} we study the connecting functors between $\dk(-)$ (for $\lam=\q$) with the ones of Beilinson motives. Since these functors are weight-exact (with respect to the corresponding Chow weight structures), the ``weights'' of  $f_*(\oo_Y)$ equal  the ones of its Beilinson-motivic analogue $f_*(\q_Y)$.

In \S\ref{sdelw} we relate the weights of $f_*(\q_Y)$ to the \'etale cohomology of $Y$ and the Deligne weights for it in the case where $X$ is the spectrum of a field. The general formalism of weight spectral sequence yields an inequality between two weight bounds of this sort. This weight estimate is conjecturally precise; we prove that this is actually the case in certain ``(almost) maximally singular'' case (that we illustrate by a simple example).

In \S\ref{sbase} we extend these results to the case of a (more) general $X$. Somewhat unfortunately, this forces us to consider certain perverse \'etale cohomology and weights for it. Note however that ``the weights'' of $\q_Y$ are not greater than that of $f_*(\q_Y)$ and this inequality is precise in some cases; thus considering the case of a ``general'' $X$ appears to be important for these matters. 
 
In \S\ref{sldh} we use an easy argument to put $f_*(\ood_Y)$ in an envelope of an explicit set of shifts of Chow motives in the case where $X$ is a variety (here one may take the motivic category $\md(-)$ being either to $\dk(-)$ or to any of the two aforementioned versions of Voevodsky motives). This certainly yields a ``nice'' explicit set of ``test schemes'' for the weights of $f_*(\ood_Y)$. We also discuss a 
possible relation of our weight spectral sequences to the {\it singularity} ones of  \cite{pasdesc} and \cite{ciguil}.

\subsection{Comparison with the categories $\dm(-)$}\label{scompdm}

Now (in the case $\lam=\q$) we compare $\dk(-)$ with the Beilinson motives $\dm(-)$ (i.e., with $\q$-linear Voevodsky motives).

\begin{pr}\label{pcompdm}
Let $X$ be a nice scheme.
Then the following statements are valid.

\begin{enumerate}
\item\label{icmot1} There exists an exact 
faithful functor 
$\comp(X)\colon\dmx\to \dkx$; it respects (small) coproducts and maps the unit object $\q_X$ of $\dmx$ into $\oo_X$.

\item\label{icmot2}
The system of these functors (when $X$ varies) is compatible with all the motivic image functors for the corresponding categories.

\item\label{icmot3} 
$\comp(X)$ possesses a faithful right adjoint $\mper(X)$ that respects coproducts; these functors also commute with all the motivic image functors.

\item\label{icmot4}
 The composition  $\comp(X)\circ \mper(X)$  (resp. $\mper(X)\circ \comp(X)$) sends $M\in \obj \dkx$ into $\coprod_{i\in \z}M$ 
(resp. sends $N\in \obj \dmx$ into  $\coprod_{i\in \z}N(i)[2i]$).

\end{enumerate}
\end{pr}
\begin{proof}
We apply several results from (\S14.2 of) \cite{cd}. 

First we note that instead of $\sht(-)[\p\ob]$ one can consider the  categories $\daoq(-)$ (see the formula (5.3.35.2) of ibid.); those are 
homotopy categories of certain ``functorial'' stable model categories.
By Theorem 14.2.9 of 
ibid., $\dmx$ can be defined as the category of modules over a certain 
ring object $H_{B,X}\in \obj \daoqx$ (with respect to the underlying model structure). 
 Next, $\dkx$ can be described as  the category of modules over $\coprod_{i\in \z}B_X(i)[2i]$ in the model category underlying  $\daoqx$ (see Corollary 14.2.17 of ibid.). Hence 
we can set $\comp(X)(M)=\coprod_{i\in \z} M(i)[2i]$ (for $M\in \obj \dkx$) and endow it with the natural structure of a $\coprod_{i\in \z}B_X(i)[2i]$-module;
$\mper(X)(N)$ can be described in terms of the corresponding forgetful functor. Certainly, both of these functors respect coproducts. 

Our remaining assertions easily follow from the general formalism described in \S7.2 of ibid. 

\end{proof}

\begin{rema}\label{rqcoeff} We have two reasons for concentrating on the case $\lam=\q$ in this section.

1. The functors $\mper(-)$ 
 are defined in this setting only.

2. Below we will study the question when a bound from below on the ``weights'' of the \'etale homology of a motif $M$ implies a similar bound on the weights of $M$ itself. Certainly, one may hope for a statement of this sort (for a more or less ``general'') $M$ only for motives with rational coefficients.

So, below we will apply several results from \cite{brelmot} and \cite{bmm}. Note however that in \cite{bonivan} similar statements for $\zop$-linear motives over characteristic $p$ schemes were established (here we set $\zop=\z$ in the case $p=0$).

\end{rema}


As we have already said (see Remarks \ref{rwwc}(\ref{rwwc-2})
and \ref{rexplwchow}(\ref{rexplwchow-2})), the categories $\dm(-)$ are also endowed with Chow weight structures. Now we verify that $\mper(-)$ and $\comp(-)$ ``detect weights''.

\begin{pr}\label{pcompw}
Let $X$ be a nice scheme. Then for any $M\in \obj\dkx$ we have $M\in \dkx_{\wchow\le 0}$ (resp. $M\in \dkx_{\wchow\ge 0}$) if and only if 
$\mper(X)(M)\in \dmx_{\wchow\le 0}$ (resp. $\mper(X)(M)\in \dmx_{\wchow\ge 0}$). 

Moreover,  for any $N\in \obj\dmx$ we have $N\in \dmx_{\wchow\le 0}$ (resp. $N\in \dmx_{\wchow\ge 0}$) if and only if 
$\comp(X)(N)\in \dkx_{\wchow\le 0}$ (resp. $\comp(X)(N)\in \dkx_{\wchow\ge 0}$).

\end{pr}
\begin{proof}

First we recall  that $\comp(X)$  ``commutes with the motivic image functors'' and sends $\q_X$ into $\oo_X$.  
Hence the descriptions of the corresponding Chow weight structures yield the left weight-exactness of $\comp(X)$ and the right weight-exactness of $\mper(X)$ (see Proposition \ref{pextws}). 

Next, to verify the right weight-exactness of $\comp(X)$ we should check the following: if $N\in \dmx_{\wchow\ge 0}$ and $t\colon T\to X$ is a finite type morphism 
then $t_*(\oo_T)[j]\perp \comp(X)(N)$ for all $j<0$. The adjunction $\comp(X)\dashv \mper(X)$ translates this into $t_*(\q_T)[i]\perp \mper(X)\circ \comp(X)(N)$. Hence it suffices to apply Proposition \ref{pcompdm}(\ref{icmot4}) (along with the compactness of $\q_T$ and the weight-exactness of $-(i)[2i]$ for $i\in \z$).

Moreover, Proposition \ref{pextws}(\ref{iextwcommcoprod}) yields that $\mper(X)$ is left weight-exact. Hence both   $\mper(X)$ and $\comp(X)$ are weight-exact.

Next, 
Proposition \ref{pextws}(\ref{iextwcoprod})  (along with part \ref{icmot4} of the previous proposition) yields that the ``strict weight-exactness'' of $\comp(X)\circ \mper(X)$ and of  $\mper(X)\circ \comp(X)$ (this means: $M\in \dkx_{\wchow\le 0}$ if and only if 
$\comp(X)(\mper(X)(M))\in \dkx_{\wchow\le 0}$; $M\in \dkx_{\wchow\ge 0}$ if and only if 
$\comp(X)(\mper(X)(M))\in \dkx_{\wchow\ge 0}$; 
$M'\in \dmx_{\wchow\ge 0}$ if and only if  $\mper(\comp(X)(M'))\in \dkx_{\wchow\ge 0}$; $M'\in \dmx_{\wchow\le 0}$ if and only if  $\mper(\comp(X)(M'))\in \dkx_{\wchow\le 0}$).  This certainly yields the results in question.


\end{proof}

\begin{rema}\label{rcompwmd}
In the current paper we are mostly interested in the ``weights'' of $M$ being of the form $f_*(\oo_Y)$ (for $f\colon Y\to X$ being a separated morphisms of nice schemes). According to the previous proposition, to compute the ``weight range'' of (this) $M$ it suffices to bound the weights of $\mper(X)(M)$. Now, according to Proposition 
\ref{pcompdm} 
  the object $\mper(X)(M)$ is isomorphic to $\bigoplus_{i\in \z} f_*(\q_Y)\lan i\ra$. Thus Proposition \ref{pextws}(\ref{iextwcoprod}) implies the following:
for $n\ge 0$ we have $M\in \dkx_{\wchow\ge -n}$  if and only if  $ f_*(\q_Y)\in \dmx_{\wchow\ge -n}$.

\end{rema}

\subsection{On ``detection of weights via
 \'etale (co)homology'' over a field}\label{sfield}

As we have just shown (in the case $\lam=\q$) it suffices to ``study weights'' for (compact)  Voevodsky motives instead of $K$-ones; so we will only consider $\dm(-)$ till the end of the paper. 

Now we recall that (for any homological $H\colon\dmcx\to \au$) Corollary \ref{cdetect} yields a method for proving that a given (compact) motif $N$ does not belong to $\dmcx_{\wchow
\ge -n}$ as follows:

\begin{coro}\label{cdetectw}

For some $M\in \obj \dmcx$ and $n\in \z$ assume  one of the following conditions is fulfilled:

1. $E_2^{pq}T_{\wchow(X)}(H,M)\neq 0$ for some $q\in \z$ and $p>n$.

2. $(W_{-n-1}H_q)(M)\neq 0$  for some $q\in \z$. 

Then $M\notin \dmcx_{\wchow(X)\ge -n}$. Moreover, condition 2 implies condition 1.

\end{coro}

 So the problem is to find such an $H$ for which the corresponding computations are manageable.

In \S\ref{sbase} below we will consider two different $H$ related  to \'etale homology.\footnote{We restrict our theories to compact motives since \'etale homology probably does not behave ``nicely enough'' on the whole $\dmx$; cf. Lemma 2.4 of \cite{ayoubconj}.}
We will start with the case $X=\spe k$, $k$ is a field; in this case our versions of $H$ coincide.

We fix a prime $l$; till the end of the section we will assume that it is invertible on all the schemes we consider (so, $l\neq \cha k$ if $X=\spe k$).
Now recall the existence of a functor $\hetl(X)\colon\dmcx\to D^b_cSh^{et}(X,\ql)$; the latter is the triangulated category of constructible \'etale complexes  of $\ql$-sheaves (i.e.,  of continuous finite dimensional $\ql$-representations of the absolute Galois group of $k$). Since we will apply a generalization of this functor below, the optimal reference for us here is \S7.2 of \cite{cdet}. 
Let $\het=\het(X)$ denote the composition of $\hetl(X)$ with the zeroth canonical truncation functor for $D^b_cSh^{et}(X,\ql)$; so, the target of $\het$ is the category  the corresponding $Sh^{et}(X,\ql)$. 

Now we describe a certain ``weight detection'' conjecture; we will generalize it later.

\begin{conj}\label{conjfield}

Let $r,n\ge 0$.
We will say that the conjecture $WD^n(X)$ holds if for $M\in \obj \dmcx$ we have $M\in \dmcx_{\wchow\ge -n}$ whenever $E_2^{pq}T_{\wchow(X)}(\het,M)= 0$ for all $q\in \z$ and $p>n$.

We will say that  $WD_r^n(X)$ is valid if this implication holds under the additional assumption that  $M\in \tn_r(X)$ (see Definition \ref{dnp}(\ref{dnp-2}); we take $\md=\dm(-)$ and $\de$ being the Krull dimension function in it). 

\end{conj}

\begin{pr}\label{predconj}
The following statements are valid for any $n\ge 0$.

\begin{enumerate}
\item\label{predconj-1} $V_r(X)$ equals the envelope of $\{\mgbmx(V)[m]\lan - \max(m,0)\ra \}$ for $V$ running through proper regular $k$-schemes  and $\dim(V)+|m|\le r$.
In particular, if $k$ is perfect then $V_r(X)\lan r\ra$ equals the envelope of $\{\mg_X(V)[m]\lan - \min(m,0)\ra\} $ for $V$ being smooth projective over $k$ and $\dim(V)+|m|\le r$. 

\item\label{predconj-2} Conjecture $WD^n(X)$ is equivalent to the combination of the conjectures  $WD^n_r(X)$ for all $r\ge 0$.

\item\label{predconj-3} Let $k'$ be an algebraic 
 extension of $k$, $X'=\spe k'$. Then for any $r\ge 0$ the conjecture $WD^n_r(X')$ implies $WD^n_r(X)$. 

\item\label{predconj-4} Conjecture $WD^n_r(X)$ holds if $n\ge r-2$.
\end{enumerate}
\end{pr}
\begin{proof}

\ref{predconj-1}. The ``in particular'' part of the assertion is an easy implication of the first part; see Remark \ref{rcompvm}(\ref{rcompvm-3}). Now, the first part of the assertion can be easily verified using the alteration arguments from \S2.4 of \cite{bondegl}. 

\ref{predconj-2}. Easy; note that $\obj \dmcx$ equals the unions of $\tn_r(X)\lan j\ra$ and of $\tnp_r(X)\lan j\ra$ for all $j\in \z$ and $r\ge 0$.

\ref{predconj-3}. Denote by $f$ the corresponding morphism $X'\to X$. 
According to Remark 7.2.25 of \cite{cdet}, 
there exists a commutative square of functors
\begin{equation}\label{ecdet} 
\begin{CD}
 \dmcx@>{\het(X)}>> Sh^{et}(X,\ql) \\
@VV{f^*}V@VV{f^*_{et}}V \\
\dmc(X')@>{\het(X')}>>Sh^{et}(X',\ql) 
\end{CD}
\end{equation}
with $f^*_{et}$ being an exact conservative functor. Thus it suffices to verify that $M\in \dmcx_{\wchow\ge -n}$ whenever $f^*(M)\in \dmc(X')_{\wchow\ge -n}$.
The latter statement is a particular case of Theorem 2.3.1(VI) of \cite{bmm}.

\ref{predconj-4}. We should verify  
that if $M\in V_r(X)$ and $E_2^{pq}T_{\wchow(X)}(\het,M)= 0$ for all $q\in \z$ and $p>n$, then
$M\in \dmcx_{\wchow\ge -n}$ (assuming that $n\ge r-2$). Certainly, the statement is trivial for $n\ge r$; so we consider the cases $n=r-1$ and $n=r-2$.

According to the previous assertion, we can (and will) assume that $k$ is algebraically closed. 
Now, we consider $t(M)=(M^i)$. According to Proposition \ref{pwc}(\ref{iwcenv}) we may assume that $M^i=0$ for $i>r$, and for any $i\ge 0$ the motif  $M^{r-i}$ is a retract of $\mg_X(V^{r-i})\lan -i\ra$ with $V^{r-i}$ being smooth projective of dimension at most $i$ over $k$ (see Remark 
 4.2.3(2) of \cite{bsosnl}). 
According to \ref{pwc}(\ref{iwc3}) we should check that $t(M)$ is homotopy equivalent to a complex concentrated in degrees $\le n$. 

Assume now that $n=r-1$. We should verify that the boundary morphism $d^{r-1}\colon M^{r-1}\to M^{r}$ splits, whereas our condition on the weight spectral sequence translates into  the assumption that $d^{r-1}$ induces surjections on the $\ql$-\'etale homology in all degrees. This is a very simple property of Chow motives that is easily seen to follow from the fact that the numerical equivalence coincides with the homological one for $0$-cycles on $V^r\times V^{r-1}$ (since $V^r$ is just a collection of $k$-points). 

Now assume that $n=r-2$. According to (just established) case $n=r-1$ of our assertion, we may assume that $M^r=0$. We should check that $d^{r-2}$ 
splits. 
We certainly can assume that all the components of $V^{r-i}$ are of dimension $i$ (for $i=1,2$). 
Consider (for convenience) the dual $d$ to $d^{r-2}$; so this is a morphism from a retract $M_1$ of $\mg_X(V^{r-1})$  
into a retract $M_2$ of $\mg_X(V^{r-2})$ that yields surjections on the \'etale cohomology  $\hetc^*$  in all degrees.
We should verify that $d$ is split injective.

 For $s=0,1$ consider the full additive subcategories $C_s$ of $\chowe$ consisting of $N\in \obj \chowe(k)$ such that $\het^j(N)=0$ for $j\neq s$; we  define $C_2$ as the subcategory 
  containing  those $N\in \obj \chowe$ such that $\hetc^j(M)=0$ for $j<2$.

Certainly,  $C_0$ is equivalent to the category of finite dimensional $\q$-vector spaces and $C_1$ is equivalent to the Karoubi envelope of the $\q$-linearization of the category of $k$-abelian varieties; thus both of them are abelian semi-simple.  We also have $C_s\perp C_t$ for $0\le t<s \le 2$ (see \cite[Theorems 6.2.1, 9.2.2, and 15.3.1, and Proposition 17.5.4]{bvk}). 
 Moreover, any object of $\chowe$ can be presented as the direct sum of objects of $C_s$ (for $s=0,1,2$).  

It follows that $d$ can be presented as a successive extension
  of certain $C_s$-morphisms $d_s$ in for $s=2,1,0$. Hence the cohomological characterization of $C_i$ allows to consider the case $d=d_s$ for a single $s$. 
	
If $s=0$ or $1$ then 
	 our assertion immediately follows from the semi-simplicity of $C_s$ along with the well-known fact that $\hetc^s$ is conservative on $C_s$.
	
	It remains to consider the case $s=2$. Since $V_{r-1}$ is of dimension at most $1$, $M_1$ is a direct sum of $\q_k\{1\}$ (recall that we assume $k$ to be algebraically closed).
Hence  the fact that numerical equivalence coincides with the homological one for divisors easily implies that $d$ splits. 

\end{proof}

\begin{rema}\label{rconjfield}

\begin{enumerate}
\item\label{idegwss}
According to 
 Remark 2.4.3 of \cite{bws} (note that it is no problem to apply it to homology instead of cohomology),  the spectral sequence  $E_2^{pq}T_{\wchow(X)}(\het,M)$ degenerates at $E_2$, and its $E_2$-terms are the corresponding factors of a certain version of Deligne's weight filtration on $\het_*(M)$. We will make this much more precise below.

\item\label{ihodgeconj}
Certainly, the argument used in the proof of part \ref{predconj-4} of the proposition also allows to reduce Conjecture \ref{conjfield} to the Beilinson's conjectures on mixed motives and their weights (see \cite{beilh}). 
Many more statements of this sort can be found in \cite{bmm}. 

Note also that in the case $k\subset \com$ (and so, $\cha k=0$) one may consider singular (co)homology instead of \'etale one (here one may use Saito's Hodge modules and the ``realization'' results of \cite{ivohodge}). In this case one may also relate our conjecture to the Hodge one; see Proposition 7.4.2 of \cite{mymot}.
Moreover, one can probably relate the corresponding spectral sequence to the {\it singularity} one constructed in \S5 of \cite{ciguil}; see Remark \ref{rldh}(\ref{ihiron}) below.


\end{enumerate}

\end{rema}

\subsection{On the relation to Deligne's weights and ``maximally singular'' examples}\label{sdelw}

Certainly, the case most interesting for the purposes of the current paper is $M=f_*(\q_Y)$ for $Y$ being a separated scheme of finite type over $X$. Note that according to Theorem \ref{tnp}(II), if $d=\dim(Y)$ then $M\in V_d(X)$. 

Since Deligne's weights on \'etale cohomology are defined only if $k$ is a finitely generated field (where $X=\spe k$), we will start from computing $E_2^{pq}T_{\wchow(X)}(\het,M)$ in  this case.

\begin{pr}\label{phomcoh}
Assume that $M=f_*(\q_Y)$ for a morphism $f\colon Y\to X$ of reasonable schemes. Then the following statements are valid.

1. $\hetl(M)$ is canonically isomorphic to the \'etale 
total direct image $Rf_{*, et}(\ql_Y)$.

2. Assume that $X=\spe k$ where $k$ is a finitely generated field (so, it is finitely generated over the corresponding prime field), $f$ is a separated morphism of finite type. Then $E_2^{pq}T_{\wchow(X)}(\het,M)$ is canonically isomorphic to the $q$th weight factor of the \'etale cohomology sheaf $R^{p+q}f_{*, et}(\ql_Y)$.

\end{pr}
\begin{proof}
1.  Once again, we apply the ($2$)-functoriality of the family of functors $\hetl(Y)$ (as given by \S7.2 of \cite{cdet} that extends the earlier results of \cite{ayoet}, cf. also Theorem 2.4.1(I) of \cite{bmm}). The assertion is immediate from the compatibility of  $\hetl(-)$  with functors of the type $f_*$  (cf.   (\ref{ecdet})).  

2. According to the previous assertion,  the spectral sequence $T_{\wchow(X)}(\het,M)$ converges to $R^{p+q}f_{*, et}(\ql_Y)$. We should prove that it degenerates and its $E_2$-terms are the corresponding weight factors of $R^{p+q}f_{*, et}(\ql_Y)$. 

By the definition of Deligne's weights, we should choose some ``models'' $\tilde{f}\colon \tilde{Y}\to \tilde{X}$ (i.e., $X$ should be the generic point of $\tilde{X}$, $\tilde{f}$ is a separated morphism of finite type of separated schemes) such that $\tilde{X}$ is of finite type over $\spe \zol$,  and the sheaves 
$R^j\tilde{f}_{*,et}(\ql_{\tilde{Y}})\cong \het^j(M) $ are punctually mixed for all $j\in \z$ (i.e., 
for any closed
embedding of a point $i_x\colon x\to \tilde{X}$ the sheaves $i_{x,et}^*(R^j\tilde{f}_{*,et}(\ql_{\tilde{Y}}))$ 
are endowed with compatible 
filtrations with pure factors).
 
We start from choosing  arbitrary models $f'\colon Y'\to X'$ (with $X'$ being reduced of finite type over $\spe \zol$); for $M'=f'(\q_{Y'})$  consider some choice 
of $t_{\wchow({X'})}({M'})=({M'}^i)$. According to Remark 2.3.7(4) of \cite{brelmot}, there exists an open (dense) embedding $j'\colon X''\subset X'$ such that 
$j'^*(M'^i)$ are certain retracts of $p^i_*(\q_{P^i})$, where $p^i$ are some compositions of universal homeomorphisms with smooth proper morphisms and $P^i$ are regular. Moreover, the continuity property for $\dmc(-)$ (cf. Theorem \ref{tcd}(I.\ref{icont}))  yields the existence of an open dense $\tilde{X}\subset X''$
such that for the 
$\tilde{j}\colon\tilde{X}\to X'$ the corresponding retracts of  $(p^i\times_{X''}\tilde{X})  _*(\q_{P^i\times_{X''}\tilde{X}})$ are the terms of (some choice of a) $w_{\chow(\tilde{X})}$-weight complex of $\tilde{j}^*(M')$ (see Proposition \ref{pwc}(\ref{iwctow})). 
Now, the sheaves $R^j(p^i\times_{X''}\tilde{X})_{*,et}(\q_{P^i\times_{X''}\tilde{X}})  $ are certainly punctually pure of weight $j$ for any $i$ and all $j\in \z$. Since the categories of punctually pure sheaves of a given weight contain all subquotients of their objects, and there are ``no non-zero morphisms between distinct weights'',  
Remark 2.4.3 of \cite{bws} yields the result (note once again that it is no problem to apply it to homology instead of cohomology).


\end{proof}

\begin{rema}\label{rdwe}
Certainly, the argument above may be applied to any $M\in \obj \dmcx$.

Moreover, it is actually not (that) necessary to assume that $k$ is a finitely generated field. Indeed, the continuity property of  $\dmc(-)$ ensures the existence of a finitely generated $k_0\subset k$ such that $M$ ``is defined over $k_0$'', i.e., for the corresponding morphism $g\colon X\to X_0=\spe k_0$ there exists $M_0\in \obj\dmc(X_0)$ such that $g^*(M_0)\cong M$. Now, the functor $g^*$ is weight-exact (with respect to the corresponding Chow weight structures) according to Theorem 2.3.1(VI) of \cite{bmm}. Thus the factors of the Chow-weight filtration for $\het_*(M)$ ``may be computed at $X_0$''; see the proof of \cite[Theorem 2.5.4(II.1)]{bmm} for more detail.
\end{rema}

Next, Proposition \ref{pcompw} justifies the usage of the notation $c(Y/X)$ (see Remark \ref{rcones}(\ref{rcones-1}))
for Beilinson motives instead of $K$-ones. 

\begin{pr}\label{pcxx}
For $f\colon Y\to X$ being any separated morphism of nice schemes 
we set  $c(Y/X)=n$ whenever 
$$
g_*(\q_{Y})\in \dmx_{\wchow\ge -n}\setminus  \dmx_{\wchow\ge 1-n}.
$$

I. Then the following statements are valid.

\begin{enumerate}
\item\label{inoneg} $c(Y/X)\ge 0$.

\item\label{ineq} $c(Y/X)\le c(Y/Y)$.

\item\label{ibou} $c(Y/Y)\le \dim(Y)$. 

\item\label{iboucodim} More generally, for any $r\ge 0$ there exists $Z\subset Y$ of dimension less than $r$ such that $c((Y\setminus Z)/X)\le \dim(Y)-r$.

\item\label{icov} Let $x_i\colon X_i\to X$ be an open cover of $X$. Then $c(Y/X)$ equals the maximum of $c(Y_i/X_i)$ for $Y_i=Y\times_X X_i$. 

\end{enumerate}

II. Assume that $X=\spe k$. Then $c(Y/X)\le n$ for $n\ge \dim(Y)-2$ if and only if $E_2^{pq}T_{\wchow(X)}(\het,f_*(\q_Y))=0$ for all $p>n$.

\end{pr}
\begin{proof}
I. The  $\dkx$-analogues of 
assertions \ref{inoneg}--
\ref{iboucodim}
are given by 
Remark \ref{rcones}(\ref{rcones-1}) and Theorem \ref{tweib}(I.\ref{3.3.1.I.1}--\ref{3.3.1.I.2}), respectively, whereas to obtain the $\dkx$-version of assertion \ref{icov} one should apply 
Proposition \ref{pwchownc}(\ref{iwexchowoc}) to the motif $f_*(\oo_Y)$.

 To carry over these statements to Beilinson motives it remains to apply Propositions \ref{pcompdm}
 and \ref{pcompw}.

Besides, assertions \ref{inoneg}, \ref{ineq}, and \ref{icov} can be easily deduced from the results of \cite{brelmot}, whereas assertions \ref{ibou} and \ref{iboucodim} follow from Theorem \ref{tnp}(II). 

II This is just a part of Proposition \ref{predconj}(\ref{predconj-4}). 

\end{proof}

\begin{rema}\label{rmaxsing}
1. Moreover, we certainly have $c(Y/Y)=c(Y/X)=0$ whenever $Y\re$ is regular. Thus one may say that $c(Y/Y)$ is a certain ``measure of singularity'' for $Y$. We will say that $Y$ is {\it maximally singular} whenever $c(Y/Y)=\dim(Y)$. Note that (according to parts  I.\ref{iboucodim}--\ref{icov} of our proposition) this property is ``concentrated around'' a finite number of closed points of $Y$.

2. 
Certainly, $c(Y/X)=\dim(Y)$ is an even stronger restriction on $Y$ (say, for $X$ being the spectrum of a field). Yet we will now demonstrate that this equality is not a ``local'' condition. 

3. 
It certainly may be  interesting to relate our results (along with the maximal singularity condition for schemes) to Theorem 7.1 of \cite{hasdescent}.
As an intermediate step, one may try to consider a cohomology theory on $\dmx$ or on $\dkx$ defined similarly to the {\it weight homology} of \cite{kellyweighomol}.

\end{rema}

We will study in detail two simple examples for our notions.
We set $X=\spe k$ for some infinite field $k$ (actually, this infiniteness restriction can easily be dropped),
fix some $r\ge 0$, $N\ge 1$, and denote by $Y_N$ the union of some collection of $N$ generic  hyperplanes in $\af^{r+1}(k)$ (we want any subset of our collection of hyperplanes to intersect properly; in particular, this means that any two non-empty intersections of distinct subsets of  our collection are distinct).

\begin{pr}\label{pexamples}
\begin{enumerate}
\item\label{ivmax} $c(Y_N/X)=r$ whenever $N>2r$. 

\item\label{icontr} 
$c(Y_{r+1}/X)\le \max(r-2,0)$.

\item\label{imax}
$Y_N$ is maximally singular for any $N>r$.

\end{enumerate}
\end{pr}
\begin{proof}
1,2. 
 Combining Proposition \ref{predconj}(\ref{predconj-4})
 with Proposition \ref{phomcoh} (along with Remark \ref{rdwe}), it suffices to compute the quotients of the weight filtration for $\hetc^*(Y_N)$.
This is easy; note that these weights  can be computed using the spectral sequence $E_1^{pq}=\bigoplus_{I\subset [1..N],\ \#I=p+1} \hetc^q(Y_I)\implies \hetc^{p+q}(Y_N)$.
coming from proper descent, where $Y_I$ is the intersection of the hyperplanes whose numbers belong to $I$. Certainly, the only cohomology of $Y_I$ is $\ql_X$ in degree $0$; so all the Deligne's weights in $\hetc^*(Y_N)$ are zero. Thus it suffices to note that   $\hetc^r(Y_N)\neq 0$ whenever $N> 2r+1$ and that  $\hetc^j(Y_{r+1})= 0$ for all $j\neq 0$.


3. According to assertion 1, $Y_{2r+1}$ is maximally singular. Now, for any $N_1,N_2> r$ the schemes $Y_{N_i}$ admit open covers such that any component of the first cover is isomorphic to some component of the second one and vice versa. 
Hence the statement follows from Proposition \ref{pcxx}(\ref{icov}).

\end{proof}

\begin{rema}\label{rvd}
1. Thus the criterion   given by  Proposition \ref{predconj}(\ref{predconj-4})
 (along with Proposition \ref{phomcoh}) is non-vacuous, and our results yield quite a non-trivial relation between (weights of) \'etale (co)homology and  negative $\kk$-groups! 
Note moreover that the corresponding cycle class maps (i.e., ``regulators'') are very far from being surjective or injective (in general).

2. On the other hand, we obtain that one cannot compute $c(Y/Y)$ looking at the $X$-cohomology of $Y$ only. This urges us to consider the relative versions of our criteria.

3. Now denote $Y_{r+1}$ by $Y$. It can be easily seen (cf. the proof of our proposition and of Proposition \ref{pldh} below) that there exists a choice of a weight complex $(M^i)$ of $M=\q_Y$ such that $M^i=\bigoplus_{I\subset [1..N],\ \#I=-i+1}\mgbd_Y(Y_I)$; in particular, $Y_{-r}=\mgbd_Y(f)$ for $f$ being the intersection of all the hyperplanes. Combining
the latter fact with Proposition \ref{pbw}(\ref{iortprecise}) and with part \ref{imax} of our proposition we obtain  that $\kk_{-r}^{BM,Y}(f)\neq \ns$ (cf. Remark \ref{rweib}(\ref{iextest}); so, $f$ is the corresponding "test scheme").  We can re-formulate this condition using the  long exact sequence 
$$\kk_{-r+1}(Y)\to \kk_{-r+1}(Y\setminus f) \dots\to \kk_{-r}^{BM,Y}(f)\to \kk_{-r}(Y)\to  \dots$$
(see (\ref{ekthlesgen})); we obtain that $Y$ being maximally singular is "seen from" the non-surjectivity of  $\kk_{-r+1}(Y)\to \kk_{-r+1}(Y\setminus f)$. Moreover, for $r>1$ the latter fact is equivalent to the non-vanishing of $\kk_{-r+1}(Y\setminus f)$ according to part \ref{icontr}  of our proposition.

4. Actually, one may apply the motivic Verdier duality (see \S4.4 of \cite{cd})  to prove that $c(Y_{r+1}/X)=0$.
\end{rema}

\subsection{``Detection of weights'' via \'etale homology over a general base}\label{sbase}
 
Now we discuss two ``relative'' generalizations of Conjecture \ref{conjfield}.
 For this purpose we fix some $\zll$-nice base scheme $B$, 
 and denote by $G'$ the set of schemes that are separated and essentially of finite type over $B$.\footnote{We consider pro-open subschemes of (finite type) $B$-schemes since we want to pass to certain limits in the theorem below.}
We will need certain (non-negative) dimension functions on $G'$; for our two homology theories two somewhat different types of dimension functions are ``optimal''.
 So, let $\de$ denote  the function obtained by the trivial extension of Definition \ref{ddf} from $G(B)$ to $G'(B)$.

Now, our second homology theory is  a certain perverse \'etale homology for motives.
The latter was  defined  in  \S2.4 of \cite{bmm} (certainly, this section  heavily relied on \cite{cdet}); yet we note that in the setting of ``general'' schemes one should invoke a ``true'' dimension function  (see Remark \ref{rdimf}(\ref{idimfkrull})) in this construction. 
So we assume that
a  non-negative function $\de'$ is defined on $G'$, and this  function satisfies the natural analogues of the properties listed in Proposition \ref{pdimf} along with  $\de
(X)-\de'(Z)=\codim_X(Z)$ for $Z\subset X$ being any irreducible schemes in $G'$ (cf. Remark \ref{rdimf}(\ref{idimfkrull})). 
Recall that the existence of $\de'$ is a  certain restriction on $B$ (we will abbreviate it by saying that ``$\de'$ exists''); however, it suffices to assume that $B$ is irreducible. 

For any $Y\in G'$ we
consider the $\ql$-\'etale realization functor 
 $\hetl(Y)\colon\dmcy 
\to D^cSh^{et}(Y,\ql)$ (this is an  exact functor whose target is a certain version of the derived category of constructible $\ql$-etale sheaves over $X$, as defined in \S7.2 of \cite{cdet} (generalizing "classical" $l$-adic categories).\footnote{This realization construction extends a closely related one from \cite{ayoet}.} Next, $D^cSh^{et}(Y,\ql)$ is endowed with

a bounded perverse $t$-structure (for the ``self-dual'' perversity corresponding to $\de'$) that will be denoted just by $t$; the heart of $t$ will be denoted by $\shy$. 
We set $\het(Y)$ to be the composition of  $\hetl(X)$ with the zeroth $t$-homology functor on $D^cSh^{et}(Y,\ql)$.

A special case here is when $Y$ is the spectrum of a field. 
Since in this case the perverse $t$-structure differs from the canonical one only by a shift, 
one may (at the price of a minor abuse of notation) 
take for $\het(Y)$ the functor introduced in the previous subsection. In particular, this functor can be defined
without (necessarily) assuming that $\de'$ exists. 

Actually, we will mostly be interested in $\de'$ in the case where $B$ is the spectrum of a field (and then it coincides with $\de$, and its restriction to $G(B)$ is just the Krull dimension function) and for $X$ being of finite type over $B$; cf. Theorem \ref{tredconj}(\ref{tredconj-4}, \ref{tredconj-6}) below. 

Now we are able to formulate our  ``relative weight detection'' conjectures. 

\begin{conj}\label{cwdhet}
Let $X\in G'$, $r,n\ge 0$.
\begin{enumerate}
\item\label{cwdhet-1} We will say that the conjecture $WD^n(X)$ holds if for any  for $M\in \obj \dmcx$ and $H$ being the product of the theories $\het(x)\circ i_x^!$, where $x$ runs through Zariski points of $X$, $i_x$ is the corresponding embedding (see Proposition \ref{pwchownc}(\ref{iwexchow8}) for the definition of $i_x^!$),  we have $M\in \dmcx_{\wchow\ge -n}$ whenever $E_2^{pq}T_{\wchow(X)}(H,M)= 0$ for all $q\in \z$ and $p>n$.

We will say that  $WD_r^n(X)$ is valid if this implication holds under the additional assumption that  $M\in \tn_r(X)$ (see Definition \ref{dnp}(\ref{dnp-2}); we take $\md=\dm$ in it).

\item\label{cwdhet-2} Assume that $\de'$ is defined (in the sense described above).  Then we will say that the conjecture $WD'^n(X)$ holds if  for $M\in \obj \dmcx$ we have $M\in \dmcx_{\wchow\ge -n}$ whenever $E_2^{pq}T_{\wchow(X)}(\het(X),M)= 0$ for all $q\in \z$ and $p>n$.

We will say that  $WD_r^n(X)$ is valid if  the latter implication holds under the additional assumption that  $M\in \tnp_r(X)$; here $\tnp_r(X)$ is defined by setting $\md=\dm$ and $\de=\de'$ in Definition \ref{dnp}(\ref{dnp-2}).
\end{enumerate}
\end{conj}

Now we describe some evidence supporting these conjectures; we will also discuss some conjectural evidence for them later.


\begin{theo}\label{tredconj}
The following statements are valid.
\begin{enumerate}
\item\label{tredconj-1}
Conjecture $WD^n(X)$ is equivalent to the combination of the conjectures  $WD^n_r(X)$ for all $r\ge 0$.

\item\label{tredconj-2}
$WD'^n(X)$ is equivalent to the combination of all $WD'^n_r(X)$. 

\item\label{tredconj-3}
Conjecture $WD^n_r(X)$ holds whenever $WD^n_r(x)$ does for $x$ running through all Zariski points of $X$.

\item\label{tredconj-4}
Conjecture $WD'^n_r(X)$ holds whenever $WD^n_r(x)$ does for $x$ running through all Zariski points of $X$ and $X$ is a variety over a field $k$ 
 (of characteristic distinct from $l$). 

\item\label{tredconj-5}
Conjecture $WD^n_r(X)$ holds whenever $n\ge r-2$. 
This statement may be applied to $M=f_*(\q_Y)$, where $Y$ is a separated 
$X$-scheme possessing an $X$-compactification of $\de$-dimension at most $r$. In particular, for $B=X$ one may take $M=\q_X$ whenever $\dim(X)\le r$.

\item\label{tredconj-6}
For $X$ as in assertion \ref{tredconj-4} conjecture $WD'^n_r(X)$ holds whenever $n\ge r-2$ also. 
\end{enumerate}
\end{theo}
\begin{proof}
\ref{tredconj-1}, \ref{tredconj-2}. Easy; note once again that $\obj \dmcx$ equals the union of $\tn_r(X)\lan j\ra$ as well as to the union of $\tnp_r(X)\lan j\ra$ for all $j\in \z$ and $r\ge 0$.

\ref{tredconj-3}. Immediate from the $\dm$-version of Theorem \ref{tgabber}(\ref{tgabber-4}) (see Theorem \ref{tnp}(II)).

\ref{tredconj-4}. According to the 
previous assertion, it suffices to verify that the non-vanishing of $E_2^{pq}T_{\wchow(X)}(\het,M)$ (for some $(p,q)$) implies the non-vanishing of  
$E_2^{pq}T_{\wchow(X)}(H,M)$ for $H$ as in Conjecture \ref{cwdhet}(\ref{cwdhet-1}).

Next, since any $M\in \obj \dmcx$ is defined over a subfield of $k$ that is finitely generated over its prime subfield, it suffices to verify the statement for $X$ of this form.
Now (similarly to \cite{huper}), one may use the seminal weight argument of Deligne and consider a  ``suitable model'' of $k$, i.e., one should take a suitable 
 variety over a finite field whose function field is $k$ and shrink it to ensure the pointwise purity of the corresponding weight factors; this enables us to assume that $k$ is a prime field itself.

Next, if $k$ is finite then \S5 of \cite{bbd} provides us with certain subcategories $\dhsl_{w\ge 0}\subset \obj D^cSh^{et}(X,\ql)$ that
 are respected by the (idempotent) $t$-truncation functors $[-i]\circ H_i^t$. 
Hence the vanishing $E_2^{pq}T_{\wchow(X)}(\het,M)$ for all $q\in \z$ and $p>n$ is equivalent to $\hetl(M)\in \dhsl_{w\ge 0}[-n]$.
Since  $\dhsl_{w\ge 0}$ is essentially   defined in terms of $i_x^!$ (and pointwise weights), this allows us to conclude the proof in the case. 

In the case when $k=\q$ 
 one should use the theory of weights developed in \S3 of \cite{huper} instead. 

\ref{tredconj-5}. The first part of the assertion is immediate from assertion \ref{tredconj-3} combined with Proposition \ref{predconj}(\ref{predconj-4}). One may apply it to $X$ of this sort according to Theorem \ref{tnp}(II).

\ref{tredconj-6}. The usage of assertion \ref{tredconj-3} in the previous argument just should be replaced by that of assertion \ref{tredconj-4}.

\end{proof}

\begin{rema}\label{rethomw}

\begin{enumerate}
\item\label{imaxsing}
Take $Y=Y_N$ for $N\ge r+1$ (see Proposition \ref{pexamples});   then we have
$\q_X\notin \dmcx_{\wchow\ge 1-r}$.  Hence our theorem yields a non-vacuous relation between ``\'etale weights'' and negative $K$-groups in this ``relative'' setting also.

\item\label{idegrel} Now we discuss the applicability of  our ``perverse'' criterion.
 
It was conjectured in  \S2.5 of \cite{bmm} that Chow-weight spectral sequences for perverse \'etale homology degenerate at $E_2$ (for any $M\in \obj \dmcx$) whenever they are defined. Thus in this case it suffices to study the factors of the corresponding Chow-weight filtrations.

\item\label{idweff} Moreover, if $X$ is a variety over a finite field, then Remark 3.6.2(1) of \cite{brelmot} 
expresses the corresponding  $E_2^{pq}T_{\wchow(X)}(\het,f_*(\q_Y))$
as the $q$th weight factor of $R^{p+q}f_{*, et}(\ql_Y)$ (cf. Proposition \ref{phomcoh}); here one should use the weights described in \S5 of \cite{bbd}. 
Thus $c(Y/X)=n$ whenever $f_{*,et}(\ql_Y)\in \dhsl_{w\ge 0}[-n]\setminus \dhsl_{w\ge 0}[1-n]$ (see the proof of the theorem); in particular, one may apply this statement for $Y=X$ (and so, $f_{*,et}(\ql_Y)=q_X$).
This seems to be (more or less) the only ``relative'' case where these $E_2$-terms may be computed (at least, one may study their vanishing).

\item\label{ihuwe} Note however that a certain theory of weights over number fields was developed in   \cite{huper}; one may apply it according to  
 the results of \cite[\S3.4]{brelmot}. More generally, one may apply the (more general) theory of  weights developed in  \S2.5 of \cite{bmm}
(along with the arguments used in the proof of Proposition \ref{phomcoh}; yet currently there does not exist a suitable theory of weights in the mixed characteristic case).

\item\label{icwd} 
One can easily verify that the cohomology theory  $H_x=\het(x)\circ i_x^!$ and the weights spectral sequence $T_{\wchow(X)}(H_x,M)$ can be described in terms of the ``restriction'' of $M$ to any open neighbourhood of $x$ in $X$. Hence for the purpose of computing $c(Y/X)$ it suffices to consider those point that belong to the image of the singular locus of $Y$. 
 
\item\label{illp}
One may apply the results proved above even in the case where there are no primes (``globally'') invertible on $X$. For this purpose one may invoke Proposition 
\ref{pwchownc}(\ref{iwexchowoc}); it enables us to replace $X$ by the components of its open  cover  such that on each of them invertible primes do exist (certainly, two components and two primes are sufficient here). 

\end{enumerate}


\end{rema}

\subsection{``Explicit test schemes'' over varieties and the relation to the singularity filtration (on $\kk$-theory and singular cohomology)}\label{sldh}

Now we relate our results to some (other) ``descent'' results and method. We start with a motivic of singularities results that it much simpler than the ones of \S\ref{slength} (and is related to the arguments used in \cite{bzp}).

\begin{pr}\label{pldh}
Let $\md$ be as in  Theorem \ref{tnp}(II) (so, it satisfies the local splitting condition) for $B$ being the spectrum of a perfect field $k$ of characteristic $p$.\footnote{Certainly, it follows that $\sss$ contains $p$ whenever it is positive.}
Then the following statements are valid for  $f\colon Y\to X$ being a separated morphism of (separated) finite type $B$-schemes, $d=\dim(Y)$.

\begin{enumerate}
\item\label{ilenv}
 $M=\mgd_X(Y)$ belongs to the envelope of $\mgd_X(Y_i)[-i]$ for some projective (separated)  regular $Y$-schemes $Y_i$ with $\dim(Y_i)\le d-i$.

\item\label{ildetect}
Assume that a full Karoubi-closed triangulated subcategory $\du$ of $\md(X)$ is endowed with a weight structure $w_X$ such that $\mgd_X(V)$ belongs to $\du_{w_X=0}$ for any regular $Y$ that is proper over $X$.\footnote{So, the first two parts of this propositions may be applied for $\md$ being equal either to  $\dk(-)$ or to $cdh$-motives with $R$-coefficients, where $R$ is any commutative unital ring in which $p$ is invertible whenever it is positive; see Remark \ref{rexplwchow}(\ref{rexplwchow-2}).}
Then $M=\mgd_X(Y)$ belongs to $\obj \du$ and there exists a choice of $t_{w_X}(M)=(M^i)$ with $M^i=0$ for $i<0$ and $i>d$, and $M_i$ is a retract of $\mgd_X(Y_i)[-i]$ for some projective (separated)  regular $X$-scheme $Y$ of dimension at most $d-i$ for $0\le i\le d$.

Thus we have $M\in \du_{w_X\ge -n}$ for some non-negative integer $n$ if and only if we have   $\mgd_X(Y_i)[-i]\perp M$ for any $i\in \z$ such that $n<i\le d$ (and for these $Y_i$).

\item\label{ilkmot}
Now assume that $\md(-)=\dk(-)$ and $f$ is proper. Then   $M\in \dkx_{\wchow\ge -n}$ if and only if  we have $\kk_{-i}^{BM,X}(Y\times_X Y_i)=\ns$
 for all $i>n$ and $Y_i$ as above.

\item\label{ilfield} If $X=B(=\spe k)$ and $k$ is perfect then the last 
 (vanishing) condition is converts into $\kk_{-i}(Y\times_X Y_i)$ being zero
 for all $i>n$.

\end{enumerate}
\end{pr}
\begin{proof}
\ref{ilenv}. Similarly to the proof of Proposition \ref{pgab} 
 (and according to Corollary 0.2 of \cite{bsnull}) we can assume that $\lam=\zol$ for some prime $l\neq p$.

Next, applying the obvious induction on dimension argument we can assume that the assertion (for $\lam=\zol$) is fulfilled whenever $\dim Y<d$.
Now, according to Theorem X.2.1 of \cite{illgabb}, there exists   a Cartesian square
\begin{equation}\label{eldh}
\begin{CD}
Z'@>{i'}>>Y'\\
@VV{g'}V@VV{g}V \\
Z@>{i}>>Y
\end{CD}\end{equation} 
such that $g$  is projective and generically finite of degree prime to $l$, $i$ is a  closed nowhere dense embedding, and $Y'$ is regular. The local splitting property
ensures that $Z$ can be increased (still remaining closed and nowhere dense in $Y$) so that $\mgd_X(Y\setminus Z)$ will be a retract of  $\mgd_X(Y'\setminus Z')$.
Hence applying the $\mgd$-version of (\ref{emgys}) we obtain that $\mgd_X(Y)$ belongs to the envelope of $\{\mgd(Z),\mgd(Y'),\mgd(Z')[-1]\}$.
Since $Z$ and  $Z'$ are of dimension less than $d$, it remains to apply our inductive assumption to these schemes.  

\ref{ildetect}. Consider an $X$-compactification $\overline{Y}$ of $Y$ (see Remark \ref{rdimf}(\ref{idcomp})); then $\dim (\overline{Y})=d$ and $\dim(\overline{Y}\setminus Y)<d$.
Hence applying the  $\mgd$-version of (\ref{emgys}) once again we reduce the assertion to the case where  $f$ is proper.
Hence the first part of the assertion follows from assertion \ref{ilenv} combined with Proposition \ref{pwc}(\ref{iwcenv}). To deduce the second part from the first one it 
suffices to invoke (the ``in particular'' part of) Proposition \ref{pbw}(\ref{iortprecise}).

\ref{ilkmot}. Similarly 
to the proof of Theorem \ref{testw}, we just use the isomorphism $\dmx(\mgbm_X(Y_i)[-i], M)\cong \kk_{-i}^{BM,X}(Y\times_X Y_i)$
given by  (\ref{ekthlesgen}). 

\ref{ilfield}. Recall that the $X$-Borel--Moore $\kk$-groups of smooth $X$-schemes are precisely the corresponding $\kk$-groups; see (\ref{ekthles}) once again. 
\end{proof}

\begin{rema}\label{rldh}
\begin{enumerate}
\item\label{idmc} In all ``reasonable'' settings the weight structure $w_X$ mentioned in part \ref{ildetect} of the proposition can be extended to the whole $\md(X)$; see Remark 3.3.3(3) of \cite{binters} for a result of this sort (though our current setting is much simpler than that of loc. cit.). 

\item\label{ildexpl} The advantage of the proposition is that its proof is quite simple, and one can more or less easily find the ``detector varieties'' $Y_i$ (explicitly; note that the argument described in Remark \ref{rweib}(\ref{iextest}) yields that is actually suffices to consider a finite number of primes when finding $Y_i$ in the case of a general $\lam$). 

 On the other hand, it seems that  to find ``test schemes''  for $M=f_*(\oo_Y)$ in the case of a non-proper $f$ one should use some more complicated arguments (similar to that described in \S\ref{slength}).

\item\label{ihiron} Certainly our argument becomes somewhat simpler in the case where Hironaka's resolution of singularities is available (so, for $p=0$; cf. Remark \ref{rvd}(3)). The authors suspect 
that in this case the corresponding connected components of $Y_i$ may be put into the ``vertices'' of the corresponding cubical diagram that yields a hyperresolution of $Y$; see \S2 of \cite{pasdesc}. 

Note moreover that in the current paper we use 
``lax'' methods that allow us not to consider any enhancements for our motivic categories. 
This makes our arguments very general;
 yet they do not yield much information on the boundaries in $t_{w_X}(M)$ (cf. part \ref{ildetect} of the proposition). So it would make sense to consider certain ``models'' for our categories (that should be descent categories in the sense described in \S2.1.3 of ibid.) and construct a {\it $\Phi$-rectified} functor (see \S2.1.5 of ibid.) from the category of cubical diagrams of finite type separated $X$-schemes into $\dkx$ (or into other motivic categories; for this purpose it actually suffices to construct a $\Phi$-rectified functor into a descent category being a model for $\sht(X)$). 
 This would enable us to compare our weight spectral sequence $T_{\wchow(X)}(\dkx(\oo_X,-),\oo_X)$ (along with the corresponding weight filtration) with the one given by Proposition 4.3 of ibid. (cf. also Corollary 5.9 of ibid.). 
Note here that the $E_2$-descent property (see Proposition 5.6 of ibid.) 
is very easy to check for any cohomology that factors through one of our motivic categories: if (\ref{eldh}) is an {\it elementary acyclic square} (i.e., $g$ is an isomorphism over  $Y\setminus Z$, and $Z$ and $Z'$ are also regular) 
then Proposition \ref{pbw}(\ref{isum}) yields that the corresponding  distinguished triangle $\ood_Y\to \mgd_Y(Z)\bigoplus \mgd_Y(Y')\to \mgd_Y(Z')$ splits; hence we also have a split  triangle $\mgd_X(Y)\to \mgd_X(Z)\bigoplus \mgd_X(Y')\to \mgd_X(Z')$ and thus a split exact sequence  $0\to H(\mgd_X(Y))\to H(\mgd_X(Z))\bigoplus H(\mgd_X(Y'))\to H(\mgd_X(Z'))\to 0$ 
for any additive 
 functor $H$ from $\md(X)$ (into an abelian category).

Similar arguments would probably allow to relate our results to \cite[\S5]{gs} and \cite{sg} (note however that in these papers simplicial diagrams instead of cubical ones are used).

Another interesting spectral sequence of this sort (converging to the singular cohomology of $X/\com$, where $X$ was allowed to be a complex analytic space) was considered in \S5 of \cite{ciguil}; it was called the {\it singularity spectral sequence} due to its triviality in the case of a regular $X$. Note that that this spectral sequence should be isomorphic to the Deligne's 
 weight one (starting from $E_2$) whenever $X$ is proper over $\com$ (see Corollary 5.3(2) of ibid. for the $\q$-linear version of this statement)
 due to the fact that $p_*$ is Chow-weight-exact whenever $p$ is a proper morphism. 
The authors would like to note that the usage of Hironaka's resolution of singularities is certainly a serious  disadvantage of ibid. (along with  \cite{pasdesc}) from the point of view of (general)  algebraic geometry (since 
cubical hyperresolutions of singularities are available over characteristic $0$ fields only). On the other hand, it would be very interesting to construct certain ``complex analytic'' weight structures related to the results of \cite{ciguil}.

\end{enumerate}




\end{rema}

\end{document}